\documentclass[12pt]{amsart}
\usepackage{pgf,tikz,pgfplots} 
\pgfplotsset{compat=1.14}
\usepackage{mathrsfs}
\usetikzlibrary{arrows}
\usetikzlibrary{patterns}
\usepackage{pgfplots}
\usetikzlibrary{intersections, pgfplots.fillbetween}
\usepackage[colorlinks=true,urlcolor=blue, citecolor=red,linkcolor=blue,linktocpage,pdfpagelabels, bookmarksnumbered,bookmarksopen]{hyperref}
\usepackage[hyperpageref]{backref}
\usepackage{cleveref}
\usepackage{xcolor}
\usepackage{amsthm} 
\usepackage{latexsym,amsmath,amssymb}
\usepackage{accents}
\usepackage[colorinlistoftodos,prependcaption,textsize=tiny]{todonotes}
\usepackage{a4wide}
\usepackage{soul}
\usepackage{mathtools} 
\usepackage{xparse} 

\pgfdeclarelayer{ft}
\pgfdeclarelayer{bg}
\pgfsetlayers{bg,main,ft}

\title[Minimizing degree one]{Minimizing and non-minimizing degree one $W^{s,1/s}$-harmonic maps between spheres}

\date{\today}
\author{Dorian Martino}
\address[Dorian Martino]{
	ETH Zürich,%
	Department of Mathematics,
	Rämistrasse 101,
	8092 Zürich, Switzerland}
\email{dorian.martino@math.ethz.ch}
\author{Katarzyna Mazowiecka}
 \address[Katarzyna Mazowiecka]{
	 Institute of Mathematics,%
	 University of Warsaw,
	 Banacha 2,
	 02-097 Warszawa, Poland}
 \email{k.mazowiecka@mimuw.edu.pl}
\author{Armin Schikorra}
 \address[Armin Schikorra]{Department of Mathematics,
	 University of Pittsburgh,
	 301 Thackeray Hall,
	 Pittsburgh, PA 15260, USA}
 \email{armin@pitt.edu}

\definecolor{indigo}{rgb}{0.29, 0.0, 0.51}
\definecolor{p1}{gray}{0.4}
\definecolor{p2}{gray}{0.6}
\definecolor{p3}{gray}{0.98}
\definecolor{p4}{gray}{0.8}
\definecolor{p5}{gray}{0.9}

plus 6pt minus 12pt
	\renewcommand{\i}{{\rm \bf i}}
	\def\eps{\varepsilon}

	\def\id{{\rm id }}

	\def\C{{\mathbb C}}

	\newcommand{\dif}{\,\mathrm{d}}
	
	\def\N{{\mathbb N}}

	\def\S{{\mathbb S}}

	\newtheorem{theorem}{Theorem}
	\newtheorem{lemma}[theorem]{Lemma}
	\newtheorem{corollary}[theorem]{Corollary}

	\newtheorem{remark}[theorem]{Remark}

	\def\lip{{\rm Lip\,}}
	
	\def\supp{{\rm supp\,}}
	
	\newcommand{\dd}{\,\mathrm{d}}
	\newcommand{\dx}{\dif x}
	\newcommand{\dy}{\dif y}
	\newcommand{\dz}{\dif z}

	\newcommand{\R}{\mathbb{R}}

	\newcommand{\Z}{\mathbb{Z}}

	\newcommand{\brac}[1]{\left (#1 \right )}
	\newcommand{\abs}[1]{\left |#1 \right |}

	\newcommand{\barint}{
		\rule[.036in]{.12in}{.009in}\kern-.16in \displaystyle\int }
	
	\newcommand{\barcal}{\mbox{$ \rule[.036in]{.11in}{.007in}\kern-.128in\int $}}
	
	\def\mvint_#1{\mathchoice
		{\mathop{\vrule width 6pt height 3 pt depth -2.5pt
				\kern -8pt \intop}\nolimits_{\kern -3pt #1}}%
	{\mathop{\vrule width 5pt height 3 pt depth -2.6pt
			\kern -6pt \intop}\nolimits_{#1}}%
	{\mathop{\vrule width 5pt height 3 pt depth -2.6pt
			\kern -6pt \intop}\nolimits_{#1}}%
	{\mathop{\vrule width 5pt height 3 pt depth -2.6pt
			\kern -6pt \intop}\nolimits_{#1}}}

\numberwithin{theorem}{section} \numberwithin{equation}{section}

\newcommand{\aleq}{\precsim}
\newcommand{\ageq}{\succsim}
\newcommand{\aeq}{\approx}

\let\latexchi\chi
\makeatletter
\renewcommand\chi{\@ifnextchar_\sub@chi\latexchi}
\newcommand{\sub@chi}[2]{
	\@ifnextchar^{\subsup@chi{#2}}{\latexchi^{}_{#2}}%
}
\newcommand{\subsup@chi}[3]{
	\latexchi_{#1}^{#3}%
}
\makeatother

\makeatletter
\def\tikz@arc@opt[#1]{
	{%
		\tikzset{every arc/.try,#1}%
		\pgfkeysgetvalue{/tikz/start angle}\tikz@s
		\pgfkeysgetvalue{/tikz/end angle}\tikz@e
		\pgfkeysgetvalue{/tikz/delta angle}\tikz@d
		\ifx\tikz@s\pgfutil@empty%
		\pgfmathsetmacro\tikz@s{\tikz@e-\tikz@d}
		\else
		\ifx\tikz@e\pgfutil@empty%
		\pgfmathsetmacro\tikz@e{\tikz@s+\tikz@d}
		\fi%
		\fi
		\tikz@arc@moveto
		\xdef\pgf@marshal{\noexpand%
			\tikz@do@arc{\tikz@s}{\tikz@e}
			{\pgfkeysvalueof{/tikz/x radius}}
			{\pgfkeysvalueof{/tikz/y radius}}}%
	}%
	\pgf@marshal%
	\tikz@arcfinal%
}
\let\tikz@arc@moveto\relax
\def\tikz@arc@movetolineto#1{%
	\def\tikz@arc@moveto{\tikz@@@parse@polar{\tikz@arc@@movetolineto#1}(\tikz@s:\pgfkeysvalueof{/tikz/x radius} and \pgfkeysvalueof{/tikz/y radius})}}
\def\tikz@arc@@movetolineto#1#2{#1{\pgfpointadd{#2}{\tikz@last@position@saved}}}
\tikzset{%
	move to start/.code=\tikz@arc@movetolineto\pgfpathmoveto,%
	line to start/.code=\tikz@arc@movetolineto\pgfpathlineto}
\makeatother

\begin{document}
	
	\begin{abstract}
	We show that $\id:\mathbb{S}^1 \to \mathbb{S}^1$ is {\it not} a minimizing $W^{s,\frac{1}{s}}$-harmonic map for $s  \in (0,\frac{1}{8}$). On the other hand, for $s \in (\frac{1}{3},1)$ it is a local minimizing map, and for $s\in [\frac{1}{2}-\eps,\frac{1}{2}+\eps]$ it is a global minimizer. The usual extension or Fourier techniques being unavailable, our argument relies instead of stability analysis in $s$.
	\end{abstract}
%
%
	\maketitle
	\tableofcontents
	\sloppy

	\section{Introduction}
	For $s \in (0,1)$ and $\psi: \S^1 \to \S^1 \subset \R^2$ set 
	\[   
	 E_s(\psi) \coloneqq \iint_{\S^1\times \S^1} \frac{|\psi(x)-\psi(y)|^{\frac{1}{s}}}{|x-y|^{2}} \dx \dy.
	\]
	
	Throughout, if $x,y\in \mathbb S^1\subset\mathbb C\simeq\mathbb R^2$, then $|x-y|$ denotes the Euclidean distance in $\mathbb C$, the measures $\dx,\dy$ denote arclength measure on $\mathbb S^1$. This ensures that for any $s \in (0,1)$ the energy is invariant under domain M\"obius transforms and target rotations, \Cref{la:mobius-invariance}.

We define $\mathcal{M}$ to be the set of M\"obius transforms $e^{\i \lambda} m_b\colon \S^1 \to \S^1$, $b \in \C$, $|b| \leq 1$ and $\lambda \in \R$
\[
 m_b(z) \coloneqq \frac{z+b}{1+\overline{b} z}.
\]
We are interested in minimizing $E_s$ among all maps of a given degree
	\[
	\#_s d \coloneqq  \inf_{u \in W^{s,\frac{1}{s}}(\S^1,\S^1): \text{ } \deg u = d} E_s(u).
	\]
	For $s= \frac{1}{2}$, i.e.\ the Hilbert-space case, it is known that 
	\[   
	 \#_{\frac{1}{2}} d = 4\pi^2 |d|
	\]
	and the minimum value is attained by the maps $z^d$, see for instance \cite{BMRS14,MP04}. For the case $s \neq \frac{1}{2}$ the picture gets murkier, since the energy is conformally invariant, even the existence of minimizers in a given homotopy class becomes unclear due to  bubbling effects. Sacks-Uhlenbeck theory has been obtained by the last two authors in \cite{Sucks23}, and in \cite{MS23} it is shown that the $s \mapsto \#_s d$ is continuous, which in particular implies that there are degree 1 maps $u\colon \S^1 \to \S^1$ such that $E_s(u) = \#_s 1$ for $s \approx \frac{1}{2}$.

In this work we are interested not in showing the existence of minimizers, but instead the identification of the minimizer and (as corollary) a basic stability analysis of degree one maps. This is challenging, since the usual methods to show that M\"obius transforms minimize degree one maps of an energy --- showing that minimizers are conformal maps either by an harmonic extension argument \cite{MS15,LS18}, Fourier analysis \cite{MP04,BMRS14,DSW23}, or realization of sharp geometric inequalities \cite{GLZ25} --- are unavailable for $W^{s,\frac{1}{s}}$ unless $s =\frac{1}{2}$. Instead our argument builds on what amounts to a precise Morse index analysis combined with the stability analysis \emph{in $s$} initiated in \cite{MS23}.

It is natural to hope that this holds for any $s \in (0,1)$, and that $u=\id\colon \S^1 \to \S^1$ is a minimizing degree $1$ map. Quite surprisingly, but actually not that difficult to prove, the latter is actually false for $s<\frac{1}{8}$, indeed we have our first result
	
\begin{theorem}\label{th:strictinequality}
		For any  $s \in (0,1/8)$ the identity map $\id\colon \S^1 \to \S^1$ is not a minimizer of $E_s$ among degree $1$ maps from $\S^1 \to \S^1$, i.e.\ we have $\#_s 1 < E_s(\id)$. Indeed, it is not even stable in the $W^{s,\frac{1}{s}}$-topology.
\end{theorem}

\Cref{th:strictinequality} follows from a precise computation of the second variation of $E_s$ at the identity, and the observation that it simply fails to be nonnegative\footnote{We note in passing that it seems that this is the case for all degrees $z^n: \S^1\to \S^1$, according to (AI generated, unverified, but if need be easily checkable trig integrals) computations, all of them loose positive definiteness at $s<\frac{1}{8}$, and the same seems to happen for slightly different thresholds in higher dimensions for $W^{s,\frac{d}{s}}(\S^d,\S^d)$} definite when $s<\frac{1}{8}$.

As we shall see in \Cref{th:stable-picture-multiplier}, when $s \in (\frac{1}{8},1)$, the second variation of the identity map is indeed positive definite, and when $s \in (\frac{1}{3},1)$ it controls a Sobolev norm. From this we conclude that in contrast to \Cref{th:strictinequality}, for $s \in (\frac{1}{3},1)$ the identity map $\id$ is at least a local minimizer:

\begin{theorem}[Identity is a local minimizer for $s \in (\frac{1}{3},1)$]\label{th:localminisid}
	Let $\frac{1}{3} < s_-  < s_+ < 1$. Then there exists $\eps= \eps(s_-,s_+)$ such that any degree one local minimizer $u \in W^{s,\frac{1}{s}}(\S^1,\S^1)$, $s \in [s_-,s_+]$ that satisfies $[u-\id]_{W^{s,\frac{1}{s}}} \leq \eps$ is actually a M\"obius transform, i.e. $u \in \mathcal{M}$.
\end{theorem}

Of course, it remains the question if $\id\colon \S^1 \to \S^1 $ is a \emph{global} degree minimizer, As we see from \Cref{th:strictinequality} it is not obvious at all why the identity should be the minimizer for all $s$. However, we obtain the following which is the main theorem of this work.

	\begin{theorem}\label{th:minisId}
		There exists $\delta > 0$ such that for $\S^1 \to \S^1$ maps and $s \in (1/2-\delta,1/2+\delta)$ then $\#_s 1 = E_s(\id)$. Moreover, any other degree one minimizer $u$ is a M\"obius transform.
	\end{theorem}

Interestingly, from the above we also obtain the \emph{existence} of degree $2$ minimizers for $s \in (1/2-\delta,1/2+\delta)$ --- providing progress towards \cite[Open Problem 24]{BM-book}.
	\begin{corollary}
	There exists $\delta > 0$ such that the infimum $\#_s 2$ is attained for all $s \in (\frac{1}{2}-\delta,\frac{1}{2}+\delta)$.
	\end{corollary}
	\begin{proof}
	By \cite{MS23}, $s \mapsto \#_s d$ is locally uniformly continuous, moreover around $s \approx \frac{1}{2}$ there exists a uniform constant so that $|d| \leq C \#_s d$. Thus we can choose $s$ sufficiently close to $\frac{1}{2}$ so that for $\#_s d > \#_s 2$ for all $s \approx \frac{1}{2}$ and all $|d| \geq 3$. If for some $s \approx \frac{1}{2}$ the degree $2$-minimizer was not attained, then by Sacks--Uhlenbeck theory, see \cite{Sucks23}, and the energy identity, see \cite{MS23}, we must have $\#_s 2 = 2\#_s 1$. However, by \Cref{th:minisId} and \Cref{la:energy-zk} $E_s(z^2) = 2E_s(z) = 2\#_s 1=\#_s 2$. Thus $z^2$ attains the minimum, contradiction. Let us stress that this does \emph{not} imply that $z^2$ is the minimizer!.
	\end{proof}

{\bf Outline}
We begin with preliminaries and basic computations in \Cref{s:preliminaries}, in particular we compute the first and second variation at the identity. In \Cref{s:idnotmin} we prove that $\id$ is not minimizing, which follows from the fact that the second variation has negative directions. The arguments are of elementary but tedious computational nature.

In \Cref{s:approxstability} we develop a crucial ingredient, which we call approximate stability, \Cref{th:stable}. It shows that when $s \approx \frac{1}{2}$ and $u$ is a minimizer of degree $1$ for $E_s$ then $u$ is close to a M\"obius transform in a very strong sense. It uses a combination of the $L^2$-stability result in \cite{DSW23} and $s$-stability arguments in \cite{MS23}.

In \Cref{s:positivity} we show that the second variation is nonnegative, with precise estimates. Again, the argument is mostly computational, elementary yet tedious.

When establishing stability, for $p<2$, we need a precise bound on the Lipschitz regularity for minimizing $E_s$-harmonic maps, and somewhat more than predicted in the classical regularity theory for the fractional $p$-Laplacian. This estimate is provided in \Cref{s:lipschitzreg}.

In \Cref{s:idmini} we then conclude the main result, that $\id$ is indeed a minimizer by a careful analysis of the second variation in all regimes of $p \in (1,3)$. In \Cref{s:localstable} we then discuss corollaries for the stability.

\subsection*{Acknowledgement}
Part of this work was carried out while K.M.\ and A.S.\ were visiting University of Bielefeld. Further mutual visits between the author's institutions are gratefully acknowledged. A.S.\ was an Alexander-von-Humboldt Fellow. A.S. is funded by NSF Career DMS-2044898. D.M. is funded by Swiss National Science Foundation, project SNF $200020\textunderscore 219429$.

The project is co-financed by the Polish National Agency for Academic Exchange within Polish Returns Programme - BPN/PPO/2021/1/00019/U/00001 (K.M.). The project is co-financed by National Science Centre grant 2022/01/1/ST1/00021 (K.M.).

Part of the work leading to this article is assisted by chatgpt. All mathematical validation, final proof decisions, and final wording remain the sole responsibility of the human authors.

\section{Preliminaries}\label{s:preliminaries}
We begin with some elementary computations and estimates, all well known or easy to obtain.

\begin{lemma}[Energy of the maps $z \mapsto z^k$]\label{la:energy-zk}
Let $0<s<1$, and let
\[
E_s(u)
\coloneqq 
\int_{\S^1}\int_{\S^1}
\frac{\abs{u(z)-u(w)}^{1/s}}{\abs{z-w}^2}
\,\dd z\,\dd w,
\]
where $\dd z,\dd w$ denote arclength measure on $\S^1$, and where $\abs{z-w}$
is the Euclidean distance in $\C$. For $k\in\Z$, set $u_k(z)\coloneqq z^k$.
Then it holds
\[
E_s(u_k)=C_s\abs{k},
\]
where
\[
C_s
=
2^{1/s}\pi^{3/2}
\frac{
\Gamma\!\left(\frac{1-s}{2s}\right)
}{
\Gamma\!\left(\frac{1}{2s}\right)
}.
\]
\end{lemma}

\begin{lemma}[Invariance under rotations and boundary Möbius maps]
\label{la:mobius-invariance}
Let $0<s<1$, set $p=\frac1s$, and let
\[
E_s(u)
\coloneqq 
\iint_{\S^1\times \S^1}
\frac{|u(x)-u(y)|^p}{|x-y|^2}\,\dx\,\dy .
\]
Let $b\in\mathbb C$ with $|b|<1$, and define
\[
m_b(z)\coloneqq \frac{z+b}{1+\overline b z}.
\]
Let $\lambda\in\mathbb R$, and for
$u\in W^{s,\frac1s}(\mathbb S^1,\R^2)$ define
\[
\widetilde u(z)\coloneqq e^{-\i\lambda}\, u(m_b(z)).
\]
Then, it holds
\begin{equation}\label{eq:compmoebiustransform}
E_s(\widetilde u)=E_s(u).
\end{equation}

Moreover, $u$ is a critical point of $E_s$ as a map into $\S^1 \to \S^1$ if and
only if $\widetilde u$ is a critical point of $E_s$ as a map into $\mathbb S^1$.
\end{lemma}
\begin{proof}
The proof of \eqref{eq:compmoebiustransform} is elementary, we refer, e.g., to \cite[Basic Example, p. 367]{BM-book}.
\end{proof}

The following is also well known.
\begin{lemma}\label{la:RtoS1}
Let
\[
\tau\colon \mathbb R\to \mathbb S^1\setminus\{(0,-1)\}
\]
be given by
\[
\tau(x)
=
\left(
\frac{2x}{1+x^2},
\frac{1-x^2}{1+x^2}
\right).
\]
Let $u\colon \mathbb S^1\to \mathbb S^1$ be a measurable map such that
\[
\iint_{\S^1\times \S^1}
\frac{|u(\xi)-u(\eta)|^2}{|\xi-\eta|^2}
\,\dd\sigma(\xi)\,\dd\sigma(\eta)
<\infty,
\]
where $\dd\sigma$ denotes arclength measure on $\mathbb S^1$, and where
$|\xi-\eta|$ denotes the Euclidean distance in $\mathbb R^2$.

We define $v\colon \mathbb R\to \mathbb S^1$ by $v(x)=u(\tau(x))$.
Then $v\in \dot H^{1/2}(\mathbb R;\mathbb R^2)$ and we have
\[
\int_{\mathbb R}\int_{\mathbb R}
\frac{|v(x)-v(y)|^2}{|x-y|^2}
\,\dx\,\dy
=
\iint_{\S^1\times \S^1}
\frac{|u(\xi)-u(\eta)|^2}{|\xi-\eta|^2}
\,\dd\sigma(\xi)\,\dd\sigma(\eta).
\]
\end{lemma}

\subsection{First and second variation of the identity map}

	For $\psi\colon \S^1 \to \S^1 \subset \R^2$ set 
	\[   
	 E_s(\psi) := \iint_{\S^1\times \S^1} \frac{|\psi(x)-\psi(y)|^{\frac{1}{s}}}{|x-y|^{2}}\, \dx\, \dy.
	\]
	We remind ourselves once more that $|\cdot|$ denotes the Euclidean distance in $\R^2$, so that this energy becomes conformally invariant.

\begin{lemma}[First and second variation at the identity]\label{la:firstsecondvariation}
Let $p=\frac1s>1$.
Let $\id\colon \mathbb S^1\to\mathbb S^1$ be the identity map. Let $\varphi\colon\mathbb S^1\to\mathbb R^2$ be smooth, and define
\[
F_\varepsilon(x)
\coloneqq 
\frac{x+\varepsilon\varphi(x)}
{|x+\varepsilon\varphi(x)|}.
\]
We denote $\eta(x)\coloneqq \langle \varphi(x),x^\perp\rangle$. Then we have
\[
\left.\frac{\dd}{\dd\varepsilon}\right|_{\varepsilon=0}
E_s(F_\varepsilon)\equiv\delta E_s(\id)[x^\perp \eta]=0,
\]
and
\[
\left.
\frac{\dd^2}{\dd\varepsilon^2}
\right|_{\varepsilon=0}
\frac1p E_s(F_\varepsilon)
\equiv \frac{1}{p} \delta^2 E_s(\id)[x^\perp \eta,x^\perp \eta] \eqqcolon
\mathcal{Q}_p(\eta),
\]
where for $\eta\colon \S^1 \to \R$
\begin{equation}\label{eq:Qpeta}
\begin{split}
\mathcal{Q}_p(\eta)
=
\iint_{\S^1\times \S^1}
\frac{
\left[
(p-2)|\langle x,y^\perp\rangle|^2 |x-y|^{p-4}
+
|x-y|^{p-2}\langle x,y\rangle
\right]
|\eta(x)-\eta(y)|^2
}{|x-y|^2}
\dx\dy .
\end{split}
\end{equation}
With a slight abuse of notation we denote the polarized version, for $\eta,\zeta\colon\mathbb S^1\to\mathbb R$, as
\[
\begin{split}
\mathcal{Q}_p(\eta,\zeta)
\coloneqq 
\iint_{\S^1\times \S^1}
\frac{
\left[
(p-2)|\langle x,y^\perp\rangle|^2 |x-y|^{p-4}
+
|x-y|^{p-2}\langle x,y\rangle
\right]
(\eta(x)-\eta(y))(\zeta(x)-\zeta(y))
}{|x-y|^2}
\dx \dy .
\end{split}
\]
\end{lemma}

\begin{proof}
Set
\[
\rho(x)\coloneqq \langle \varphi(x),x\rangle .
\]
We first compute the derivatives of the normalized perturbation. Since
\[
\varphi(x)=\rho(x)x+\eta(x)x^\perp,
\]
we have
\[
x+\varepsilon\varphi(x)
=
(1+\varepsilon\rho(x))x+\varepsilon\eta(x)x^\perp.
\]
Thus
\[
|x+\varepsilon\varphi(x)|
=
\sqrt{(1+\varepsilon\rho(x))^2+\varepsilon^2\eta(x)^2}.
\]
Expanding at $\varepsilon=0$,
\[
|x+\varepsilon\varphi(x)|
=
1+\varepsilon\rho(x)
+
\frac{\varepsilon^2}{2}\eta(x)^2
+
O(\varepsilon^3).
\]
Therefore
\[
\frac1{|x+\varepsilon\varphi(x)|}
=
1-\varepsilon\rho(x)
+
\varepsilon^2
\left(
\rho(x)^2-\frac12\eta(x)^2
\right)
+
O(\varepsilon^3).
\]
Multiplying,
\[
\begin{split}
F_\varepsilon(x)
&=
\left((1+\varepsilon\rho(x))x+\varepsilon\eta(x)x^\perp\right)
\left[
1-\varepsilon\rho(x)
+
\varepsilon^2
\left(
\rho(x)^2-\frac12\eta(x)^2
\right)
\right]
+
O(\varepsilon^3)
\\
&=
x
+
\varepsilon\eta(x)x^\perp
-
\varepsilon^2
\left(
\rho(x)\eta(x)x^\perp
+
\frac12\eta(x)^2x
\right)
+
O(\varepsilon^3).
\end{split}
\]
Hence
\[
\left.\frac{\dd}{\dd\eps} \right|_{\eps =0}F_\eps(x)=\eta(x)x^\perp,
\]
and
\begin{equation}\label{eq:secondder}
\left.\frac{\dd^2}{\dd\eps^2} \right|_{\eps =0}F_\eps(x)
=
-2\rho(x)\eta(x)x^\perp-\eta(x)^2x.
\end{equation}

We now compute the first variation.
\[
\left.
\frac{\dd}{\dd\varepsilon}
\right|_{\varepsilon=0}
|F_\varepsilon(x)-F_\varepsilon(y)|^p
=
p|x-y|^{p-2}
\left\langle
x-y,
\eta(x)x^\perp-\eta(y)y^\perp
\right\rangle.
\]
Since
\[
\langle x-y,x^\perp\rangle
=
\langle x,y^\perp\rangle,
\qquad
\langle x-y,y^\perp\rangle
=
\langle x,y^\perp\rangle,
\]
we get
\[
\left\langle
x-y,
\eta(x)x^\perp-\eta(y)y^\perp
\right\rangle
=
\langle x,y^\perp\rangle
\bigl(\eta(x)-\eta(y)\bigr).
\]
Therefore
\[
\left.
\frac{\dd}{\dd\varepsilon}
\right|_{\varepsilon=0}
E_s(F_\varepsilon)
=
p
\iint_{\S^1\times \S^1}
|x-y|^{p-4}
\langle x,y^\perp\rangle
\bigl(\eta(x)-\eta(y)\bigr)
\dx\dy.
\]

We show that this integral is zero. Write
\[
x=e^{\i\theta},\qquad y=e^{\i\sigma}.
\]
Then
\[
|x-y|=2\left|\sin\frac{\theta-\sigma}{2}\right|,
\qquad
\langle x,y^\perp\rangle=\sin(\theta-\sigma).
\]
Thus we compute
\[
\begin{split}
\frac{\dd}{\dd\varepsilon}
\Big |_{\varepsilon=0} E_s(F_\varepsilon)
&=p\, 2^{p-4}\int_0^{2\pi}\int_0^{2\pi}
\left|\sin\brac{\frac{\theta-\sigma}{2}}\right|^{p-4}
\sin (\theta-\sigma)
\bigl(\eta(e^{\i\theta})-\eta(e^{\i\sigma})\bigr)
\,\dd\theta\,\dd\sigma\\
&=p\, 2^{p-4}\int_0^{2\pi}\int_{0}^{2\pi}
\left|\sin\brac{\frac{\tau}{2}}\right|^{p-4}
\sin (\tau)
\bigl(\eta(e^{\i(\tau+\sigma)})-\eta(e^{\i\sigma})\bigr)
\,\dd\tau\,\dd\sigma.\\
\end{split}
\]
Observe this integral converges for $p \in (1,\infty)$ since $\eta$ is Lipschitz, and then since the singularity at $\tau =0$ is of order $p-2$ which for $p>1$ is integrable.

Observe that
\[
\int_0^{2\pi}\eta(e^{\i(\sigma+\tau)})\,\dd\sigma
=
\int_0^{2\pi}\eta(e^{\i\sigma})\,\dd\sigma.
\]
Thus, we obtain
\[
\left.
\frac{\dd}{\dd\varepsilon}
\right|_{\varepsilon=0}
E_s(F_\varepsilon)=0.
\]
That is to say, the identity map is indeed a critical point of $E_s$ for all $s \in (0,1)$.

We now compute the second variation. We use the following formula for $f,g\colon \R \to \mathbb R^2$,
\[
\begin{split}
\left.\frac{\dd^2}{\dd\eps^2} \right|_{\eps =0} |f(\varepsilon)-g(\varepsilon)|^p
&=
p(p-2)|f-g|^{p-4}
\left|
\left\langle f-g,f'-g'\right\rangle
\right|^2
\\
&\quad
+
p|f-g|^{p-2}|f'-g'|^2
\\
&\quad
+
p|f-g|^{p-2}
\left\langle f-g,f''-g''\right\rangle .
\end{split}
\]
Apply this with
\[
f(\varepsilon)=F_\varepsilon(x),
\qquad
g(\varepsilon)=F_\varepsilon(y).
\]
Recall
\[
\left. \frac{\dd}{\dd\eps}\right|_{\eps =0}F_\eps(x)-\left.\frac{\dd}{\dd\eps}\right|_{\eps =0}F_\eps(y)
=
\eta(x)x^\perp-\eta(y)y^\perp,
\]
and \eqref{eq:secondder}
\[
\left.\frac{\dd^2}{\dd\eps^2}\right|_{\eps =0}F_\eps(x)-\left.\frac{\dd^2}{\dd\eps^2}\right|_{\eps =0}F_\eps(y)
=
-2\rho(x) \eta(x)x^\perp+2\rho(y) \eta(y)y^\perp-\eta(x)^2x+\eta(y)^2y.
\]
First, we have
\[
\begin{split}
\left\langle x-y, \eta(x)x^\perp-\eta(y)y^\perp\right\rangle
&=
\eta(x)\langle x-y,x^\perp\rangle
-
b\langle x-y,y^\perp\rangle
\\
&=
\langle x,y^\perp\rangle(\eta(x)-b).
\end{split}
\]
Second, it holds
\begin{align*}
|\eta(x)x^\perp-\eta(y)y^\perp|^2
&=
\eta(x)^2+\eta(y)^2-2\langle x^\perp,y^\perp\rangle \eta(x)\eta(y) \\
&=
\eta(x)^2+\eta(y)^2-2\langle x,y\rangle \eta(x)\eta(y).
\end{align*}
Third, we obtain
\[
\begin{split}
\left\langle x-y,-\eta(x)^2x+\eta(y)^2y\right\rangle
&=
-\eta(x)^2\langle x-y,x\rangle
+
\eta(y)^2\langle x-y,y\rangle
\\
&=
-\eta(x)^2(1-\langle x,y\rangle)
+
\eta(y)^2(\langle x,y\rangle-1)
\\
&=
-(1-\langle x,y\rangle)(\eta(x)^2+\eta(y)^2).
\end{split}
\]
Therefore, we deduce
\[
\begin{split}
&|\eta(x)x^\perp-\eta(y)y^\perp|^2
+
\left\langle x-y,-\eta(x)^2x+\eta(y)^2y\right\rangle
\\
&\qquad
=
\eta(x)^2+\eta(y)^2-2\langle x,y\rangle \eta(x)\eta(y)
-
(1-\langle x,y\rangle)(\eta(x)^2+\eta(y)^2)
\\
&\qquad
=
\langle x,y\rangle(\eta(x)^2+\eta(y)^2-2\eta(x)\eta(y))
\\
&\qquad
=
\langle x,y\rangle(a-b)^2.
\end{split}
\]
The remaining part of $F''_0(x)-F''_0(y)$ contains the radial coefficient $\rho$:
\[
-2\rho(x) ax^\perp+2\rho(y) \eta(y)y^\perp.
\]
Its contribution is
\[
\begin{split}
\left\langle x-y,-2\rho(x) ax^\perp+2\rho(y) \eta(y)y^\perp\right\rangle
&=
-2\rho(x) a\langle x-y,x^\perp\rangle
+
2\rho(y) b\langle x-y,y^\perp\rangle
\\
&=
2\langle x,y^\perp\rangle(\rho(y) b-\rho(x) a).
\end{split}
\]
Hence, for $x\neq y$,
\[
\begin{split}
\left.
\frac{\dd^2}{\dd\varepsilon^2}
\right|_{\varepsilon=0}
|F_\varepsilon(x)-F_\varepsilon(y)|^p
&=
p(p-2)|x-y|^{p-4}
|\langle x,y^\perp\rangle|^2
|\eta(x)-\eta(y)|^2
\\
&\quad
+
p|x-y|^{p-2}
\langle x,y\rangle
|\eta(x)-\eta(y)|^2
\\
&\quad
+
2p|x-y|^{p-2}
\langle x,y^\perp\rangle
\bigl(\rho(y)\eta(y)-\rho(x)\eta(x)\bigr).
\end{split}
\]
Dividing by $|x-y|^2$, we obtain
\[
\begin{split}
\left.
\frac{d^2}{d\varepsilon^2}
\right|_{\varepsilon=0}
\frac{|F_\varepsilon(x)-F_\varepsilon(y)|^p}{|x-y|^2}
&=
p
\frac{
\left[
(p-2)|\langle x,y^\perp\rangle|^2|x-y|^{p-4}
+
|x-y|^{p-2}\langle x,y\rangle
\right]
|\eta(x)-\eta(y)|^2
}{|x-y|^2}
\\
&\quad
+
2p|x-y|^{p-4}
\langle x,y^\perp\rangle
\bigl(\rho(y)\eta(y)-\rho(x)\eta(x)\bigr).
\end{split}
\]

It remains to show that the last term integrates to zero. Let
\[
h(x)\coloneqq \rho(x)\eta(x).
\]
In angular variables,
\[
x=e^{\i\theta},
\qquad
y=e^{\i\sigma},
\]
the last term is
\[
2p
\left(2\left|\sin\frac{\theta-\sigma}{2}\right|\right)^{p-4}
\sin(\theta-\sigma)
\bigl(h(e^{\i\sigma})-h(e^{\i\theta})\bigr).
\]
As above, define
\[
A(\tau)\coloneqq 
\left(2\left|\sin\frac{\tau}{2}\right|\right)^{p-4}
\sin\tau.
\]
The radial contribution equals
\[
2p
\int_0^{2\pi}\int_0^{2\pi}
A(\theta-\sigma)
\bigl(h(e^{i\sigma})-h(e^{i\theta})\bigr)
\,\dd\theta\,\dd\sigma.
\]
Substituting $\tau = \theta-\sigma$ we see again that this integral vanishes.

This proves the lemma.
\end{proof}

\begin{lemma}\label{la:Qprep}
With the notations of \Cref{la:firstsecondvariation}, we have for any $\eta\colon \S^1 \mapsto \R$
\[
\mathcal{Q}_p(\eta) = \int_0^{2\pi} \int_{0}^{2\pi}
K_p(t)
\abs{\eta(e^{\i (\sigma+t)})-\eta(e^{\i \sigma})}^2
\,\dd t \,\dd\sigma.
\]
with 
\begin{equation}\label{eq:Kp}
K_p(t) = 2^{p-4} \left|\sin\brac{t/2}\right|^{p-4} \left[
p  \cos^2(t/2) -  1
\right].
\end{equation}
\end{lemma}
\begin{proof}
We first rewrite the integral \eqref{eq:Qpeta} in angular variables. Set
\begin{equation}
x=e^{\i\theta},
\qquad
y=e^{\i\sigma},
\qquad
t=\theta-\sigma.
\end{equation}
Then\footnote{$\langle x,y\rangle = \Re (x\bar{y}) = \Re (e^{\i (\theta - \sigma)}) = \cos(\theta-\sigma)$. $y^\perp = \i y$, so $|\langle x, y^\perp\rangle | = |\Re (x \i \bar{y})|=|\Re (\i e^{\i t})|=|\sin (t)|$. And $|x-y|^2 = |1-e^{\i t}|^2 = 2 - 2 \cos(t)=4 \sin^2(t/2)$}
\begin{equation}\label{eq:scalar_products}
\langle x,y\rangle_{\R^2}=\cos t,
\qquad
|\langle x,y^\perp\rangle|^2=\sin^2 t,
\qquad
|x-y|=2\left|\sin\brac{\frac t2}\right|.
\end{equation}
Then we have  
\[
\begin{split}
 &\left[
(p-2)|\langle x,y^\perp\rangle|^2 |x-y|^{p-4}
+
|x-y|^{p-2}\langle x,y\rangle
\right]\\
=&|x-y|^{p-4} \left[
(p-2)|\langle x,y^\perp\rangle|^2 
+
|x-y|^{2}\langle x,y\rangle
\right]\\
=&2^{p-4} \left|\sin\brac{\frac t2}\right|^{p-4} \left[
(p-2)\sin^2 t
+
4 \abs{\sin \frac{t}{2}}^2\cos(t)
\right].
\end{split}
\]
Observe that
\[
 \sin(t) =2 \sin(t/2) \cos(t/2), \quad \cos(t) = \cos^2(t/2) - \sin^2(t/2).
\]
Hence, we have
\[
\begin{split}
 &\left[
(p-2)|\langle x,y^\perp\rangle|^2 |x-y|^{p-4}
+
|x-y|^{p-2}\langle x,y\rangle
\right]\\
& =2^{p-4} \left|\sin\brac{\frac t2}\right|^{p-4} \left[
(p-2) 4\sin^2(t/2) \cos^2(t/2)
+
4 \abs{\sin \frac{t}{2}}^2 \cos^2(t/2) - 4 \abs{\sin \frac{t}{2}}^2\sin^2(t/2)
\right]\\
& =2^{p-2} \left|\sin\brac{\frac t2}\right|^{p-2} \left[
(p-2)  \cos^2(t/2)
+
  \cos^2(t/2) -  \sin^2(t/2)
\right]\\
& = 2^{p-2} \left|\sin\brac{\frac t2}\right|^{p-2} \left[
(p-1)  \cos^2(t/2) -  \sin^2(t/2)
\right]\\
& = 2^{p-2} \left|\sin\brac{\frac t2}\right|^{p-2} \left[
p  \cos^2(t/2) -  1
\right].
\end{split}
\]
We set 
\[
K_p(t) \coloneqq \frac{
(p-2)|\langle x,y^\perp\rangle|^2 |x-y|^{p-4}
+
|x-y|^{p-2}\langle x,y\rangle
}{|x-y|^2}
=
2^{p-4} \left|\sin\brac{\frac t2}\right|^{p-4} \left[
p  \cos^2(t/2) -  1
\right].
\]
Observe that $K_p(t+k2\pi) = K_p(t)$ for all $k \in \Z$.
Then, we have
\[
\begin{split}
\mathcal{Q}_p(\eta)
=&\int_0^{2\pi}\int_0^{2\pi}
K_p(\theta-\sigma)
|\eta(e^{\i \theta})-\eta(e^{\i \sigma})|^2
\,\dd\theta\,\dd\sigma\\
=&\int_0^{2\pi}\int_{-\sigma}^{2\pi-\sigma}
K_p(t)
|\eta(e^{\i \sigma+t})-\eta(e^{\i \sigma})|^2
\,\dd t \,\dd\sigma\\
=&\int_0^{2\pi} \brac{\int_{0}^{2\pi-\sigma}
K_p(t)
|\eta(e^{\i \sigma+t})-\eta(e^{\i \sigma})|^2
\,\dd t + \int_{2\pi-\sigma}^{2\pi}
K_p(t-2\pi)
|\eta(e^{\i \sigma+t-2\pi})-\eta(e^{\i \sigma})|^2
\,\dd t}\,\dd\sigma.
\end{split}
\]
By $2\pi$-periodicity, we obtain 
\[
\begin{split}
\mathcal{Q}_p(\eta)
=&\int_0^{2\pi} \brac{\int_{0}^{2\pi}
K_p(t)
|\eta(e^{\i (\sigma+t)})-\eta(e^{\i \sigma})|^2
\dd t }\dd\sigma.
\end{split}
\]
\end{proof}

\subsection{Estimates about lifting}

\begin{lemma}
\label{la:lipschitz-phase}
There exists $\varepsilon_0>0$ such that the following holds. Let $u\colon \mathbb S^1\to\mathbb S^1$ be such that
\[
        \|u-\id\|_{L^\infty(\mathbb S^1)}\le\varepsilon_0.
\]
Let $a\colon \mathbb S^1\to[-\pi/4,\pi/4]$ be the unique phase such that
\[
        u(x)=\cos(a(x))x+\sin(a(x))x^\perp .
\]
Then we have
 \begin{equation}\label{eq:axmayest}
\begin{split}
|a(x)-a(y)| 
\aleq&\abs{\cos(a(x))-\cos(a(y))}+
\abs{\sin(a(x))-\sin(a(y))}\\
\aleq&|u_1(x)-u_1(y)-(x-y)| + |a(y)||x-y|
\end{split}
\end{equation}
In particular, for any $t \in (0,1)$, $q \in (1,\infty)$
\[
 [a]_{W^{t,q}(\S^1)} \aleq [u_1-\id]_{W^{t,q}(\S^1)} + \|a\|_{L^\infty} \aleq [u_1-\id]_{W^{t,q}(\S^1)} + \|u_1-\id\|_{L^\infty}
\]
and for $\gamma \in (0,1]$,
\[
 [a]_{C^\gamma} \aleq [u_1-\id]_{C^\gamma} + \|u_1-\id\|_{L^\infty}.
\]
\end{lemma}
\begin{proof}
We have 
\[
 |u_1(x)-x|^2 = |\cos(a(x))-1|^2 + |\sin(a(x))|^2 =2-2\cos(a(x)) = 4 \sin^2(a(x)/2).
\]
Since $|\sin(t)| \geq \frac{2}{\pi} |t|$ for $|t| \leq \frac{\pi}{2}$, thus we find in particular
\[
 |a(x)| \aleq |u_1(x)-x|.
\]
Moreover, we have
\[
\begin{split}
u_1(x)-u_1(y)-(x-y) 
&= (\cos(a(x))-1)x-(\cos(a(y))-1)y
   +\sin(a(x))x^\perp-\sin(a(y))y^\perp\\
   &=
(\cos(a(x))-\cos(a(y)))x
+(\sin(a(x))-\sin(a(y)))x^\perp \\
&\quad
+(\cos(a(y))-1)(x-y)
+\sin(a(y))(x^\perp-y^\perp).
\end{split}
\]
In particular, it holds
\[
\begin{split}
 & \abs{\cos(a(x))-\cos(a(y))}+
\abs{\sin(a(x))-\sin(a(y))}\\
&\leq
|u_1(x)-u_1(y)-(x-y)|  
+
|\cos(a(y))-1|\,|x-y|
+
|\sin(a(y))|\,|x^\perp-y^\perp|\\
& \aleq |u_1(x)-u_1(y)-(x-y)| + |a(y)||x-y|.
\end{split}
\]
The left-hand side verifies
\[
\begin{split}
&|\cos(a(x))-\cos(a(y))|^2+|\sin(a(x))-\sin(a(y))|^2\\
& =2-2\brac{\cos(a(x))\cos(a(y))+\sin(a(x))\sin(a(y))}\\
& = 2-2\cos(a(x)-a(y))\\
& = 4 \sin^2\brac{\frac{a(x)-a(y)}{2}}\\
& \ageq \abs{a(x)-a(y)}^2,
\end{split}
\]
again using that $|a| \leq \frac{\pi}{4}$.

\end{proof}

\begin{lemma}[Lifting]
\label{lem:uniform-small-arc-lifting}
Let $\alpha>0$ and $\lambda>0$.
There exist constants $r_{\mathrm{lift}}>0$ and $C_{{lift}}>0$ depending only on $\alpha$ and $\lambda$ such that the following holds.

Let $u\in C^\alpha(\S^1,\S^1)$ such that 
\[
        [u]_{C^\alpha(\S^1)}\le \lambda.
\]
If $I\subset\S^1$ is an arc of length
\[
        0<|I|\le r_{\mathrm{lift}}.
\]
Then there exists a unique lift $\varphi\colon  I \to [-\frac{\pi}{4},\frac{\pi}{4}]$ such that $u=e^{i\varphi}$ on $I$ and
\[
        [\varphi]_{C^\alpha(I)}
        \le
        C_{{lift}}(1+\lambda).
\]
\end{lemma}

\begin{proof}
Choose $r_{\mathrm{lift}}>0$ so small that
\[
        \lambda r_{\mathrm{lift}}^\alpha\le \frac1{100}.
\]
If $I$ is an arc of length at most $r_{\mathrm{lift}}$, then
\[
        \operatorname{diam} u(I)\le [u]_{C^\alpha(\S^1)} |I|^\alpha
        \le \lambda r_{\mathrm{lift}}^\alpha
        \le \frac1{100}.
\]
Thus $u(I)$ is contained in an open semicircle of the target. Hence there is a single-valued phase $\varphi$ on $I$ such that
\[
        u=e^{i\varphi}
        \qquad\text{on }I.
\]
On a target arc of length less than $\frac1{100}$, chordal distance and phase distance are comparable. Therefore, it holds
\[
        |\varphi(x)-\varphi(y)|
        \le
        C|u(x)-u(y)|
        \le
        C\lambda |x-y|^\alpha.
\]
This gives
\[
        [\varphi]_{C^\alpha(I)}
        \le
        C_{{lift}}(1+\lambda).
\]
\end{proof}

\subsection{Vector Calculus}
\begin{lemma}
Let $U,V\in\mathbb R^2$ and take any real numbers $A,B,C,D \in \R$, then
\begin{equation}\label{eq:identity1}
\begin{split}
& (p-2)|U-V|^{p-4}
\left\langle U-V,AU^\perp-BV^\perp\right\rangle
\left\langle U-V,CU^\perp-DV^\perp\right\rangle
\\
&\quad
+
|U-V|^{p-2}
\left\langle AU^\perp-BV^\perp,CU^\perp-DV^\perp\right\rangle
\\
&\quad
+
|U-V|^{p-2}
\left\langle U-V,-ACU+BDV\right\rangle
\\
&=
\left[
(p-2)|\langle U,V^\perp\rangle|^2|U-V|^{p-4}
+
|U-V|^{p-2}\langle U,V\rangle
\right]
(A-B)(C-D).
\end{split}
\end{equation}
If moreover $|U| = |V| =1$, then it holds
\begin{equation}\label{eq:identity2}
\begin{split}
& (p-2)|\langle U,V^\perp\rangle|^2 |U-V|^{p-4}
+
|U-V|^{p-2}\langle U,V\rangle
\\
&=|U-V|^{p-2}
\left[
p-1-\frac p4|U-V|^2
\right].
\end{split}
\end{equation}
\end{lemma}
\begin{proof}
As for \eqref{eq:identity1}, we have
\[
\begin{split}
& (p-2)|U-V|^{p-4}
\left\langle U-V,AU^\perp-BV^\perp\right\rangle
\left\langle U-V,CU^\perp-DV^\perp\right\rangle
\\
&\quad
+
|U-V|^{p-2}
\left\langle AU^\perp-BV^\perp,CU^\perp-DV^\perp\right\rangle
\\
&\quad
+
|U-V|^{p-2}
\left\langle U-V,-ACU+BDV\right\rangle
\\
=& (p-2)|U-V|^{p-4}
(A-B)(C-D)\brac{\left\langle U,V^\perp\right\rangle}^2\\
&\quad
+
|U-V|^{p-2}
\brac{
AC |U|^2
-AD\left\langle U,V\right\rangle
-BC\left\langle U,V \right\rangle
+BD|V|^2
}
\\
&\quad
+
|U-V|^{p-2}
\brac{
-AC|U|^2
+BD\left\langle U,V\right\rangle
+AC\left\langle U,V\right\rangle
-BD|V|^2
}
\\
=& (p-2)|U-V|^{p-4}
(A-B)(C-D)\brac{\left\langle U,V^\perp\right\rangle}^2\\
&\quad
+
|U-V|^{p-2} (A-B)(C-D)
 \left\langle U,V\right\rangle.
\end{split}
\]
As for \eqref{eq:identity2}, now we can assume that $|U| = |V| =1$. Using $\sin^2 + \cos^2 =1$, we have
\[
 |\langle U,V\rangle|^2 +|\langle U,V^\perp\rangle|^2 =|U|^2|V|^2.
\]
Indeed, it holds
\[
\begin{split}
 &(U_1 V_1 + U_2 V_2)^2 + (-U_1 V_2 + U_2 V_1)^2 \\
 =&U_1^2 V_1^2 + U_2^2 V_2^2 + 2 U_1 U_2 V_1  V_2 + U_1^2 V_2^2 + U_2^2 V_1^2 - 2 U_1 U_2  V_2 V_1\\
 =&U_1^2 V_1^2 + U_2^2 V_2^2 + U_1^2 V_2^2 + U_2^2 V_1^2 \\
 =&(U_1^2 + U_2^2) (V_1^2 + V_2^2).
 \end{split}
\]
Coming back to \eqref{eq:identity2}, we obtain
\[
\begin{split}
& (p-2)|\langle U,V^\perp\rangle|^2 |U-V|^{p-4}
+
|U-V|^{p-2}\langle U,V\rangle
\\
=& (p-2) \brac{|U|^2|V|^2 -  |\langle U,V\rangle|^2}|U-V|^{p-4}
+
|U-V|^{p-2}\langle U,V\rangle
\\
=& |U-V|^{p-4} \brac{(p-2) \brac{|U|^2|V|^2 -  |\langle U,V\rangle|^2}
+\brac{|U|^2 + |V|^2 - 2 \langle U,V\rangle}
\langle U,V\rangle}
\\
=& |U-V|^{p-4} \brac{(p-2) |U|^2|V|^2 -  (p-2) |\langle U,V\rangle|^2
+|U|^2 \langle U,V\rangle + |V|^2 \langle U,V\rangle- 2 |\langle U,V\rangle|^2
}
\\
=& |U-V|^{p-4} \brac{(p-2) |U|^2|V|^2 -  p |\langle U,V\rangle|^2
+|U|^2 \langle U,V\rangle + |V|^2 \langle U,V\rangle
}
\\
\overset{|U|=|V|=1}{=}& |U-V|^{p-4} \brac{(p-2)  -  p |\langle U,V\rangle|^2
+ 2\langle U,V\rangle }
\\
=& |U-V|^{p-4} \Big( (p-2)  -  \frac{p}{4} \brac{|U|^2 + |V|^2 -2\langle U,V\rangle - |U|^2-|V|^2}^2 \\
& \qquad 
-|U|^2-|V|^2 +2\langle U,V\rangle  + |U|^2 + |V|^2 \Big)\\
=& |U-V|^{p-4} \brac{(p-2)  -  \frac{p}{4} \brac{|U-V|^2 - 2}^2
-|U-V|^2  + 2}\\
=& |U-V|^{p-4} \brac{p  -  \frac{p}{4} \brac{|U-V|^4 + 4 - 4|U-V|^2}
-|U-V|^2  }\\
=& |U-V|^{p-4} \brac{  -  \frac{p}{4} |U-V|^4 +p|U-V|^2-|U-V|^2  }.
\end{split}
\]

\end{proof}

\section{Identity is not minimizing for small \texorpdfstring{$s$}{s}: Proof of Theorem~\ref{th:strictinequality}}\label{s:idnotmin}

In this section we show that for $s<\frac{1}{8}$ there exist a negative direction for the identity map, which implies \Cref{th:strictinequality}. We recall the second variation formula from \eqref{la:firstsecondvariation}, and $\mathcal{Q}_p(\eta)$ from \eqref{eq:Qpeta}.

We will compute $\mathcal{Q}_p$ on the single Fourier mode
\[
 \eta_3(z) \coloneqq  \frac{1}{2} \brac{z^3+z^{-3}},
\]
i.e.
\[
\eta_3(e^{\i \theta})=\cos(3\theta).
\]

\begin{lemma}[The mode $\cos(3\theta)$]\label{la:cos3theta}
For every $p>1$,
\begin{equation}
\mathcal{Q}_p(\eta_3)
=(8-p)\, 2^{p+1}\, \pi^{\frac{3}{2}} \frac{\Gamma\!\left(\frac{1+p}{2}\right)}
{\Gamma\!\left(3+\frac{p}{2}\right)}
\end{equation}
In particular, it holds
\[
\begin{cases} 
 \mathcal{Q}_p(\eta_3) > 0 & \text{if $p \in (1,8)$}, \\
 \mathcal{Q}_p(\eta_3) < 0 & \text{if $p > 8$}.
 \end{cases} 
\]
\end{lemma}

\begin{proof}
From \Cref{la:Qprep}, we have
\begin{equation}
\mathcal{Q}_p(\eta_3)
=
\int_0^{2\pi}
K_p(t)
\left[
\int_0^{2\pi}
\left(\cos(3(\sigma+t))-\cos(3\sigma)\right)^2 \dd\sigma
\right] \dd t.
\end{equation}
The inner integral is
\begin{equation*}
\begin{split}
& \int_0^{2\pi}
\left(\cos(3(\sigma+t))-\cos(3\sigma)\right)^2
\dd\sigma \\
& = \frac{1}{4}\int_0^{2\pi}
\left(e^{3\i (\sigma+t)} + e^{-3\i (\sigma+t)} - e^{3\i \sigma} - e^{-3\i \sigma}\right)^2
\dd\sigma\\
& = \frac{1}{4}\int_0^{2\pi} \left( e^{6\i (\sigma+t)} + 2 + e^{-6\i (\sigma+t)} - 2e^{3\i (2\sigma + t)} - 2(e^{3\i t}+e^{-3\i t}) + e^{6\i \sigma} + e^{-6\i \sigma} +2 -2e^{-3\i (2\sigma+t)}\right) \dd \sigma \\
& = 2\pi(1-\cos(3t))\\
&=
4\pi\sin^2\brac{\frac{3t}{2}}.
\end{split}
\end{equation*}
Hence, we have by \eqref{eq:Kp}
\[
\begin{split}
\mathcal{Q}_p(\eta_3)
&=
4\pi
\int_0^{2\pi}
K_p(t) \sin^2\brac{\frac{3t}{2}}\,\dd t\\
& = 4\pi
\int_0^{2\pi}
 2^{p-4} \left|\sin\brac{\frac t2}\right|^{p-4} \left[
p  \cos^2(t/2) -  1
\right] \sin^2\brac{\frac{3t}{2}}\,\dd t\\
& = 2^{p-1} \pi
\int_0^{\pi}
  \sin^{p-4}\brac{r} (
p  \cos^2(r) -  1)
 \sin^2(3r)\,\dd r.
\end{split}
\]
We have 
\[
\sin(3r)
=
\sin (r)(4\cos^2( r)-1).
\]
Hence, it holds
\[
\begin{split}
\mathcal{Q}_p(\eta_3)
& = 2^{p-1} \pi
\int_0^{\pi}
  \sin^{p-2}\brac{r} (
p  \cos^2(r) -  1)
 (4\cos^2 r-1)^2\,\dd r\\
 &= 2^{p-1} \pi
\int_0^{\pi}
  \sin^{p-2}\brac{r} (
p  \cos^2(r) -  1)
 (16\cos^4 r+1-8 \cos^2 r)\,\dd r\\
 & = 2^{p-1} \pi
\int_0^{\pi}
  \sin^{p-2}\brac{r} (
p  \cos^2(r) -  1)
 (16\cos^4 r+1-8 \cos^2 r)\,\dd r.
\end{split}
\]
Mathematica confirms
\[
\mathcal{Q}_p(\eta_3)
=(8-p)\, 2^{p+1}\, \pi^{\frac{3}{2}}\, \frac{\Gamma\!\left(\frac{1+p}{2}\right)}
{\Gamma\!\left(3+\frac{p}{2}\right)}.
\]
\end{proof}

We now complete the proof of \Cref{th:strictinequality}.

\begin{proof}[Proof of Theorem~\ref{th:strictinequality}]
Let 
\[
 \eta(z) \coloneqq  \frac{1}{2} \brac{z^3+z^{-3}}
\]
Observe that $\eta\colon \S^1 \to \R$ is a smooth map. Set $\varphi\colon \S^1 \to \R^2$ be given by
\[
 \varphi(z) \coloneqq \eta(z) z^\perp  = \i z \eta(z).
\]
Consider 
\[
 u_\eps \coloneqq  \frac{\id + \eps \varphi}{|\id + \eps \varphi|}
\]
Then for $0<\eps \ll 1$ we have $u_\eps\colon \S^1 \to \S^1$ is a smooth map of degree $1$, it also smoothly depends on $\eps$, and thus $(\eps \mapsto E_s(u_\eps))$ is at least $C^3$. By Peano's theorem we have for $\eps \to 0$
\[
 E_s(u_\eps) = E_s(u_0) + \eps \left.\frac{\dd}{\dd\delta} \right|_{\delta= 0} E_s(u_\delta) + \frac{1}{2}\eps^2 \brac{ \left. \frac{\dd^2}{\dd\delta^2} \right|_{\delta= 0} E_s(u_\delta) + o(1)}.
\]
In view of \Cref{la:firstsecondvariation}, it holds
\[
 E_s(u_\eps) = E_s(u_0) + \frac{1}{2}\eps^2 \brac{p \mathcal{Q}_p(\eta) + o(1)}.
\]
By \Cref{la:cos3theta}, for some constant $C_p > 0$, we have $\mathcal{Q}_p(\eta) = C_p (8-p)$. Hence, we have for $p = \frac{1}{s}$,
\[
 E_s(u_\eps) = E_s(u_0) + \frac{1}{2}\eps^2 \brac{ C_p (8-p) + o(1)}.
\]
Thus, if $p > 8$, i.e. $s < \frac{1}{8}$ we have for all suitably small $\eps$
\[
 C_p (8-p) + o(1) \leq -\frac{1}{2} C_p \abs{8-p}.
\]
This implies 
\[
 E_s(u_\eps) \leq E_s(\id) -\frac{1}{2} C_p \abs{8-p} \eps^2 < E_s(\id).
\]
Since $u_\eps$ is of degree $1$, we conclude that for $s < \frac{1}{8}$, it holds $\#_s 1 < E_s(\id)$.
\end{proof}

	\section{Approximate Stability}\label{s:approxstability}

One of the most crucial ingredients is the following result that relies on the stability theory developed in \cite{MS23} (this is stability in $s$) combined with the ``lateral'' stability  for $s= \frac{1}{2}$ in \cite[Theorem 1.3]{DSW23}.

	\begin{theorem}\label{th:stable}
	There exists $s_0 > \frac{1}{2}$, $p_0 > \frac{1}{s_0}$, $\gamma_0 > 0$ with the following properties.

	For any $\eps > 0$ there exists a $\delta > 0$ such that the following holds for any $s \in (\frac{1}{2}-\delta,\frac{1}{2}+\delta)$.

	If $u\colon  \S^1 \to \S^1$ is a degree one minimizer for $E_s$, i.e.\ $E_s(u) = \#_s 1$ then there exists a Möbius transform $m_b(z) \coloneqq  \frac{z+b}{1+\overline b z}$ for some $|b| < 1$ and a rotation $e^{-\i \lambda}$ for some $\lambda\in \R$ such that for $\widetilde{u} \coloneqq  e^{-\i \lambda} u \circ m_b $ we have
		\[
		\|\widetilde{u} - \id\|_{C^{\gamma_0}(\S^1)} + \|\widetilde{u}-\id\|_{W^{s_0,p_0}(\S^1)} \leq \eps
		\]
	\end{theorem}

The starting point are the following two results. Firstly, \cite[Theorem 1.3]{DSW23} using also \Cref{la:RtoS1}, we have the following result.
\begin{theorem}\label{th:H12stable}
			There exists $C>0$ such that the following holds. Assume $u \in W^{\frac{1}{2},2}(\S^1,\S^1)$ with $\deg u = 1$. Then there  exists $\phi = e^{\i \lambda} \frac{z+b}{1+\overline b z}$ with $|b| <1$ and $\lambda \in \R$ such that
			\[
			[u-\phi]_{W^{1/2,2}(\S^1)}^2 \leq C \brac{[u]_{W^{\frac{1}{2},2}(\S^1)}^2 - \#_{\frac{1}{2}} 1 }.
			\]
\end{theorem}
	
We also recall the usual $\eps$-regularity theorem (which follows easily from Cacciopoli's inequality and iteration, cf.\ \cite[Theorem 3.1]{Sucks23})
\begin{theorem}[Conformal regularity for minimizers]\label{th:epsregularity}
			There exists $\eps_0 > 0$, $\gamma >0$ such that the following holds.
			
			For any $R > 0$, $\Lambda  >0$, $0 < s_0 < s_1 < 1$ there exists a constant $\Gamma$ with the following properties.
			
			If for any $s \in [s_0,s_1]$ a map $u\in W^{s,\frac{1}{s}}(\S^1,\S^1)$ is local minimizer of $E_s$ with $E_s(u) \leq \Gamma$ and $R>0$ is such that 
			\begin{equation}\label{eq:uBxorsmall}         
			 \sup_{x_0 \in \S^1} [u]_{W^{s,\frac{1}{s}}(B(x_0,R)\cap \S^1)} < \eps_0,
			\end{equation}
			then it holds
			\[         
			 \|u\|_{C^\gamma(B(x_0,R/2)\cap \S^1)} + \|u\|_{W^{s+\gamma,\frac{1}{s} + \gamma}(B(x_0,R/2)\cap \S^1)} \leq \Gamma.
			\]
		\end{theorem}	
Observe that \eqref{eq:uBxorsmall} does not play well with reparametrizations of M\"obius transforms. That is to say, if $u$ satisfies \eqref{eq:uBxorsmall} in general $u\circ m_b$ will not, and there is no hope of getting a uniform $C^\gamma$-estimate as above. For this reason, we also require the conformal regularity theorem in \cite[Theorem 4.1]{MS23}.
 	\begin{theorem}[Conformal regularity for minimizers]\label{th:regularity}
			There exists $s_0 > \frac{1}{2}$ and constants $C_0 > 0$ and $\delta > 0$ such that for any $s \in (\frac{1}{2}-\delta,\frac{1}{2}+\delta)$ we have for a given $\deg 1$-minimizer of $E_s$ denoted $u$, i.e.\ if $E_s(u) = \#_s 1$
			\[         
			 [u]_{W^{s_0,\frac{1}{s_0}}(\S^1,\R^2)} \leq C.
			\]

		\end{theorem}

We will also use the following ``Poincar\`{e}''-type result.
\begin{lemma}\label{la:poincareonS1}
Let $0<s<1$ and $1\le q<\infty$. Then there exists
$C=C(s,q)>0$ such that for every map
$g\colon \mathbb S^1\to\mathbb S^1$ with $\id-g\in W^{s,q}(\mathbb S^1)$, one has
\[
\|\id-g\|_{L^q(\mathbb S^1)}
\le
C[\id-g]_{W^{s,q}(\mathbb S^1)}.
\]
\end{lemma}
\begin{proof}
We define
\[
A\coloneqq 
\frac1{2\pi}
\int_{\S^1}
\bigl(x-g(x)\bigr)\,\dx =
-\frac1{2\pi}\int_{\S^1}g(x)\,dx \in \R^2.
\]
Because $|g(x)|=1$, we have $|A|\le1$.
By the fractional Poincare inequality, it holds
\[
\left\|
\id-g-A
\right\|_{L^q(\mathbb S^1)}
\le
C[\id-g]_{W^{s,q}(\mathbb S^1)}.
\]
In order to conclude we will prove 
\begin{equation}\label{eq:pinconspheregoal}
|A|
\aleq \left\|
x-g(x)-A
\right\|_{L^q(\mathbb S^1)}
\end{equation}
If $A = 0$ there is nothing to show, so assume $A \neq 0$. Since $|g(x)|=1$, for every $x\in\mathbb S^1$, we have
\[
\left|
|x-A|-1
\right|
=
\left|
|x-A|-|g(x)|
\right|
\le
|x-g(x)-A|.
\]
Hence, it holds
\[
\left\|
|x-A|-1
\right\|_{L^q(\mathbb S^1)}
\le
\left\|
x-g(x)-A
\right\|_{L^q(\mathbb S^1)}.
\]
We now show that
\begin{equation}\label{eq:pinconspheregoalv2}
|A| \aleq_{q} \left\|
|x-A|-1
\right\|_{L^q(\mathbb S^1)}.
\end{equation}
Observe that for every $x \in \S^1$,
\begin{equation}\label{eq:poinc1mxma}
\begin{split}
        1-|x-A|
        &=
        \frac{1-|x-A|^2}{1+|x-A|}
        \\
        &=
        \frac{1-\bigl(|x|^2-2x\cdot A+|A|^2\bigr)}
        {1+|x-A|}
        \\
        &=
        \frac{2x\cdot A-|A|^2}
        {1+|x-A|}.
\end{split}
\end{equation}
We define 
\[
        e\coloneqq \frac{A}{|A|}\in\mathbb S^1 .
\]
Consider the arc
\[
        \mathcal{S} \coloneqq  \{x\in\mathbb S^1: x\cdot e\ge 3/4\}.
\]
For any $x \in \mathcal{S}$, it holds
\[
        2x\cdot A-|A|^2
        =
        |A|\bigl(2x\cdot e-|A|\bigr)
        \ge
        |A|\left(\frac32-1\right)
        =
        \frac{|A|}{2},
\]
where we used $|A|\le1$. Moreover, we have
\[
        1+|x-A|
        \le
        1+|x|+|A|
        \le
        3 .
\]
Therefore, it holds
\[
        1-|x-A|
        \ge
        \frac{|A|}{6} \qquad \forall x\in\mathcal{S}.
\]
We deduce that
\[
 |\mathcal{S}|^{\frac{1}{q}} \frac{|A|}{6} \leq \||x-A|-1\|_{L^q(\S^1)}
\]
Since the length $|\mathcal{S}|$ is independent of $A$, we have established \eqref{eq:pinconspheregoalv2} and this concludes the proof.
\end{proof}

	\begin{proof}[Proof of Theorem~\ref{th:stable}]
	Let $\eps>0$. By continuity of $s \mapsto \#_s 1$ there exists some $\delta>0$ such that 
	\[   
	 \abs{\#_s 1-\#_{\frac{1}{2}} 1} \leq \eps \quad \forall s \in \left( \frac{1}{2}-\delta,\frac{1}{2} +\delta \right).
	\]
    Taking $\delta$ potentially smaller, namely smaller than the $\delta_0$ from \Cref{th:regularity}, if $u$ is any degree 1 minimizer of $E_s$ for some $s \in [\frac{1}{2}-\delta,\frac{1}{2}+\delta]$ then $u \in W^{s_0,\frac{1}{s_0}}$ (with a uniform bound), and up to reducing $\delta$ depending only on the $W^{s_0,\frac{1}{s_0}}$-bound, $u$ satisfies
    \[
     \sup_{t \in [\frac{1}{2}-\delta,\frac{1}{2}+\delta]} \abs{E_s(u)-E_t(u)} < \eps.
    \]
In particular $u \in H^{\frac{1}{2}}$, and thus by \Cref{th:H12stable} there exists some M\"obius transform $\phi=e^{-\i \lambda} \frac{z+b}{1+\overline b z}$ such that
	\[   
	\begin{split}
	 [u-\phi]_{W^{1/2,2}}^2 \leq& C \brac{E_{\frac{1}{2}}(u) - \#_{\frac{1}{2}} 1 }\\
	 \leq& C \brac{E_{\frac{1}{2}}(u) - E_{s}(u)+ \#_s 1-\#_{\frac{1}{2}} 1 }\\
	 \leq& 2C\eps.
	 \end{split}
	\]
We now set $\tilde{u}\coloneqq e^{-\i \lambda}\, u \circ m_b$. By conformal invariance of the $W^{s,\frac{1}{s}}$-seminorm, \Cref{la:mobius-invariance}, it holds
\begin{equation}\label{eq:smallness_u_id}
 [\tilde{u}-\id]_{W^{1/2,2} (\S^1)}^2 \leq C\eps.
\end{equation}
Now by Gagliardo--Nirenberg, for $s  \in (\frac{1}{2},s_0)$ there exists a constant $C_2$, so that 
\[
[\tilde{u}-\id]_{W^{s,\frac1s}(\mathbb S^1)}
\le
[\tilde{u}-\id]_{W^{s_0,\frac1{s_0}}(\mathbb S^1)}^{\frac{s-\frac12}{s_0-\frac12}}
[\tilde{u}-\id]_{W^{\frac12,2}(\mathbb S^1)}^{\frac{s_0-s}{s_0-\frac12}}.
\]
By taking $s_1\coloneqq \frac{s_0 + \frac{1}{2}}{2}>\frac{1}{2}$, we find that for some positive $\sigma > 0$, it holds
\begin{equation}\label{eq:difference_u_id}
[\tilde{u}-\id]_{W^{s_1,\frac{1}{s_1}}(\mathbb S^1)}
\le C \eps^\sigma.
\end{equation}
On the other hand, for $\eps_0 > 0$ from \Cref{th:epsregularity} there exists some $R > 0$  (uniform for $s \in [\frac{1}{2}-\delta,\frac{1}{2}+\delta]$, since we stay away from $1$ and $0$) such that 
\begin{equation}\label{eq:small_energy_id}
 \sup_{x \in \S^1} [\id]_{W^{s_1,\frac{1}{s_1}}(\S^1 \cap B(x,R))} < \frac{\eps_1}{2}.
\end{equation}
By choosing $\eps$ possibly smaller, we may also assume that $ C \eps^\sigma < \frac{\eps_0}{2}$.
Combining \eqref{eq:difference_u_id} and \eqref{eq:small_energy_id}, we obtain
\[
 \sup_{x \in \S^1} [\tilde{u}]_{W^{s_1,\frac{1}{s_1}}(\S^1 \cap B(x,R))} < \eps_1
\]
From \Cref{th:epsregularity} we find, for any $x\in \S^1$,
\begin{equation}\label{eq:tildeu}
	\|\tilde{u}\|_{C^\gamma(\S^1 \cap B(x,R/2))} + \|\tilde{u}\|_{W^{s_1+\gamma,\frac{1}{s_1} + \gamma}(\S^1 \cap B(x,R/2))} \leq \Gamma.
\end{equation}
Up to reducing $\delta$ and $s_1>1/2$, we can assume that for any $s \in (\frac{1}{2}-\delta,\frac{1}{2}+\delta)$, it holds $s+\frac{\gamma}{3} > s_1$ and $\frac{1}{s} + \gamma > 2$. Again by Gagliardo--Nirenberg, for some powers $\theta \in (0,1)$, it holds for $s \in (\frac{1}{2}-\delta,\frac{1}{2}+\delta)$,
\[
 [\id-\tilde{u}]_{W^{s+\frac{\gamma}{3},2}(\S^1)} \aleq \|\id-\tilde{u}\|_{W^{s+\frac{\gamma}{2},2}(\S^1)}^{\theta} [\id-\tilde{u}]_{\dot{W}^{\frac{1}{2},2}(\S^1)}^{1-\theta}.
\]
By Sobolev embedding it holds $\|\id-\tilde{u}\|_{W^{s+\frac{\gamma}{2},2}(\S^1)} \aleq \|\id-\tilde{u}\|_{W^{s+\gamma,\frac{1}{s}+\gamma}(\S^1)}$, so that we have by \eqref{eq:tildeu} and \eqref{eq:smallness_u_id}
\[
 [\id-\tilde{u}]_{W^{s+\frac{\gamma}{3},2}(\S^1)} \aleq_{\Gamma} \eps^{\tilde{\sigma}} 
\]
By \Cref{la:poincareonS1}, potentially increasing the constant,
\[
 \|\id-\tilde{u}\|_{W^{s+\frac{\gamma}{3},2}(\S^1)} \aleq_{\Gamma} \eps^{\tilde{\sigma}}
\]
We obtain \Cref{th:stable} by Sobolev--Morrey embedding.
	\end{proof}

\section{Quantitative positivity of second variation in the \texorpdfstring{$s \in (1/8,1)$}{s>=1/8}-regime}\label{s:positivity}

In this section, we prove in \Cref{th:stable-picture-multiplier} below an estimate of $\mathcal{Q}_p(\eta)$ introduced in \Cref{la:Qprep}, with the identification $p=\frac{1}{s}$. We observe that $K_p$ is integrable if $p>3$ only (but it is integrable against $1-\cos(nt) =\frac{\sin^2(nt)}{1+\cos(nt)}$ for any $n \in \N$, which is what we will use for $p>1$).

\begin{theorem}[The case $s\in(\frac18,1)$]\label{th:stable-picture-multiplier}
Let $s\in(\frac18,1)$, and set $p=\frac1s\in(1,8)$. Let
$\eta\colon \S^1\to\R$ be such that
\begin{equation}\label{eq:decompo_eta}
\eta(e^{\i\theta})
=
\sum_{n\in\mathbb Z}c_n e^{\i n\theta}
\end{equation}
is a trigonometric polynomial, with $c_n\in\C$. Then we can write
\[
\mathcal{Q}_p(\eta)
=
4\pi
\sum_{n\in\mathbb Z}
M_{|n|}(p)|c_n|^2,
\]
and
\[
M_0(p)=M_1(p)=0.
\]
For $s \in [\frac{1}{8},1)$, equivalently for $p = \frac{1}{s} \in (1,8)$, we have $M_{n}(p) \geq 0$ for all $n \in \N$.
Moreover, if $1<p<3$, equivalently $s\in(\frac13,1)$, then there exist
constants $0<a_p\le b_p<\infty$ such that
\[
a_p\, n^{3-p}
\le
M_n(p)
\le
b_p\, n^{3-p}
\qquad
\text{for every }n\ge2.
\]
Consequently, if
\[
\widetilde\eta(e^{\i\theta})
\coloneqq 
\sum_{\substack{n\in\mathbb Z\\ n\neq0,\pm1}}
c_n e^{\i n\theta},
\]
then it holds
\[
a_p
\sum_{|n|\ge2}|n|^{3-p}|c_n|^2
\le
\mathcal{Q}_p(\eta)
\le
b_p
\sum_{|n|\ge2}|n|^{3-p}|c_n|^2.
\]
The constants $b_p,a_p$ can be chosen uniform as long as $p$ keeps a positive distance from $1$ and $3$. Equivalently, we have $\mathcal{Q}_p(\eta)
\asymp
\|\eta\|_{\dot H^{\frac{3-p}{2}}(\mathbb S^1)}^2$ whenever
\begin{equation}\label{eq:Fouriermodesvanishv1}
        \int_0^{2\pi}\eta (e^{\i\theta})\,\dd\theta=0,
        \qquad
        \int_0^{2\pi}\eta (e^{\i\theta})e^{\i\theta}\,\dd\theta=0.
\end{equation}

\end{theorem}

For later use we record \Cref{th:stable-picture-multiplier} combined with \Cref{la:firstsecondvariation} implies the following
\begin{corollary}\label{co:th:stable-picture-multiplier}
 Let $s\in(\frac13,1)$, and set $p=\frac1s\in(1,3)$. 
Then there exists constants $0<a_p < 1 < b_p$ satisfying the following property. Let $\eta\colon \mathbb S^1\to\mathbb R$ be such that for $n \in \{0,-1,1\}$,
\[
\frac{1}{2\pi}
\int_0^{2\pi}
\eta(e^{\i\theta}) e^{-\i n\theta}\,\dd\theta = 0.
\]
Then we have 
\[
 \|\eta\|_{\dot H^{\frac{3-p}{2}}(\mathbb S^1)}^2 \leq \delta^2 E_s(\id)[x^\perp \eta,x^\perp \eta] \leq b_p \|\eta\|_{\dot H^{\frac{3-p}{2}}(\mathbb S^1)}^2
\]
\end{corollary}

We start by proving the following lemma.
\begin{lemma}\label{la:QvsMn}
With the notations of \Cref{th:stable-picture-multiplier}, we have for any $p > 1$ 
\[
\mathcal{Q}_p(\eta)
=
4\pi
\sum_{n\in\mathbb Z}
M_{|n|}(p)|c_n|^2,
\]
where
\begin{equation}\label{eq:Mndef}
\begin{split}
M_n(p)
& \coloneqq \int_0^{2\pi}K_p(t)(1-\cos(nt))\,\dd t\\
& =2^{p-4} \int_0^{2\pi}  \left|\sin\brac{t/2}\right|^{p-4} \left[
p  \cos^2(t/2) -  1
\right] (1-\cos(nt))\,\dd t\\
& =2^{p-3} \int_0^{\pi}  \sin^{p-4}\brac{r} \brac{
p  \cos^2(r) -  1} (1-\cos(2nr))\,\dd r.
\end{split}
\end{equation}
\end{lemma}

Observe that the integral \eqref{eq:Mndef} makes sense for any $p>1$, since $1-\cos(nt) =\frac{\sin^2(nt)}{1+\cos(nt)}$.

\begin{proof}
By \eqref{eq:decompo_eta}, we have
\[
 \begin{split}
 \abs{\eta(e^{\i (\sigma+t)})-\eta(e^{\i \sigma})}^2
 =&\abs{\sum_{n \in \Z} c_n (e^{\i nt}-1)e^{\i n \sigma}}^2.
 \end{split}
\]
Thanks to \Cref{la:Qprep}, we obtain
\[
\mathcal{Q}_p(\eta) = \int_0^{2\pi} K_p(t)
\int_{0}^{2\pi} \abs{\sum_{n \in \Z} c_n (e^{\i nt}-1)e^{\i n \sigma}}^2\, \dd\sigma
\,\dd t .
\]
By orthogonality, for every $t \in [0,2\pi]$ we have 
\[
\begin{split}
&\int_{0}^{2\pi} \abs{\sum_{n \in \Z} c_n (e^{\i nt}-1)e^{\i n \sigma}}^2\, \dd\sigma
=2\pi \sum_{n \in \Z} \abs{ c_n}^2 \abs{e^{\i nt}-1}^2.
\end{split}
\]
Thus, we obtain
\begin{equation}\label{eq:Qp_v1}
\mathcal{Q}_p(\eta) =2\pi \sum_{n \in \Z}  \abs{ c_n}^2 \int_0^{2\pi} K_p(t)
\abs{e^{\i nt}-1}^2
\,\dd t. 
\end{equation}
Now, we have
\[
 \abs{e^{\i nt}-1}^2 = 2 - 2 \Re e^{\i nt} = 2-2 \cos(nt).
\]
Plugging this into \eqref{eq:Qp_v1}, we arrive at 
\[
\mathcal{Q}_p(\eta) =4\pi \sum_{n \in \Z}  \abs{ c_n}^2 \int_0^{2\pi} K_p(t)
\brac{1- \cos(nt)}
\,\dd t. 
\]
We conclude by observing that $\cos(nt) = \cos(|n|t)$, so that $M_{-n}(p)=M_{+n}(p)$.
\end{proof}

\begin{lemma}\label{la:m0-m1} 
For $M_n$ as in \eqref{eq:Mndef} we have $M_0(p)=0$ and $M_1(p)=0$.
\end{lemma}
\begin{proof}
$M_0(p) = 0$ is obvious since $1-\cos(0t) = 0$.

We have 
\[
\begin{split}
M_1(p)
&=
2^{p-3}
\int_0^{\pi}  \sin^{p-4}\brac{r} \brac{
p  \cos^2(r) -  1} (1-\cos(2nr))\,\dd r =0.
\end{split}
\]
\end{proof}

In order to go further, we will need some trigonometric identities.

\begin{lemma}[Some trigonometric identities]\label{la:trig-identities}
For $k\ge0$, $r \in [0,\pi]$ set
\[
U_k(r)\coloneqq \frac{\sin((k+1)r)}{\sin r},
\]
Then the following
identities hold. For every $n\ge1$ and $j\geq 0$, we have
\begin{equation}\label{eq:Sn-first-identity}
1-\cos(2nr)
=
2\sin^2r\,U_{n-1}(r)^2,
\end{equation}
\begin{equation}\label{eq:Sn-square-decomposition}
U_{n-1}(r)^2
=
\sum_{j=0}^{n-1}U_{2j}(r),
\end{equation}
\begin{equation}\label{eq:S2j-cos-expansion}
U_{2j}(r)
=
1+2\sum_{\ell=1}^{j}\cos(2\ell r).
\end{equation}
\end{lemma}

\begin{proof}
First, by definition,
\[
U_{n-1}(r)=\frac{\sin(nr)}{\sin r}.
\]
Therefore it holds
\[
2\sin^2r\,U_{n-1}(r)^2
=
2\sin^2(nr).
\]
From the identity
\[
1-\cos(2nr)=2\sin^2(nr),
\]
we obtain
\[
1-\cos(2nr)
=
2\sin^2r\,U_{n-1}(r)^2.
\]
This proves \eqref{eq:Sn-first-identity}.
Next we prove \eqref{eq:Sn-square-decomposition}. Since
\[
U_{2j}(r)\sin r=\sin((2j+1)r),
\]
we have
\[
\sin r\sum_{j=0}^{n-1}U_{2j}(r)
=
\sum_{j=0}^{n-1}\sin((2j+1)r).
\]
Multiplying by $2\sin (r)$ gives
\[
2\sin^2 (r)\sum_{j=0}^{n-1}U_{2j}(r)
=
2\sin (r)\sum_{j=0}^{n-1}\sin((2j+1)r).
\]
Using
\[
2\sin( r)\sin((2j+1)r)
=
\cos(2jr)-\cos((2j+2)r),
\]
we get the telescoping sum
\[
\begin{split}
2\sin^2 r\sum_{j=0}^{n-1}U_{2j}(r)
&=
\sum_{j=0}^{n-1}
\left[
\cos(2jr)-\cos((2j+2)r)
\right]
=
1-\cos(2nr).
\end{split}
\]
By \eqref{eq:Sn-first-identity}, we obtain
\[
2\sin^2 r\sum_{j=0}^{n-1}U_{2j}(r)
=
2\sin^2r\,U_{n-1}(r)^2.
\]
This proves \eqref{eq:Sn-square-decomposition} for any $r \in (0,\pi)$, by continuity also for $r =0$ and $r=\pi$.
Finally, we prove \eqref{eq:S2j-cos-expansion}. For $j=0$, this is simply the definition $U_0(r)=1$.
Now let $j\ge1$. We use
\[
2\sin (r)\cos(2\ell r)
=
\sin((2\ell+1)r)-\sin((2\ell-1)r).
\]
Then, we have
\[
\begin{split}
\left(
1+2\sum_{\ell=1}^{j}\cos(2\ell r)
\right)\sin (r)
&=
\sin (r)
+
\sum_{\ell=1}^{j}
\left[
\sin((2\ell+1)r)-\sin((2\ell-1)r)
\right]
\\
&=
\sin((2j+1)r).
\end{split}
\]
Therefore
\[
1+2\sum_{\ell=1}^{j}\cos(2\ell r)
=
\frac{\sin((2j+1)r)}{\sin r}
=
U_{2j}(r).
\]
This proves
\eqref{eq:S2j-cos-expansion}.
\end{proof}

\begin{lemma}\label{la:coefficient-computation}
Set for $p >1$
\begin{equation}\label{eq:Spdef}
S_p\coloneqq \int_0^\pi\sin^{p-2}(r) \dd r > 0.
\end{equation}
For $\ell\ge1$, define
\[
B_\ell(p)
\coloneqq 
\int_0^\pi
\sin^{p-2}(r)
(p\cos^2(r)-1)
\cos(2\ell r)
\dd r.
\]
Then it holds
\begin{equation}\label{eq:Bellformula}
B_\ell(p)
=
4\ell^2(p-1)\, S_p\,
\frac{
\prod_{m=1}^{\ell-1}(2m-p)
}{
\prod_{m=0}^{\ell}(p+2m)
}.
\end{equation}
The empty product is interpreted as $1$.
\end{lemma}

\begin{proof}
For $\ell\ge0$, define
\[
C_\ell(p)\coloneqq 
\int_0^\pi\sin^{p-2}(r)\cos(2\ell r)\dd r.
\]
Thus $C_0(p)=S_p$. We now establish the following identity
\begin{equation}\label{eq:recursionformula}
C_\ell(p)
=
S_p
\prod_{m=0}^{\ell-1}
\frac{2m+2-p}{p+2m}.
\end{equation}
Indeed, we compute the following integral with $p>1$
\[
\begin{split}
	0 & = \int_0^{\pi} \frac{\dd}{\dd r}
	\left[
	\sin^{p-1}(r)\cos((2\ell+1)r)
	\right]\dd r \\
	&=
	(p-1)
	\int_0^\pi
	\sin^{p-2}(r)\cos (r)\cos((2\ell+1)r)\dd r
	-
	(2\ell+1)
	\int_0^\pi
	\sin^{p-1}(r)\sin((2\ell+1)r)\dd r.
\end{split}
\]
Using
\[
2\cos (r)\cos((2\ell+1)r)
=
\cos(2\ell r)+\cos((2\ell+2)r),
\]
and
\[
2\sin (r)\sin((2\ell+1)r)
=
\cos(2\ell r)-\cos((2\ell+2)r),
\]
we get
\[
\begin{split}
0
&=
(p-1)
\int_0^\pi
\sin^{p-2}(r)\, \brac{\cos(2\ell r)+\cos((2\ell+2)r)}\dd r
\\
&\quad
-
(2\ell+1)
\int_0^\pi
\sin^{p-2}(r) \brac{\cos(2\ell r)-\cos((2\ell+2)r)}\dd r.
\end{split}
\]
In other words, it holds
\[
0
=
(p-1)(C_\ell(p)+C_{\ell+1}(p))
-
(2\ell+1)(C_\ell(p)-C_{\ell+1}(p)).
\]
Therefore, we have
\[
(p+2\ell)C_{\ell+1}(p)
=
(2\ell+2-p)C_\ell(p).
\]
Thus we have
\[
C_{\ell+1}(p)
=
\frac{(2\ell+2-p)}{p+2\ell}C_\ell(p).
\]
and since $C_0(p) = S_p$ we have established \eqref{eq:recursionformula}. In order to compute $B_{\ell}(p)$, we start from the following identities
\[
p\cos^2(r)-1
=
\frac{p-2}{2}+\frac p2\cos(2r).
\]
and 
\[
 2\cos(2r) \cos(2\ell r) = \cos(2(\ell+1) r) + \cos(2(\ell-1) r).
\]
Thus for any $\ell \geq 1$, we have
\[
\begin{split}
B_\ell(p)
& =\int_0^\pi
\sin^{p-2}(r)
\brac{\frac{p-2}{2}+\frac p2\cos(2r)}
\cos(2\ell r)
\dd r\\
& = \frac{p-2}{2}\int_0^\pi
\sin^{p-2}(r)
\cos(2\ell r)
\dd r + \frac {p}{4} \int_0^\pi
\sin^{p-2}(r)
\brac{\cos(2(\ell+1) r) + \cos(2(\ell-1) r)}
\dd r\\
& =
\frac{p-2}{2}C_\ell(p)
+
\frac p4
\left(C_{\ell-1}(p)+C_{\ell+1}(p)\right).
\end{split}
\]
Using \eqref{eq:recursionformula}, for $\ell\ge1$ we have
\[
C_{\ell-1}(p)
=
S_p
\frac{
\prod_{m=1}^{\ell-1}(2m-p)
}{
\prod_{m=0}^{\ell-2}(p+2m)
},
\]
\[
C_\ell(p)
=
S_p
\frac{
\prod_{m=1}^{\ell}(2m-p)
}{
\prod_{m=0}^{\ell-1}(p+2m)
},
\]
and
\[
C_{\ell+1}(p)
=
S_p
\frac{
\prod_{m=1}^{\ell+1}(2m-p)
}{
\prod_{m=0}^{\ell}(p+2m)
}.
\]
Therefore, it holds
\[
\begin{split}
B_\ell(p)
&=
\frac{p-2}{2}C_\ell(p)
+
\frac p4
\left(C_{\ell-1}(p)+C_{\ell+1}(p)\right)
\\
&=
S_p
\frac{
\prod_{m=1}^{\ell-1}(2m-p)
}{
\prod_{m=0}^{\ell}(p+2m)
}
\Bigg[
\frac{p-2}{2}(2\ell-p)(p+2\ell)
+
\frac p4
\left(
(p+2\ell-2)(p+2\ell)
+
(2\ell-p)(2\ell+2-p)
\right)
\Bigg].
\end{split}
\]
It remains to simplify the bracket. Set $q\coloneqq 2\ell$. Then we have
\[
\begin{split}
&\frac{p-2}{2}(2\ell-p)(p+2\ell)
+
\frac p4
\left(
(p+2\ell-2)(p+2\ell)
+
(2\ell-p)(2\ell+2-p)
\right)
\\
&=
\frac{p-2}{2}(q-p)(p+q)
+
\frac p4
\left(
(p+q-2)(p+q)
+
(q-p)(q+2-p)
\right)
\\
&=
\frac{p-2}{2}(q^2-p^2)
+
\frac p4
\left(
2p^2+2q^2-4p
\right)
\\
&=
\frac{p-2}{2}q^2
-
\frac{p-2}{2}p^2
+
\frac p2p^2
+
\frac p2q^2
-
p^2
\\
&=
(p-1)q^2 \\
&=
4\ell^2(p-1).
\end{split}
\]
We end up with 
\[
B_\ell(p)
=
4\ell^2(p-1)\,S_p\,
\frac{
\prod_{m=1}^{\ell-1}(2m-p)
}{
\prod_{m=0}^{\ell}(p+2m)
}.
\]
\end{proof}
\begin{lemma}[Partial-sum formula]\label{la:partial-sum-formula}
For $j\ge1$, $p>1$ define
\begin{equation}\label{eq:alphaj}
\alpha_j(p)
\coloneqq 
4j(j+1)(p-1)
\frac{
\prod_{m=2}^{j}(2m-p)
}{
\prod_{m=0}^{j}(p+2m)
},
\end{equation}
with the empty product interpreted as $1$. Then for $M_{n}(p)$ from \eqref{eq:Mndef}, 
\begin{equation}\label{eq:Mnpinalphaj}
M_n(p)
=
2^{p-2}S_p
\sum_{j=1}^{n-1}\alpha_j(p)
\qquad
\text{for every }n\ge1.
\end{equation}
Where $S_p$ is defined in \eqref{eq:Spdef}.
\end{lemma}
\begin{proof}
By \eqref{eq:Sn-first-identity},
\[
1-\cos(2nr)
=
2\sin^2(r)\,U_{n-1}(r)^2.
\]
Then by \eqref{eq:Mndef} 
\[
M_n(p)
=
2^{p-2}
\int_0^\pi
\sin^{p-2}(r)
(p\cos^2(r)-1)
U_{n-1}(r)^2
\dd r.
\]
Using \eqref{eq:Sn-square-decomposition}, we obtain
\begin{equation}\label{eq:MnAj}
M_n(p)
=
2^{p-2}\sum_{j=0}^{n-1}A_j(p),
\end{equation}
where
\[
A_j(p)
\coloneqq 
\int_0^\pi
\sin^{p-2}(r)
(p\cos^2(r)-1)
U_{2j}(r)
\dd r.
\]
We now aim to prove
\begin{equation}\label{eq:AjpSpalphajp}
A_j(p)=S_p\, \alpha_j(p).
\end{equation}
Once we have \eqref{eq:AjpSpalphajp} from \eqref{eq:MnAj} we find \eqref{eq:Mnpinalphaj} and can conclude.

For $p>1$, we claim
\begin{equation}\label{eq:integral_sin-cos}
	\int_0^\pi \sin^{p-2}(r)(p\cos^2(r)-1)\dd r =0.
\end{equation}
Indeed, we have by integration by parts
	\begin{align*}
		\int_0^{\pi} \sin^p(t)\dd t & = -\int_0^{\pi} \frac{\dd \cos(t)}{\dd t}\, \sin^{p-1}(t)\dd t \\
		& = (p-1)\int_0^{\pi} \cos^2(t)\, \sin^{p-2}(t)\dd t \\
		& = (p-1)\int_0^{\pi}  \left( \sin^{p-2}(t)- \sin^p(t)\right) \dd t.
	\end{align*}
	Thus, we have 
	\begin{equation*}
		p\int_0^{\pi} \sin^p(t)\dd t = (p-1)\int_0^{\pi} \sin^{p-2}(t)\dd t.
	\end{equation*}
	On the other hand, it holds
\begin{align*}
	\int_0^\pi \sin^{p-2}(r)(p\cos^2(r)-1)\dd r
	& = \int_0^\pi \sin^{p-2}(r)(p[1-\sin^2(r)]-1)\dd r \\
	& = (p-1)\int_0^{\pi} \sin^{p-2}(r)\dd r - p\int_0^{\pi} \sin^p(r)\dd r \\
	 & = 0.
\end{align*}
This proves \eqref{eq:integral_sin-cos}

We arrive at the following, where $B_{\ell}(p)$ defined in \Cref{la:coefficient-computation} and using \eqref{eq:S2j-cos-expansion}:
\begin{equation}\label{eq:AjBell}
A_j(p)
=
2\sum_{\ell=1}^jB_\ell(p).
\end{equation}
To obtain \eqref{eq:AjpSpalphajp}, it is thus enough to show
\begin{equation}\label{eq:alphajrecursion}
\alpha_j(p)-\alpha_{j-1}(p)
=
\frac{2B_j(p)}{S_p},
\end{equation}
since $\alpha_0(p) = 0$.

To prove \eqref{eq:alphajrecursion} we start from the definition \eqref{eq:alphaj}.
For $j=1$, by \eqref{eq:Bellformula}
\[
\alpha_1(p)
=
\frac{8(p-1)}{p(p+2)}
=
\frac{2B_1(p)}{S_p}.
\]
Since $\alpha_0 = 0$ this implies \eqref{eq:alphajrecursion} for $j=1$.

For $j \geq 2$, we have
\[
\begin{split}
&\alpha_j(p) - \alpha_{j-1}(p)\\
& = 4j(j+1)(p-1)
\frac{
\prod_{m=2}^{j}(2m-p)
}{
\prod_{m=0}^{j}(p+2m)
} - 4(j-1)j(p-1)
\frac{
\prod_{m=2}^{j-1}(2m-p)
}{
\prod_{m=0}^{j-1}(p+2m)
}\\
& = 4j(j+1)(p-1)
(2j-p) \frac{
 \prod_{m=2}^{j-1}(2m-p)
}{
\prod_{m=0}^{j}(p+2m)
} - 4(j-1)j(p-1)
(p+2j)\frac{
\prod_{m=2}^{j-1}(2m-p)
}{
\prod_{m=0}^{j}(p+2m)
}\\
& = \Big( (j+1)
(2j-p)  - (j-1)
(p+2j) \Big) 4 j (p-1) \frac{
\prod_{m=2}^{j-1}(2m-p)
}{
\prod_{m=0}^{j}(p+2m)
}\\
& =\brac{2-p} 8 j^2 (p-1) \frac{
\prod_{m=2}^{j-1}(2m-p)
}{
\prod_{m=0}^{j}(p+2m)
}\\
&= 8 j^2 (p-1) \frac{
\prod_{m={1}}^{j-1}(2m-p)
}{
\prod_{m=0}^{j}(p+2m)
}.
\end{split}
\]
By \eqref{eq:Bellformula}, we end up with
\[
B_j(p)
=
4j^2(p-1)\, S_p\,
\frac{
\prod_{m=1}^{j-1}(2m-p)
}{
\prod_{m=0}^{j}(p+2m)
}.
\]
This implies \eqref{eq:alphajrecursion}

\end{proof}

As a consequence of \Cref{la:partial-sum-formula}, we obtain the following result.

\begin{lemma}[Sign of the multipliers for $1<p<8$]\label{la:sign-stable-range}
For $1<p<8$ and $n\geq 2$, it holds $M_n(p)>0$.
\end{lemma}

\begin{proof}
Using \eqref{eq:Mnpinalphaj} and $2^{p-2}S_p>0$, we deduce that the sign of $M_n(p)$ is the sign of the partial sum
\[
\sum_{j=1}^{n-1}\alpha_j(p).
\]
The denominator of $\alpha_j(p)$ in \eqref{eq:alphaj} is positive for $p>1$. Therefore the sign of
$\alpha_j(p)$ is determined by
\[
\prod_{m=2}^{j}(2m-p).
\]

\underline{If $1<p<4$}, then $2m-p>0$ for every $m\ge2$. Hence
\[
\alpha_j(p)>0
\qquad
\text{for every }j\ge1 \quad \text{if $p \in (1,4)$}
\]
and therefore
\[
M_n(p)>0
\qquad
\text{for every }n\ge2 \quad \text{if $p \in (1,4)$}
\]

\underline{If $p=4$}, then
\[
\alpha_1(4)>0,
\qquad
\alpha_j(4)=0
\quad
\text{for every }j\ge2.
\]
Thus again
\[
M_n(4)>0
\qquad
\text{for every }n\ge2.
\]

Assume now that \underline{$4<p<6$}. Then
\[
\alpha_1(p)>0,
\]
and
\[
\alpha_j(p)<0
\qquad
\text{for every }j\ge2.
\]
Hence the following partial sums are decreasing
\[
N \mapsto P_N(p)\coloneqq \sum_{j=1}^N\alpha_j(p).
\]
Since $p>3$, the kernel $t \mapsto K_p(t)$ is integrable on $(0,2\pi)$. Thus, by
\cite[Proposition 3.2.1]{GrafakosCF}, \eqref{eq:Mnpinalphaj} and \eqref{eq:Mndef} we have
\[
M_n(p)
= 2^{p-2}S_p
P_{n-1} = 
\int_0^{2\pi}K_p(t)(1-\cos(nt))\dd t
\xrightarrow{n \to \infty}
\int_0^{2\pi}K_p(t)\dd t.
\]
Since the $P_N$ are decreasing in $N$ we obtain 
\begin{equation}\label{eq:Mnpgeqasd}
 M_n(p) \geq  \int_0^{2\pi}K_p(t)\dd t, \quad \forall n \geq 2.
\end{equation}
We compute this limit. With $t=2r$, recall the formula from \eqref{eq:Kp},
\[
\begin{split}
\int_0^{2\pi}K_p(t)\dd t
&=
2^{p-3}
\int_0^\pi
\sin^{p-4}(r)(p\cos^2(r)-1)\dd r.
\end{split}
\]
For $p>3$ we have by \eqref{eq:integral_sin-cos}
\[
\begin{split} 
\int_0^\pi\sin^{p-4}(r)\cos^2(r)\dd r  =
\frac1{p-2}
\int_0^\pi\sin^{p-4}(r)\dd r.
\end{split} 
\]
Therefore, it holds
\[
\begin{split}
\int_0^{2\pi}K_p(t)\dd t
&=
2^{p-3}
\left(
\frac p{p-2}-1
\right)
\int_0^\pi\sin^{p-4}(r)\dd r
=
\frac{2^{p-2}}{p-2}
\int_0^\pi\sin^{p-4}(r)\dd r
>0.
\end{split}
\]
By \eqref{eq:Mnpgeqasd}
\[
M_n(p)>0
\qquad
\text{for every }n\ge2 \quad \text{if $p \in (2,6)$}
\]

\underline{Now suppose $6\le p<8$}. 
Then it holds
\[
\alpha_1(p)>0,
\qquad
\alpha_2(p)<0,
\]
and
\[
\alpha_j(p)\ge0
\qquad
\text{for every }j\ge3.
\]
Moreover,
\[
\alpha_1(p)
=
\frac{8(p-1)}{p(p+2)},
\]
and
\[
\alpha_2(p)
=
24(p-1)\frac{4-p}{p(p+2)(p+4)}.
\]
Hence
\[
\begin{split}
\alpha_1(p)+\alpha_2(p)
&=
\frac{8(p-1)}{p(p+2)}
+
24(p-1)\frac{4-p}{p(p+2)(p+4)}
=
\frac{16(p-1)(8-p)}{p(p+2)(p+4)}.
\end{split}
\]
This is positive for $6\le p<8$. Therefore every partial sum in \eqref{eq:Mnpinalphaj} is positive, and we have shown 
\[
M_n(p)>0
\qquad
\text{for every }n\ge2 \quad \forall p \in [6,8).
\]
\end{proof}

In the case where $1<p<3$, we obtain the following asymptotic growth.

\begin{lemma}[Growth for $1<p<3$]\label{la:growth-1p3}
Let $1<p<3$. Then there exist constants $0<a_p\le b_p<\infty$ such that
\begin{equation}\label{eq:growthp13}
\forall n\ge2,\qquad a_p n^{3-p}
\le
M_n(p)
\le
b_p n^{3-p}.
\end{equation}
\end{lemma}

\begin{proof}
We first prove the upper bound. Let
\[
\rho(t)\coloneqq \min\{t,\pi-t\}
\qquad
\text{for }t\in[0,2\pi].
\]
Since $1-\cos(2nr) \geq 0$ and $\abs{\sin(t)} \ageq \rho(t)$ on $[0,\pi]$ we have whenever $p<4$ by \eqref{eq:Mndef}
\[
\begin{split}
M_n(p)
& = 2^{p-3} \int_0^{\pi}  \sin^{p-4}\brac{r} \brac{
p  \cos^2(r) -  1} (1-\cos(2nr))\dd r \\
& \aleq \int_0^\pi \rho(r)^{p-4} (1-\cos(2nr)) \dd r\\
& \aleq \int_0^{\frac{\pi}{2}} r^{p-4} (1-\cos(2nr)) \dd r.
\end{split}
\]
Observe that $1-\cos(2nr) \equiv \frac{\sin^2(2nr)}{1+\cos(2nr)} \aleq \min\{(nr)^2,1\}$. Thus, since $p \in (1,3)$, we obtain
\[
\begin{split}
M_n(p)
\aleq& \int_0^{\frac{1}{n}} r^{p-4} (nr)^2 \dd r + \int_{\frac{1}{n}}^{\infty} r^{p-4} \dd r\\
=& \frac{1}{p-1} n^2 \brac{\frac{1}{n}}^{p-1}  + \frac{1}{3-p} \brac{\frac{1}{n}}^{p-3}\\
\aleq& \max\left\{\frac{1}{p-1},\frac{1}{p-3}\right\}\, n^{3-p}.
\end{split}
\]
This proves the upper bound in \eqref{eq:growthp13}.

We now prove the lower bound using again \eqref{eq:Mndef}. We first observe that there exists some $N \in \N$ (depending on $p$) so that 
$\brac{p  \cos^2(r) -  1} \geq \frac{1}{2}$ and $\sin(r) \geq \frac{1}{2}r$ holds for all $r \in (\frac{\pi}{n},\frac{\pi}{2n})$ for any $n \geq N$. Thus we have
\[
\begin{split}
 \int_{\frac{\pi}{2n}}^{\frac{\pi}{n}}  \sin^{p-4}\brac{r} \brac{
p  \cos^2(r) -  1} (1-\cos(2nr))\dd r
\ageq & \int_{\frac{\pi}{2n}}^{\frac{\pi}{n}}  r^{p-4} (1-\cos(2nr))\dd r \\
\ageq & \int_{\frac{\pi}{2n}}^{\frac{2}{3}\frac{\pi}{n}}  r^{p-4} (1-\cos(2nr))\dd r \\
\ageq & \int_{\frac{\pi}{2n}}^{\frac{2}{3}\frac{\pi}{n}}  r^{p-4} \dd r.
\end{split}
\]
We obtain
\[
\int_{\frac{\pi}{2n}}^{\frac{\pi}{n}}  \sin^{p-4}\brac{r} \brac{
p  \cos^2(r) -  1} (1-\cos(2nr))\dd r
\ageq  \frac{1}{p-3} n^{p-3} \quad \forall n \geq N.
\]
Secondly, we observe that since $p >1$ there exists some $\delta = \delta(p) > 0$ such that $p \cos^2(r) - 1 \geq 0$ for all $r \in [0,\delta] \cup [\pi-\delta,\pi]$. We can assume that $N$ above is so large that $(\frac{\pi}{2n},\frac{\pi}{n}) \subset (0,\delta)$. Thus, for any $n \geq N$ we have
\[
\begin{split}
M_n(p)
& \geq 2^{p-3} \int_{\frac{\pi}{2n}}^{\frac{\pi}{n}}  \sin^{p-4}\brac{r} \brac{
p  \cos^2(r) -  1} (1-\cos(2nr))\dd r \\
&\qquad +2^{p-3} \int_{\delta}^{\pi-\delta}  \sin^{p-4}\brac{r} \brac{
p  \cos^2(r) -  1} (1-\cos(2nr))\dd r \\
& \ageq \frac{1}{p-3} n^{p-3} -c\int_{\delta}^{\pi-\delta}  \abs{\sin^{p-4}\brac{r}} dr \\
&\aeq \frac{1}{p-3} n^{p-3} -c \delta^{-(4-p)} \\
\end{split}
\]
In conclusion, we have shown
\[
M_n(p) \geq  \frac{C_{1;p}}{p-3} n^{p-3} -C_{2;p},  \qquad \forall n \geq N(p).
\]
On the other hand, from \Cref{la:sign-stable-range} we have in particular
\[
M_n(p)>0, \qquad \forall n \in \{2,\ldots,N(p)\}.
\]
Since these are finitely many $n$, we can find a new constant $a_p >0$ so that 
\[
M_n(p)\geq a_p\, n^{3-p}, \qquad \forall n \in \{2,\ldots,N(p)\}
\]
This concludes the proof of \Cref{la:growth-1p3}.
\end{proof}

\begin{proof}[Proof of \Cref{th:stable-picture-multiplier}]
The formula
\[
\mathcal{Q}_p(\eta)
=
4\pi
\sum_{n\in\mathbb Z}
M_{|n|}(p)|c_n|^2,
\]
follows from \Cref{la:QvsMn}, and
\[
M_0(p)=M_1(p)=0
\]
follows from \Cref{la:m0-m1}. The fact that for $p \in (1,8)$ $M_{n}(p) \geq 0$ for all $n \in \N$ is proven in \Cref{la:sign-stable-range} .  

For $p > 3$ the bound 
\[
a_p n^{3-p}
\le
M_n(p)
\le
b_p n^{3-p}
\qquad
\text{for every }n\ge2
\]
is shown in \Cref{la:growth-1p3}
\end{proof}

We will also need to following result.
\begin{corollary}
\label{co:la:truncated-coercivity-one-kernel}
Fix $1<p_- < p_+ <3$ and $A > 0$.

Then there exists $\rho_0 > 0$ and $c>0$ depending on $A$ and $p_+$ such that any $a\colon \S^1 \to \R$ with the property
\begin{equation}\label{hyp:a}
    \int_0^{2\pi}a(e^{\i\theta})\dd\theta=0,
        \qquad
        \int_0^{2\pi}a(e^{\i\theta})e^{\i\theta}\dd\theta=0,
\end{equation}
satisfies for any $p \in [p_-,p_+]$ and any $\rho \in (0,\rho_0)$ the following inequality
\[
\begin{split}
&A
\iint_{\{|x-y|\leq\rho\}}
\frac{|a(x)-a(y)|^2}{|x-y|^{4-p}}\dx\dy
\\
&
+
\frac12
\iint_{\{|x-y|\geq\rho\}}
\frac{
|x-y|^{p-2}
\left[
p-1-\frac p4|x-y|^2
\right]
}{|x-y|^2}
|a(x)-a(y)|^2\dx\dy
\geq
c
[a]_{H^{\frac{3-p}{2}}(\S^1)}^2 .
\end{split}
\]
\end{corollary}
\begin{proof}
Let us assume that $\frac{p}{4}(\rho_0)^2  < \frac{1}{2} (p-1)$ then whenever $|x-y| \leq \rho\leq\rho_0$, we have
\[
\frac{p-1}{2} \leq p-1-\frac{p}{4} |x-y|^2 \leq p-1.
\]
Thus, it holds
\[
\begin{split}
&A \iint_{\{|x-y|\leq\rho\}} \frac{|a(x)-a(y)|^2}{|x-y|^{4-p}}\dx\dy\\
& \geq A\frac{2}{p-1} \iint_{\{|x-y|\leq\rho\}} \frac{|a(x)-a(y)|^2\brac{p-1-\frac{p}{4} |x-y|^2}}{|x-y|^{4-p}}\dx\dy.
\end{split}
\]
If $A \frac{2}{p-1} \geq \frac{1}{2}$, then we obtain 
\[
\begin{split}
	&A
	\iint_{\{|x-y|\leq\rho\}}
	\frac{|a(x)-a(y)|^2}{|x-y|^{4-p}}\dx\dy
	\\
	&
	+
	\frac12
	\iint_{\{|x-y|\geq\rho\}}
	\frac{
		|x-y|^{p-2}
		\left[
		p-1-\frac p4|x-y|^2
		\right]
	}{|x-y|^2}
	|a(x)-a(y)|^2\dx\dy \\
	& \geq \iint_{\S^1\times\S^1}
	\frac{
		|x-y|^{p-2}
		\left[
		p-1-\frac p4|x-y|^2
		\right]
	}{|x-y|^2}
	|a(x)-a(y)|^2\dx\dy \\
	& \geq \frac{p-1}{2}\, [a]_{H^{\frac{3-p}{2}}(\S^1)}^2 .
\end{split}
\]
If $A \frac{2}{p-1} < \frac{1}{2}$, then we go in to frequencies. Fix from \Cref{th:stable-picture-multiplier} $c_* > 0$ such that (using \eqref{eq:identity2}, \eqref{eq:Kp} and \eqref{eq:scalar_products})
\[
\begin{split}
\iint_{\S^1\times\S^1}
\frac{
|x-y|^{p-2}
\left[
p-1-\frac p4|x-y|^2
\right]
}{|x-y|^2}
|x^n-y^n|^2\dx\dy
&=2 \pi \int_0^{2\pi} K_p(t)\, |e^{\i n t}-1|\dd t\\
&\geq c_* |n|^{3-p}.
\end{split}
\]
for $n\in \Z \setminus \{0,1,-1\}$.
We write
\[
        a(e^{\i\theta})=\sum_{n\in\Z}c_ne^{\i n\theta}.
\]
By \eqref{hyp:a}, we have $c_0=c_1=c_{-1}=0$.
Below we first choose $R > 0$ and then $\rho_0 > 0$ very small, and assume $\rho \in (0,\rho_0)$. Fix $|n|\geq2$.

\underline{Case 1: $|n| \rho \leq R$.}

For $x,y \in \S^n$ and $n \in \Z \setminus \{0\}$, it holds $|x^n-y^n| \leq n |x-y|$. Hence, we have
\[
\begin{split}
\left|
\iint_{\{|x-y|\leq\rho\}}
\frac{
|x-y|^{p-2}
\left[
p-1-\frac p4|x-y|^2
\right]
}{|x-y|^2}
|x^n-y^n|^2\dx\dy
\right|
& \leq 
\iint_{\{|x-y|\leq\rho\}}
|x-y|^{p-2}
n^2 \dx\dy.
\end{split}
\]
Since $p>p_->1$, it holds
\[
\begin{split}
	\left|
	\iint_{\{|x-y|\leq\rho\}}
	\frac{
		|x-y|^{p-2}
		\left[
		p-1-\frac p4|x-y|^2
		\right]
	}{|x-y|^2}
	|x^n-y^n|^2\dx\dy
	\right|
	& \leq C(p_-)\, n^2\, \rho^{p-1}\\
& =
C(p_-)\, (|n|\rho)^{p-1}\, |n|^{3-p}.
\end{split}
\]
The constant $C$ depends only on $p_-$, hence we can choose $R>0$ small enough so that
\[
    C R^{p-1}\leq \frac{c_*}{2}
        \qquad
        \forall p\in [p_-,p_+].
\]
For this choice of $R$, it holds
\[
\begin{split}
&\iint_{\{|x-y|\geq\rho\}}
\frac{
|x-y|^{p-2}
\left[
p-1-\frac p4|x-y|^2
\right]
}{|x-y|^2}
|x^n-y^n|^2\dx\dy
\\
&
=
\iint_{\S^1\times\S^1}
\frac{
|x-y|^{p-2}
\left[
p-1-\frac p4|x-y|^2
\right]
}{|x-y|^2}
|x^n-y^n|^2\dx\dy
\\
&\qquad
-
\iint_{\{|x-y|\leq\rho\}}
\frac{
|x-y|^{p-2}
\left[
p-1-\frac p4|x-y|^2
\right]
}{|x-y|^2}
|x^n-y^n|^2\dx\dy
\\
&
\geq
\frac{c_*}{2}|n|^{3-p}.
\end{split}
\]
Therefore in the case $|n| \rho \leq R$ we have
\[
\begin{split}
&A
\iint_{\{|x-y|\leq\rho\}}
\frac{|x^n-y^n|^2}{|x-y|^{4-p}}\dx\dy
+
\frac12
\iint_{\{|x-y|\geq\rho\}}
\frac{
|x-y|^{p-2}
\left[
p-1-\frac p4|x-y|^2
\right]
}{|x-y|^2}
|x^n-y^n|^2\dx\dy
\\
&\qquad\geq
\frac{c_*}{4}|n|^{3-p}.
\end{split}
\]
\underline{Case 2: $|n| \rho > R$.}

First we observe that for $x = e^{\i \theta}$ then $\dx = \dd\theta$ (arclength measure) and we obtain
\begin{equation*}
\begin{split}
A\iint_{\S^1 \times \S^1} \chi_{\{|x-y|\leq\rho\}}
\frac{|x^n-y^n|^2}{|x-y|^{4-p}}\dx\dy
& =A
\int_{0}^{2\pi}\int_{0}^{2\pi}\chi_{\{|1-e^{\i (\theta-\sigma)}|\leq\rho \} }\,
\frac{|1-e^{\i (\theta -\sigma)n}|^2}{|1-e^{\i (\theta-\sigma)}|^{4-p}}\dd\theta\dd\sigma
\\
& = A
\int_{0}^{2\pi}\int_{-\sigma}^{2\pi-\sigma}\chi_{\{ |1-e^{\i t}|\leq\rho \} }\,
\frac{|1-e^{\i tn}|^2}{|1-e^{\i t}|^{4-p}}\dd t\dd\sigma
\\
& = 2\pi A
\int_{0}^{2\pi}\chi_{\{|1-e^{\i t}|\leq\rho\} }\, \frac{|1-e^{\i tn}|^2}{|1-e^{\i t}|^{4-p}}\, \dd t
\\
& =c_p\pi A
\int_{0}^{2\pi}\chi_{ \left\{|4 \sin\brac{\frac{t}{2}}|\leq\rho\right\} } \frac{\sin^2\brac{\frac{tn}{2}}}{\abs{\sin\brac{\frac{t}{2}}}^{4-p}}\dd t.
\end{split} 
\end{equation*} 
Therefore, we have
\begin{equation}\label{eq:weirdfreq:1}
\begin{split} 
A\iint_{\S^1 \times \S^1} \chi_{\{|x-y|\leq\rho\}}
\frac{|x^n-y^n|^2}{|x-y|^{4-p}}\dx\dy & \geq \tilde{c}\,_p\,\pi\, A
\int_{0}^{c\rho} \frac{\sin^2\brac{\frac{tn}{2}}}{t^{4-p}}\dd t \\
& = \hat{c}_p\, \pi\, A\, |n|^{3-p}
\int_{0}^{\frac{c|n|}{2}\rho} \frac{\sin^2\brac{\tau}}{\tau^{4-p}}\dd\tau \\
& \geq \hat{c}_p\, \pi\, A\, |n|^{3-p}
\int_{0}^{\frac{c}{2}R} \frac{\sin^2\brac{\tau}}{\tau^{4-p}}\dd\tau \\
& \eqqcolon A\, c_{R,p}\, |n|^{3-p}.
\end{split}
\end{equation}
where we observe that since $1<p_- < p < p_+<3$, $c_{R,p}>0$ is a constant independent of $\rho$, $n$, and the particular $p$. Next observe that
\[
 p-1-\frac p4|x-y|^2 \leq 0 \quad \Leftrightarrow \quad |x-y|^2 \geq 4\frac{p-1}{p}.
\]
Thus, we have
\[
\begin{split}
&\frac{1}{2}\iint_{\{|x-y|\geq\rho\}}
\frac{
|x-y|^{p-2}
\left[
p-1-\frac p4|x-y|^2
\right]
}{|x-y|^2}
|a(x)-a(y)|^2\dx\dy \\
& \geq - \frac{1}{2}\iint_{\left\{|x-y|\geq \sqrt{4\frac{p-1}{p}}\right\}}
\frac{|a(x)-a(y)|^2}{|x-y|^{4-p}}\dx\dy\\
& \geq - \frac{1}{2}\, \left(4\, \frac{p-1}{p}\right)^{\frac{p-4}{2}}\iint_{\left\{|x-y|\geq \sqrt{4\frac{p-1}{p}}\right\}}|a(x)-a(y)|^2\dx\dy.
\end{split}
\]
In particular for $a(x)=x^n$, we obtain
\[
\begin{split}
\frac{1}{2}\iint_{\{|x-y|\geq\rho\}}
\frac{
|x-y|^{p-2}
\left[
p-1-\frac p4|x-y|^2
\right]
}{|x-y|^2}
|x^n-y^n|^2\dx\dy
\geq- C(p_-,p_+).
\end{split}
\]
Using $|n| \rho \geq R$ and $p < p_+ < 3$, we obtain
\begin{equation}\label{eq:weirdfreq:2}
\begin{split}
&\frac{1}{2}\iint_{\{|x-y|\geq\rho\}}
\frac{
|x-y|^{p-2}
\left[
p-1-\frac p4|x-y|^2
\right]
}{|x-y|^2}
|x^n-y^n|^2\dx\dy  \geq - C\, |n|^{3-p}\, \brac{\frac{R}{\rho}}^{p-3}.
\end{split}
\end{equation}
Thus combining \eqref{eq:weirdfreq:1} and \eqref{eq:weirdfreq:2},
\[
\begin{split}
 &A
\iint_{\{|x-y|\leq\rho\}}
\frac{|x^n-y^n|^2}{|x-y|^{4-p}}\dx\dy
\\
&\qquad
+
\frac12
\iint_{\{|x-y|\geq\rho\}}
\frac{
|x-y|^{p-2}
\left[
p-1-\frac p4|x-y|^2
\right]
}{|x-y|^2}
|x^n-y^n|^2\dx\dy\\
\geq&\brac{A\, c_{R,p} - C \rho^{3-p} R^{p-3}}|n|^{3-p}.
\end{split}
\]
We choose $\rho < \rho_0$ and $\rho_0$ small enough such that
\[
 C \rho_0^{3-p} R^{p-3} \leq \frac{1}{2} A c_{R,p}.
\]
We conclude that
\[
\begin{split}
 &A
\iint_{\{|x-y|\leq\rho\}}
\frac{|x^n-y^n|^2}{|x-y|^{4-p}}\dx\dy
\\
&\qquad
+
\frac12
\iint_{\{|x-y|\geq\rho\}}
\frac{
|x-y|^{p-2}
\left[
p-1-\frac p4|x-y|^2
\right]
}{|x-y|^2}
|x^n-y^n|^2\dx\dy\\
& \geq \frac{A c_{R,p}}{2} |n|^{3-p}.
\end{split}
\]
in the case $|n| \rho \geq R$.
\end{proof}

\section{Regularity of minimizing harmonic maps}\label{s:lipschitzreg}
To apply our argument in the next section, we need Lipschitz regularity of minimizing maps. 
\begin{theorem}\label{th:Lipregularity}
For any $\alpha > 0$, $s_0>\frac{1}{2}$, $\lambda > 0$ there exists $\delta > 0$ and $\Lambda > 0$ such that the following holds.
If for $s \in (\frac{1}{2},\frac{1}{2}+\delta)$, i.e.\ $p=\frac{1}{s} <2$ if $u \in C^\alpha \cap W^{s,\frac{1}{s}} \cap W^{s_0,2}(\S^1,\S^1)$ with
\[
 [u]_{C^\alpha} + [u]_{W^{s,\frac{1}{s}}} + [u]_{W^{s_0,2}} \leq \lambda 
\]
is a $E_s$ minimizer within maps $\S^1 \to \S^1$ of the same degree then $u \in C^{1,\sigma}(\S^1)$ for some $\sigma = \sigma_{p-2} > 0$ and we have the estimate
\begin{equation}\label{eq:H3pbetaest}
 [u]_{\lip} \leq  \frac{\Lambda}{\sqrt{\abs{p-2}}}
\end{equation}
\end{theorem}

\begin{remark}[Comparison to standard literature]
	The fractional $p$-Laplacian not being very well understood, especially in the vectorial case, we did not manage to get a uniform higher regularity bound. However, for our purposes a non-exponential blowup as $p \to 2^{-}$ is acceptable.
	
Our argument for \eqref{eq:H3pbetaest} essentially follows ideas such as the one used \cite{BL17,GL24}. At first it may look like we get a slightly better exponent than \cite{GL24}, namely we move beyond $C^1$, but this is because we get a Sobolev space estimate that (in one dimension) implies the desired estimate. The methods below should work in any domain and target dimension but it seems unlikely they give a better result than the state of the art with these methods, see also \cite[Corollary 1.9]{DKLN25}.

Also let us remark that the assumption $s_0>\frac{1}{2}$ is likely not necessary, but avoids a boot-strapping argument for the (a priori) absorption argument in \Cref{lem:localized-difference-quotient-beta}.

We state \Cref{th:Lipregularity} for minimizers, which makes a comparison argument \eqref{eq:energyineq} possible. It is likely that the same ideas work also for critical points, using the $C^\alpha$-condition to pertubatively compare the solution to the one of unconstrained fractional $s$-$p$-Laplace (which is minimizing for the unconstrained $E_s$ energy)
\end{remark}

\begin{lemma}[Inner variations]
\label{lem:uniform-inner-variations}
There exist constants $r_g>0$, $c_g>0$ and $C_g>0$ such that the following holds.

Let $I\subset\S^1$ be an arc of length $0<r\le r_g$
Let $X\in C^\infty(\S^1,\mathbb R)$ such that $\operatorname{spt}X\Subset I$ and, in arclength coordinates,
\begin{equation}\label{eq:Xestimates}
        |X|\le C_g,
        \qquad
        \left|\frac{dX}{d\ell}\right|\le \frac{C_g}{r},
        \qquad
        \left|\frac{d^2X}{d\ell^2}\right|\le \frac{C_g}{r^2} \qquad
        \left|\frac{d^3X}{d\ell^3}\right|\le \frac{C_g}{r^3}.
\end{equation}
For $|h|\le c_g r$, the map $\psi_h\colon \S^1\to \S^1$ defined by $\psi_h(e^{\i t})\coloneqq e^{\i(t+hX(e^{it}))}$ is an orientation-preserving diffeomorphism. Consequently, for every continuous map
$u\colon \S^1\to\S^1$ one has $\deg(u\circ\psi_h)=\deg u$.

Furthermore, let $p \in (1,\infty)$ and $s = \frac{1}{p} \in (0,1)$.
If $u\in W^{s,p}(\S^1,\S^1)$ is a minimizer of $E_{s}$ among maps $v\colon \S^1 \to \S^1$ with the same degree as $u$, then
\begin{equation}\label{eq:energyestimateucircpsih}
        0\le
        E_s(u\circ\psi_h)-E_s(u)
        \le
        C_g\, r^{-2}E_s(u)h^2.
\end{equation}
\end{lemma}

\begin{proof}
We have $\psi_h(x) = x e^{\i hX(x)}$.

Computing the derivative we confirm that for any $|h|\le c_g\, r$, the map $\psi_h\colon  \S^1 \to \S^1$ is an orientation-preserving diffeomorphism of $\S^1$ -- choosing $c_g$ sufficiently small, depending only on $C_g$. Also, the map $h \mapsto u\circ\psi_h$ is a homotopy, so $\deg(u\circ\psi_h)=\deg u$.

By change of variables,
\[
\begin{aligned}
        E_s(u\circ\psi_h)
        &=
        \iint_{\S^1\times \S^1}
        \frac{|u(a)-u(b)|^p}{|a-b|^2}
        A_h(a,b)
        \dd a\dd b,
\end{aligned}
\]
where for $a,b\in\S^1$
\begin{equation}\label{eq:Ah}
\begin{aligned}
        A_h(a,b)
        &\coloneqq 
        \left|\frac{d}{d\ell}\psi_h^{-1}(a)\right|
        \left|\frac{d}{d\ell}\psi_h^{-1}(b)\right|
        \frac{|a-b|^2}{|\psi_h^{-1}(a)-\psi_h^{-1}(b)|^2}.
\end{aligned}
\end{equation}
Since $\psi_h$ is a diffeomorphism, we can write its inverse in arclength coordinates as \[ \psi_h^{-1}(e^{\i t})=e^{\i\chi_h(t)} \]
where
\[ \chi_h(t)+hX(e^{\i\chi_h(t)})=t. \]
Differentiating in $t$, we get 
\[ \partial_t \chi_h(t) = \frac{1}{ 1+\i h\, e^{\i \chi_h(t)}\, \frac{dX}{d\ell}(e^{\i\chi_h(t)}) }. \] 
Differentiating in $h$, we get 
\[
 \partial_h \chi_h(t) + X(e^{\i \chi_h(t)}) +\i h\, e^{\i \chi_h(t)} \partial_{\ell} X(e^{\i \chi_h(t)}) \partial_h \chi_h(t) = 0.
\]
Using \eqref{eq:Xestimates} and up to reducing the choice of $c_g$ in the definition of $h$, we successively find 
\[
 \partial_t \chi_h(t) \aeq 1,\ |\partial_h \chi_h(t)| \aeq 1,\ |\partial_{hh} \chi_h(t)| \aleq r^{-1},\ |\partial_h \partial_t \chi_h(t)| \aleq r^{-1}, |\partial_{hh} \partial_t \chi_h(t)| \aeq r^{-2}
\]
Now write $a=e^{\i t}$ and $b=e^{\i\tau}$ in \eqref{eq:Ah}. Then we have
\[ \begin{split} 
A_h(e^{\i t},e^{\i\tau}) =& \partial_t\chi_h(t)\, \partial_t\chi_h(\tau)\, \frac{|e^{\i t}-e^{\i\tau}|^2} {|e^{\i\chi_h(t)}-e^{\i\chi_h(\tau)}|^2} 
=\partial_t\chi_h(t)\, \partial_t\chi_h(\tau)\, \frac{\left|\sin\brac{\frac{t-\tau}{2}}\right|^2} {\left|\sin\brac{\frac{\chi_h(t)-\chi_h(\tau)}{2}}\right|^2}.
\end{split} 
\]
We want to estimate $\partial_{hh} A_h(e^{\i t},e^{\i\tau})$. By rotation we may assume that the interval $I$ is centred around $1$. The estimates where $|e^{\i t} - e^{\i \tau}| \geq c$ are then simple: if $x$ is away from $I$, then it holds $\psi_h(x) = x$, and so if $a$ and $b$ are close to each other and not in $I$, $A_h(a,b) = 1$ and thus the second derivative in $h$ simply vanishes. Similarly if $a \in I$ but $b$ is not close to $I$, then also $\psi^{-1}(a)$ is not close to $b=\psi^{-1}(b)$ and the expression can be estimated by the above estimates.
In the remaining case we can assume that $|t-\tau| + |\chi_h(t)-\chi_h(\tau)| < \frac{\pi}{4}$.
Set \[ M_h(t,\tau) \coloneqq  \int_0^1 \partial_t\chi_h\bigl(\tau+\theta(t-\tau)\bigr) \,\dd \theta. \] Then \[ \chi_h(t)-\chi_h(\tau) = (t-\tau)M_h(t,\tau). \]

From above we see that 
\[
 M_h(t,\tau) \aeq 1, \quad |\partial_h M_h(t,\tau)| \aleq r^{-1}, \quad |\partial_{hh} M_h(t,\tau)| \aleq r^{-2}
\]

and since 
\[
\left|\sin\brac{\frac{t-\tau}{2}}\right| \aeq |t-\tau| \aeq \left|\sin\brac{\frac{\chi_h(t)-\chi_h(\tau)}{2}}\right| 
\]
we can conclude that 
\[
 |\partial_{hh} A_h|\aleq r^{-2}.
\]
with a constant depending only on $C_g$, $c_g$ and $r_g$. This implies for all $h \ll r$,
\[
\begin{aligned}
        | \partial_{hh} E_s(u\circ\psi_h)|\aleq r^{-2} E_s(u) 
\end{aligned}
\]
On the other hand, since $h \mapsto E_s(u \circ \psi_h)$ is twice differentiable, and has a minimum at $h=0$ we get the claim.
\end{proof}

\begin{lemma}[Convexity]
\label{lem:chordal-midpoint-convexity}
\label{lem:semiconvexity-chordal}
Fix $1<p_-\leq 2$. There exist constants $\eta_{\mathrm{cvx}}\in\left(0,\frac12\right)$, $c_{{cvx}}>0$ and $C_{{cvx}}>0$ depending only on $p_-$, such that the following hold.
For every $p\in[p_-,2]$ and every $\xi,\zeta\in [-\eta_{{\rm cvx}}, \eta_{{\rm cvx}}]$, one has
\begin{equation}\label{eq:closebyconvexity}
\begin{aligned}
        &|e^{\i\xi}-1|^p
        +
        |e^{\i\zeta}-1|^p
        -
        2\left|e^{\i\frac{\xi+\zeta}{2}}-1\right|^p \geq
        c_{{cvx}}
        (|\xi|+|\zeta|)^{p-2}
        |\xi-\zeta|^2 .
\end{aligned}
\end{equation}
In particular, it holds
\[
\begin{aligned}
        &|e^{\i\xi}-1|^p
        +
        |e^{\i\zeta}-1|^p
        -
        2\left|e^{\i\frac{\xi+\zeta}{2}}-1\right|^p \geq
        c_{{cvx}}|\xi-\zeta|^2 .
\end{aligned}
\]
Moreover, for every $p\in[p_-,2]$, every $\omega\in\mathbb S^1$, and every
$\xi,\zeta\in\mathbb R$, one has
\begin{equation}\label{eq:faryconvexity}
\begin{aligned}
        &|e^{\i\xi}-\omega|^p
        +
        |e^{\i\zeta}-\omega|^p
        -
        2\left|e^{\i\frac{\xi+\zeta}{2}}-\omega\right|^p
        \geq
        -C_{{cvx}}|\xi-\zeta|^2 .
\end{aligned}
\end{equation}
\end{lemma}

\begin{proof}
Let
\[
        f(t)\coloneqq |e^{\i t}-1|^p
        =
        \left(2\left|\sin\frac t2\right|\right)^p .
\]
For $p>1$, this function is $C^1(-\pi,\pi)$ and $C^2((-\pi,\pi) \setminus \{0\})$. For $0<|t|<\pi$, it holds
\begin{equation*}
	f'(t) = {\rm sign}(t)\, 2^{p-1}\, p\, \left| \sin\frac{t}{2}\right|^{p-1}\, \cos\left(\frac{t}{2}\right).
\end{equation*}
We differentiate once more
\[
        f''(t)
        =
        p\left(2\left|\sin\frac t2\right|\right)^{p-2}
        \left[
        (p-1)\cos^2\frac t2
        - {\rm sign}(t)\, 
        \sin^2\frac t2
        \right].
\]
Since $p>1$ we have that $f \in W^{2,1}(-\pi,\pi)$.

We first prove \eqref{eq:closebyconvexity}. Choose
$\eta_{\mathrm{cvx}}\in(0,\frac12)$ so small that
\[
        \sin^2\frac{\eta_{\mathrm{cvx}}}{2}
        \leq
        \frac{p_- -1}{2p_-}.
\]
Then, for every $p\in[p_-,2]$ and every $0<|t|\leq\eta_{\mathrm{cvx}}$,
\[
\begin{aligned}
        (p-1)\cos^2\frac t2-\sin^2\frac t2
        &\geq
        (p_- -1)\cos^2\frac t2-\sin^2\frac t2        \\
        &=
        (p_- -1)-p_-\sin^2\frac t2                  \\
        &\geq
        \frac{p_- -1}{2}.
\end{aligned}
\]
Hence
\[
        f''(t)
        \geq
        \frac{p_-(p_- -1)}{2}
        \left(2\left|\sin\frac t2\right|\right)^{p-2}.
\]
Since $2|\sin(t/2)|\leq |t|$ and $p-2\leq0$, we get for $0<|t|\leq\eta_{\mathrm{cvx}}$,
\[
        f''(t)
        \geq
        \frac{p_-(p_- -1)}{2}|t|^{p-2}.
\]
Consider
\[
        s\in [-1,1]\mapsto
        \Phi(s) \coloneqq  f\left(
        \frac{\xi+\zeta}{2}
        +
        s\frac{\xi-\zeta}{2}
        \right).
\]
By the fundamental theorem
\[
\begin{split}
 \Phi(1) + \Phi(-1)-2\Phi(0) &= \int_0^1 \int_{-\tau}^\tau \Phi''(\sigma) \dd\sigma \dd\tau\\
 &= \int_0^1 \int_{-\tau}^\tau f''\left(
        \frac{\xi+\zeta}{2}
        +
        \sigma \frac{\xi-\zeta}{2}
        \right) \brac{\frac{\xi-\zeta}{2}}^2 \dd\sigma \dd\tau\\
 & \ageq_{p_-}  \int_0^1 \int_{-\tau}^\tau \left| \frac{\xi+\zeta}{2}
        + \sigma \frac{\xi-\zeta}{2} \right|^{p-2}  \dd\sigma \dd\tau \brac{\frac{\xi-\zeta}{2}}^2.
\end{split} 
\]
By restricting the integral for $\tau\in [\frac{1}{2},1]$, we obtain 
\[
\begin{split} 
	\Phi(1) + \Phi(-1)-2\Phi(0)  & \ageq_{p_-}  \int_{\frac{1}{2}}^1 \int_{-\frac{1}{2}}^\frac{1}{2} \abs{\frac{\xi+\zeta}{2}
        + \sigma \frac{\xi-\zeta}{2}}^{p-2}  \dd\sigma \dd\tau \brac{\frac{\xi-\zeta}{2}}^2\\
        & \ageq_{p_-}  \int_{-\frac{1}{2}}^\frac{1}{2} \abs{\frac{\xi+\zeta}{2}
        + \sigma \frac{\xi-\zeta}{2}}^{p-2}  \dd\sigma\, \brac{\frac{\xi-\zeta}{2}}^2\\
        & \overset{p<2}{\ageq} \int_{-\frac{1}{2}}^\frac{1}{2} \max\left \{\abs{\frac{\xi+\zeta}{2}},\abs{\frac{\xi-\zeta}{2}}\right \}^{p-2}  \dd\sigma\, \brac{\frac{\xi-\zeta}{2}}^2.
 \end{split}
\]
This proves \eqref{eq:closebyconvexity}. 

As for \eqref{eq:faryconvexity}, we have
\[
 f''(t) \geq -p 2^{p-2} \abs{\sin\brac{\frac{t}{2}}}^p \geq -p 2^{p-2}.
\]
Therefore, we obtain
\[
\begin{split}
 \Phi(1) + \Phi(-1)-2\Phi(0) =& \int_0^1 \int_{-\tau}^\tau \Phi''(\sigma) d\sigma d\tau\\
 =& \int_0^1 \int_{-\tau}^\tau f''\left(
        \frac{\xi+\zeta}{2}
        +
        \sigma \frac{\xi-\zeta}{2}
        \right) \brac{\frac{\xi-\zeta}{2}}^2 \dd\sigma \dd\tau\\
 \geq & -p 2^{p-2} \int_0^1 \int_{-\tau}^\tau \brac{\frac{\xi-\zeta}{2}}^2 \dd\sigma \dd\tau.
 \end{split}
\]
This readily implies \eqref{eq:faryconvexity}.
\end{proof}

The following Lemma is the backbone of the argument.
\begin{lemma}
\label{lem:localized-difference-quotient-beta}
Fix $\alpha>0$, $\lambda>0$ and $0<\delta_*\le \frac16$.
There exist constants $R_{\mathrm{outer}}>0$, $\varepsilon_{\mathrm{sep}}>0$ and $C_{{dq}}>0$
depending only on $\alpha$, $\lambda$, and $\delta_*$, such that the following holds.
Let
\[
        \frac12\le s<\frac12+\delta_*,
        \qquad
        p=\frac1s \in \left[ \frac{3}{2},2\right].
\]
Let $u\in C^\alpha(\S^1,\S^1)\cap W^{s,p}(\S^1,\S^1)$ satisfy
\begin{equation}\label{eq:ubound}
        [u]_{C^\alpha(\S^1)}
        +
        [u]_{W^{s,p}(\S^1)} + E_s(u)
        \le
        \lambda.
\end{equation}
Assume that $u$ minimizes $E_s$ in its degree class.

Fix $x_0 \in \S^1$, and let $0<R < R_{\rm outer}<\frac{r_{\rm lift}}{100}$ where $r_{\rm lift}=r_{\rm lift}(\lambda,\alpha)$ is defined in  \Cref{lem:uniform-small-arc-lifting}. We write $u=e^{\i\varphi}$ on $B(x_0,R)$.
Fix $0< \rho \leq \eps_{\mathrm{sep}} R$
and $0<r \leq \frac{\rho}{4}$.
Let $X\in C^\infty(\S^1,\mathbb R)$
such that $X=1$ on $B(x_0,r)$ with  $\operatorname{spt}X\Subset B(x_0,\rho)$,
and, in arclength coordinates,
\[
        |X|\le C_g,
        \qquad
        \left|\frac{dX}{d\ell}\right|
        \le
        \frac{C_g}{\rho},
        \qquad
        \left|\frac{d^2X}{d\ell^2}\right|
        \le
        \frac{C_g}{\rho^2}.
\]
For $|h|\le c_g\rho$, we define
\[
        \psi_h(e^{\i t})
        =
        e^{\i(t+hX(e^{\i t}))}.
\]
Here $c_g>0$ is chosen small enough, so that $\psi_h$ is an orientation-preserving diffeomorphism of $\S^1$ and
\[
        \psi_h\bigl(B(x_0,R)\bigr)\subset B(x_0,R).
\]
We define
\begin{equation}\label{eq:thetah}
        \theta_h(x)
        \coloneqq 
        \varphi(\psi_h(x))-\varphi(x)
        \qquad\forall x\in I.
\end{equation}
Then, for every $0\le \beta\le \min\{\alpha(2-p),2s_0-1\}$, one has
\[
\begin{aligned}
        &\iint_{\S^1\times \S^1}
        \frac{|\theta_h(x)-\theta_h(y)|^2}
        {|x-y|^{2+\beta}}
        \dd\ell(x)\dd\ell(y) \le
        C_{{dq}}
        \bigl(C_{{lift}}(1+\lambda)\bigr)^{2-p}
        (1+\lambda)^2
        \rho^{-2}
        h^2.
\end{aligned}
\]
\end{lemma}

\begin{proof}
Choose $R_{\mathrm{outer}} \in \big(0,\min(1,r_{\mathrm{lift}},r_g)\big)$, where $r_{\mathrm{lift}}$ is the lifting radius from \Cref{lem:uniform-small-arc-lifting}, and $r_g$ is defined in \Cref{lem:uniform-inner-variations}. For $R<R_{\mathrm{outer}}\le r_{\mathrm{lift}}$, \Cref{lem:uniform-small-arc-lifting} gives a continuous lift
\[
        u=e^{\i\varphi}
        \qquad\text{on }B(x_0,R),
\]
with
\begin{equation}\label{eq:varphicalphas324}
        [\varphi]_{C^\alpha(B(x_0,R))}
        \le
        C_{{lift}}(1+\lambda).
\end{equation}
Let $\eps\in(0,1)$ and choose $\rho\leq \eps R_{\rm outer}$. Define $X$ and $\psi_h$ as in \Cref{lem:uniform-inner-variations} with $I=B(x_0,\rho)$. Because $\operatorname{spt}X\Subset B(x_0,\rho)$, one has
\[
        \psi_h(x)=x
        \qquad\forall x\in B(x_0,R)\setminus B(x_0,\rho).
\]
Consequently, it holds
\[
        \theta_h(x)=0
        \qquad\forall x\in B(x_0,R)\setminus B(x_0,\rho).
\]

Define $v_h\coloneqq u\circ\psi_h$. On $B(x_0,R)$ one has
\[
        v_h(x)=e^{\i\varphi(\psi_h(x))}.
\]
We define the map 
\begin{equation*}
	m_h(x) = \begin{cases}
		e^{\i\frac{\varphi(x)+\varphi(\psi_h(x))}{2}} & \text{ if }x\in B(x_0,R),\\
		u(x) & \text{ if }x\in\S^1\setminus B(x_0,R).
	\end{cases}
\end{equation*}
Since $\theta_h=0$ near $\partial B(x_0,R)$, the map $m_h$ is globally well-defined and continuous.
Both, $v_h$ and $m_h$ have the same regularity as $u$ and also the same degree, as $h \mapsto v_h$ and $h\mapsto m_h$ are homotopies. Since $u$ is assumed to be an $E_s$ minimizer, we have 
\[
        E_s(u)\le E_s(v_h),
        \qquad \text{ and }\qquad 
        E_s(u)\le E_s(m_h).
\]
By \Cref{lem:uniform-inner-variations}, using also \eqref{eq:ubound},
\[
        E_s(v_h)-E_s(u)
        \le
        C(1+\lambda)^2\rho^{-2}h^2.
\]
Hence we obtain 
\begin{equation}\label{eq:energyineq}
        E_s(u)+E_s(v_h)-2E_s(m_h) = E_s(v_h)-E_s(u)+2\brac{E_s(u)-E_s(m_h)}
        \le
        C(1+\lambda)^2\rho^{-2}h^2.
\end{equation}
For $x,y\in B(x_0,R)$, one has
\begin{align*}
	|u(x)-u(y)|^p & = \left|e^{\i(\varphi(x)-\varphi(y))}-1\right|^p, \\
	|v_h(x)-v_h(y)|^p & = \left|e^{\i(\varphi(\psi_h(x))-\varphi(\psi_h(y)))}-1\right|^p, \\
	|m_h(x)-m_h(y)|^p & =	\left| e^{\i\frac{
			\varphi(x)-\varphi(y)
			+
			\varphi(\psi_h(x))-\varphi(\psi_h(y))
		}{2}}
	-1
	\right|^p.
\end{align*}
By \eqref{eq:varphicalphas324}, we have
\begin{equation}\label{eq:varphixvarphipsihx}
\begin{aligned}
        |\varphi(x)-\varphi(y)|
        +
        |\varphi(\psi_h(x))-\varphi(\psi_h(y))|
        & \le
        C_{{lift}}(1+\lambda)|x-y|^\alpha \frac{|h|}{\rho}\\
        &\le
        C_{{lift}}(1+\lambda)|x-y|^\alpha
\end{aligned}
\end{equation}
We set 
\[
\xi \coloneqq \varphi(x)-\varphi(y), \qquad \text{ and }\qquad 
 \zeta\coloneqq  (\varphi(\psi_h(x))-\varphi(\psi_h(y))).
\]
Then we have if $R_{\rm{outer}}$ is chosen small enough,
\[
 |\xi| + |\zeta| \leq C_{\rm lift} (1+\lambda)  R^\alpha \leq \eta_{\rm{cvx}}, 
\]
where $\eta_{\rm{cvx}}$ is defined in \Cref{lem:chordal-midpoint-convexity}. Hence by \eqref{eq:closebyconvexity} 
\[
\begin{aligned}
        &|u(x)-u(y)|^p
        +
        |v_h(x)-v_h(y)|^p
        -
        2|m_h(x)-m_h(y)|^p
        \\
        \ge &
        c_{\rm cvx}
        \big(
        |\varphi(x)-\varphi(y)|
        +
        |\varphi(\psi_h(x))-\varphi(\psi_h(y))|
        \big)^{p-2}
        \left|
        \varphi(\psi_h(x))-\varphi(\psi_h(y))
        -
        \varphi(x)+\varphi(y)
        \right|^2.
\end{aligned}
\]
Since $p-2\leq 0$ we can use \eqref{eq:varphixvarphipsihx} to estimate further, and find for any $x,y \in B(x_0,R)$.
\[
\begin{aligned}
        &|u(x)-u(y)|^p
        +
        |v_h(x)-v_h(y)|^p
        -
        2|m_h(x)-m_h(y)|^p
        \\
        \ge&
        c_{\rm cvx}
        \bigl(C_{{lift}}(1+\lambda)\bigr)^{-(2-p)}
        \max\{1,|x-y|^{-\alpha(2-p)}\} \,
        \left|
        \varphi(\psi_h(x))-\varphi(\psi_h(y))
        -
        \varphi(x)+\varphi(y)
        \right|^2.
\end{aligned}
\]
By definition of $\theta_h$ in \eqref{eq:thetah}, we have 
\[
\begin{aligned}
        &\varphi(\psi_h(x))-\varphi(\psi_h(y))
        -
        \varphi(x)+\varphi(y) =
        \theta_h(x)-\theta_h(y).
\end{aligned}
\]
We thus have shown
\begin{equation}\label{eq:lowerbound}
\begin{aligned}
        &|u(x)-u(y)|^p
        +
        |v_h(x)-v_h(y)|^p
        -
        2|m_h(x)-m_h(y)|^p
        \\
        &\qquad\ge
        c_{\rm cvx}
        \bigl(C_{{lift}}(1+\lambda)\bigr)^{-(2-p)}
        \max\{1,|x-y|^{-\alpha(2-p)}\}
        \cdot
        \left|
         \theta_h(x)-\theta_h(y)
        \right|^2.
\end{aligned}
\end{equation}
whenever  $x,y\in B(x_0,R)$. The same inequality holds trivially whenever $x,y \in \S^1 \setminus B(x_0,R)$ since $u=v_h=m_h$ and $\theta_h=0$ on $\S^1 \setminus B(x_0,R)$.

By symmetry in $x$ and $y$, the only case left to consider is $y \in \S^1 \setminus B(x_0,R)$ and $x \in B(x_0,R)$.
In this case, it holds $v_h(y)=m_h(y)=u(y)$. If $x \in B(x_0,R) \setminus B(x_0,\rho)$ we have $\psi_h(x) = x$ and thus $v_h(x)=m_h(x)=u(x)$, which means $|v_h(x)-v_h(y)| = |u_h(x)-u_h(y)|  = |u(x)-u(y)|$ 
and $|\theta_h(x)-\theta_h(y)| =0$, so \eqref{eq:lowerbound} holds also in this case.

We now consider the case $y \in \S^1 \setminus B(x_0,R)$ and $x \in B(x_0,\rho)$. In particular, it holds for some small uniform $c>0$
\[
|x-y| \geq |y| -\rho \geq c( |y| + R).
\] 
We apply \Cref{lem:semiconvexity-chordal}, \eqref{eq:faryconvexity}, and now have (recall that $\theta_h(y) = 0$)
\begin{equation}\label{eq:lowerboundv2}
\begin{aligned}
        &|u(x)-u(y)|^p
        +
        |v_h(x)-u(y)|^p
        -
        2|m_h(x)-u(y)|^p
        \\
        \ge& 
        -C_{{sc}}|\theta_h(x)|^2 \\
	    \ge &
        -C_{{sc}}|\theta_h(x)-(\theta_h)_{B}|^2.
\end{aligned}
\end{equation}
Here $B$ is a ball of radius $\frac{1}{100} \rho$ inside the set $B(x_0,R) \setminus \supp \theta_h$ -- such a ball exists since $\rho \leq \eps R$ and we can choose $\eps$ as small as needed. We then have from \eqref{eq:energyineq} combined with \eqref{eq:lowerbound} applied for 
\[
(x,y) \in \Sigma \coloneqq \S^1 \times \S^1 \setminus \Bigg( \Big( \big(\S^1 \setminus B(x_0,R)\big) \times B(x_0,\rho) \Big) \cup \Big( B(x_0,\rho) \times \big(\S^1 \setminus B(x_0,R) \big) \Big) \Bigg)
\] 
and \eqref{eq:lowerboundv2} in the other points for any $\beta \in [0,\min\{\alpha(2-p),2s_0-1\}]$
\begin{equation*}
	\begin{split} 
	C(1+\lambda)^2\rho^{-2}h^2 & \geq \bigl(C_{{lift}}(1+\lambda)\bigr)^{-(2-p)} \iint_{\Sigma}
	\frac{|\theta_h(x)-\theta_h(y)|^2}{|x-y|^{2+\beta}} \dx \dy \\
	 & \quad -2 \int_{\S^1 \setminus B(x_0,R)}  \frac{1}{c^2 (|y|+R)^2} \left(\int_{B(x_0,\rho)} |\theta_h(x)-(\theta_h)_{B}|^2 \dx \right) \dy \\
	 \geq& \bigl(C_{{lift}}(1+\lambda)\bigr)^{-(2-p)} \iint_{\S^1 \times \S^1}
	 \frac{|\theta_h(x)-\theta_h(y)|^2}{|x-y|^{2+\beta}} \dx \dy\\
	 &-\brac{\bigl(C_{{lift}}(1+\lambda)\bigr)^{-(2-p)} +2} \int_{\S^1 \setminus B(x_0,R)}  \frac{1}{c^2 (|y|+R)^2} \left( \int_{B(x_0,\rho)} |\theta_h(x)-(\theta_h)_{B}|^2 \dx\right) \dy.
	 \end{split} 
\end{equation*}
Therefore, we have
\begin{equation}\label{eq:Sobolev_theta}
\begin{split}
 &C(1+\lambda)^2\rho^{-2}h^2\\
 \geq& \bigl(C_{{lift}}(1+\lambda)\bigr)^{-(2-p)} \iint_{\S^1 \times \S^1}
         \frac{|\theta_h(x)-\theta_h(y)|^2}{|x-y|^{2+\beta}} \dx \dy\\
         &-\brac{\bigl(C_{{lift}}(1+\lambda)\bigr)^{-(2-p)} +2} \frac{1}{R} \int_{B(x_0,\rho)} |\theta_h(x)-(\theta_h)_{B}|^2 \dx\\
         \geq& \bigl(C_{{lift}}(1+\lambda)\bigr)^{-(2-p)} \iint_{\S^1 \times \S^1}
         \frac{|\theta_h(x)-\theta_h(y)|^2}{|x-y|^{2+\beta}} \dx \dy\\
         &-C\brac{\bigl(C_{{lift}}(1+\lambda)\bigr)^{-(2-p)} +2} \frac{\rho^{1+\beta}}{R} \iint_{\S^1\times \S^1} \frac{|\theta_h(x)-\theta(z)|^2}{|x-y|^{2+\beta}} \dx \dz.
\end{split}
\end{equation}
We observe that since $1+\beta < 2s_0$ by definition of $\beta$, the integrals on the right-hand side are finite.

Now if we choose $\eps_{\rm sep}>0$ small enough to verify
\[
 C\brac{\bigl(C_{{lift}}(1+\lambda)\bigr)^{-(2-p)} +2} \eps_{\rm sep}^{1+\beta} R^\beta \leq \frac{1}{2} \bigl(C_{{lift}}(1+\lambda)\bigr)^{-(2-p)}.
\]
Plugging this inequality in \eqref{eq:Sobolev_theta}, we obtain
\[
\begin{split}
 &C(1+\lambda)^2\rho^{-2}h^2 \geq \frac{1}{2} \bigl(C_{{lift}}(1+\lambda)\bigr)^{-(2-p)} \iint_{\S^1\times \S^1}
         \frac{|\theta_h(x)-\theta_h(y)|^2}{|x-y|^{2+\beta}} \dx \dy.
\end{split}
\]
We can conclude.
\end{proof}

Next, we recall the Morrey embedding, which is of course standard. The point is the dependency on $\beta$. 
\begin{lemma}[Sobolev-Morrey embedding]
\label{thm:morrey-beta}
Let $0<\beta\le \frac12$. There exists a universal constant $C>0$, independent of $\beta$, such that for $f \in L^1(B_R)$ we have for a.e.\ $x \in B_R$
\[
        |f(x)-(f)_{B_R}|
        \le
        C\beta^{-\frac12}R^{\frac\beta2}
        \left(
        \int_{B_{4R}}\int_{B_{4R}}
        \frac{|f(a)-f(b)|^2}{|a-b|^{2+\beta}}
        \dd a\dd b
        \right)^{\frac12},
\]
whenever the right-hand side is finite.
\end{lemma}

\begin{proof}
We denote\footnote{Below one could also work with balls instead of annuli in the proof, but the annuli have the advantage that we can find $\sqrt{\beta}$ instead of $\beta$ as the dependency.}
\[
A_{2^k R}(x) \coloneqq B_{2^k R}(x)\setminus B_{2^{k-1} R}(x).
\]
If $x$ is a Lebesgue point 
\[
 |f(x)-(f)_{B_R}| \leq \sum_{k=-\infty}^0 |(f)_{A_{2^k R}(x)}-(f)_{A_{2^{k-1} R}(x)}|+|(f)_{A_{R}(x)} - (f)_{B_R}|
\]
For any $k \leq 0$, it holds by Jensen inequality
\[
\begin{split}
|(f)_{A_{2^k R}(x)}-(f)_{A_{2^{k-1} R}(x)}|^2 \aleq & (2^k R)^{-2} \int_{A_{2^{k} R}(x)}\int_{A_{2^{k-1} R}(x)} |f(a)-f(b)|^2 \dd a \dd b\\
\aleq&(2^k R)^{\beta} \int_{A_{2^{k} R}(x)}\int_{A_{2^{k-1} R}(x)} \frac{|f(a)-f(b)|^2}{|a-b|^{2+\beta}} \dd a \dd b\\
\aleq&(2^k R)^{\beta} \int_{A_{2^{k} R}}\int_{B_{4R}} \frac{|f(a)-f(b)|^2}{|a-b|^{2+\beta}} \dd a \dd b.
\end{split}
\]
This implies 
\[
\begin{split}
|(f)_{A_{2^k R}(x)}-(f)_{A_{2^{k-1} R}(x)}|
\aleq&(2^k R)^{\frac{\beta}{2}} \brac{\int_{A_{2^{k} R}}\int_{B_{4R}} \frac{|f(a)-f(b)|^2}{|a-b|^{2+\beta}} \dd a \dd b}^{\frac{1}{2}}.
\end{split}
\]
Thus,
\[
\begin{split}
\sum_{k =-\infty}^0 |(f)_{A_{2^k R}(x)}-(f)_{A_{2^{k-1} R}(x)}|
& \aleq R^{\frac{\beta}{2}} \sum_{k=-\infty}^0  \brac{2^{k \frac{\beta}{2}}\, \int_{A_{2^{k} R}}\int_{B_{4R}} \frac{|f(a)-f(b)|^2}{|a-b|^{2+\beta}} \dd a \dd b}^{\frac{1}{2}} \\
& \aleq R^{\frac{\beta}{2}}  \brac{\sum_{k=-\infty}^0 2^{k \beta}}^{\frac{1}{2}}\, \brac{\sum_{k=-\infty}^0 \int_{A_{2^{k} R}}\int_{B_{4R}} \frac{|f(a)-f(b)|^2}{|a-b|^{2+\beta}} \dd a \dd b}^{\frac{1}{2}} \\
& \aleq R^{\frac{\beta}{2}}  \brac{\sum_{k=-\infty}^0 2^{k \beta}}^{\frac{1}{2}}\, \brac{\int_{B_{4R}}\int_{B_{4R}} \frac{|f(a)-f(b)|^2}{|x-y|^{2+\beta}} \dd a \dd b}^{\frac{1}{2}}.
\end{split}
\]
We estimate the first sum using that the inequality $e^u\geq 1+u$,
\[
 \sum_{k=-\infty}^0 2^{k \beta} = \frac{1}{1-2^{-\beta}} \leq \frac{1}{\beta\, \log(2)}.
\]
Hence, we conclude that
\[
\begin{split}
\sum_{k =-\infty}^\infty |(f)_{A_{2^k R}(x)}-(f)_{A_{2^{k-1} R}(x)}|
\aleq R^{\frac{\beta}{2}}  \beta^{-\frac{1}{2}} \brac{\int_{B_{4R}}\int_{B_{4R}} \frac{|f(a)-f(b)|^2}{|a-b|^{2+\beta}} \dd a \dd b}^{\frac{1}{2}}.
\end{split}
\]
With a similar argument we find 
\[
 |(f)_{B_{R}(x)} - (f)_{B_R}| \aleq R^{\frac{\beta}{2}} \brac{\int_{B_{4R}}\int_{B_{4R}} \frac{|f(a)-f(b)|^2}{|a-b|^{2+\beta}} \dd a \dd b}^{\frac{1}{2}}.
\]
\end{proof}

\begin{proof}[Proof of \Cref{th:Lipregularity}]
From \Cref{lem:localized-difference-quotient-beta} we find for a fixed number $r$ (which depends on $\lambda$ and $\alpha$) for any $x_0 \in \S^1$ that $u=e^{\i \varphi}$ in $B(x_0,r)$ and $\varphi$ satisfies by \Cref{lem:localized-difference-quotient-beta}, with $\beta = \min\{\alpha(2-p),2s_0-1\}$,
\[
 \sup_{|h| \leq 1} \int_{B(x_0,r)}\int_{B(x_0,r)} \frac{1}{|h|^2}\frac{|\delta_h \varphi(x)-\delta_h \varphi(y)|^2}{|x-y|^{2+\beta}} \dx \dy \leq C.
\]
By the standard difference quotient argument,see, e.g.\ \cite[Proposition 4.8]{GiqauintaMartinazzi},
\[
 \int_{B(x_0,r)}\int_{B(x_0,r)} \frac{|\varphi'(x)-\varphi'(y)|^2}{|x-y|^{2+\beta}} \dx \dy \leq C.
\]
This implies that $\varphi' \in H^{\frac{1+\beta}{2}}_{\rm loc}(B(x_0,r))$. By \Cref{thm:morrey-beta} we also conclude 
\[
 |\varphi'(x)-(\varphi')_{B(x_0,r)}| \aleq_{r,C} \beta^{-\frac{1}{2}}.
\]
Since $(\varphi')_{B(x_0,r)}$ is bounded (by the fundamental theorem and using \Cref{lem:uniform-small-arc-lifting}, $\varphi$ is unique)
\[
|\varphi'(x)| \aleq_{r,C} \beta^{-\frac{1}{2}} \quad \text{in $B(x_0,r)$}.
\]
Since $u'(x) = e^{\i \varphi} \varphi'$, we find that
\[
|u'(x)| \aleq_{r,C} \beta^{-\frac{1}{2}} \quad \text{in $B(x_0,r)$}.
\]
Since $\varphi' \in H^{\frac{1+\beta}{2}}_{\rm loc}(B(x_0,r))$ we also find the same for $u'$, and conclude.
\end{proof}

\section{Local and global uniqueness: Proof of Theorem~\ref{th:localminisid} and Theorem~\ref{th:minisId}}\label{s:idmini}

If $[u-\id]_{W^{s,\frac{1}{s}}(\S^1,\S^1)} \ll 1$ then we get a uniform $\eps$-regularity result, and any merely critical point is local $C^\alpha$-regular, cf. \cite{S15,MS18} and for local minimizers \cite{Sucks23}, and we have for some $\alpha > 0$ (independent of $(s_-,s_+)$
\[
 [u]_{C^\alpha} \leq \Lambda.
\]
By Gagliardo--Nirenberg, again using $[u-\id]_{W^{s,\frac{1}{s}}(\S^1,\S^1)} \ll 1$, using also \Cref{la:poincareonS1}, we then can assume
\[
\|u-\id\|_{L^\infty} \ll 1.
\]
If $p<2$, \Cref{th:Lipregularity}, for local minimizers\footnote{This is likely extendable to critical points without too much work following the same ideas as in the proof of minimizers, but to keep the length of this work withing reasonable ranges, we will not discuss this here.} we also find
\[
 \|u\|_{\lip} \leq \frac{\Lambda}{|p-2|}.
\]
Thus, \Cref{th:localminisid} is a consequence of

\begin{theorem}[Identity is a local minimizer for $s \in (\frac{1}{3},1)$]\label{th:localminisidv2}
	Let $\frac{1}{3} < s_-  < s_+ < 1$ and $\Lambda \geq 1$. Then there exists $\eps= \eps(s_-,s_+, \Lambda)>0$ such that any degree one critical point $u \in W^{s,\frac{1}{s}}(\S^1,\S^1)$ with $s \in [s_-,s_+]$ that satisfies $\|u-\id\|_{L^\infty} \leq \eps$ and if $p=\frac{1}{s}<2$ then additionally assume $[u]_{\lip} \leq \frac{\Lambda}{|p-2|}$, then $u \in \mathcal{M}$.
\end{theorem}

We begin with several elementary results.
\begin{lemma}\label{la:acontrol}
Let $0<s<1$ and $1\le q<\infty$. There exists $C=C(s,q)>0$ satisfying the following property. Let $a\in W^{s,q}(\mathbb S^1)$ be real-valued, and define
\[
u(x)=\cos(a(x))x + \sin(a(x)) x^\perp \colon \S^1 \to \S^1.
\]
Then it holds
\[
\inf_{\lambda\in\mathbb R}
\bigl[e^{\i\lambda }u-\id\bigr]_{W^{s,q}(\mathbb S^1)}
\le
C\, [a]_{W^{s,q}(\mathbb S^1)}.
\]
\end{lemma}

\begin{proof}
We have $u(e^{\i t}) = e^{\i t+\i a(e^{\i t})}$. Let
\[
\bar a=\frac1{2\pi}\int_0^{2\pi} a(e^{\i t})\dd t \in \R
\]
We rotate $u$ by $e^{-\i\bar a}$. Observe,
\[
e^{-\i\bar a}u(e^{\i t})-e^{\i t}
=
e^{\i t}\bigl(e^{\i(a(e^{\i t})-\bar a)}-1\bigr).
\]
For two points $e^{\i t},e^{\i r}\in \mathbb S^1$, we estimate
\[
\begin{aligned}
&\left|
\bigl(e^{-\i\bar a}u(e^{\i t})-e^{\i t}\bigr) - \bigl(e^{-\i\bar a}u(e^{\i r})-e^{\i r}\bigr) \right|
\\
&\le
\left|e^{\i(a(e^{\i t})-\bar a)}-e^{\i(a(e^{\i r})-\bar a)}\right|
+
\left|e^{\i(a(e^{\i r})-\bar a)}-1\right|\,|e^{\i t}-e^{\i r}|.
\end{aligned}
\]
Since $|e^{\i\alpha}-e^{\i\beta}|\le |\alpha-\beta|$ and $|e^{\i\alpha}-1|\le |\alpha|$,
we get
\[
\left| \bigl(e^{-\i\bar a}u-\id\bigr)(e^{\i t}) - \bigl(e^{-\i\bar a}u-\id\bigr)(e^{\i r}) \right|
\le
|a(e^{\i t})-a(e^{\i r})|
+
|a(e^{\i r})-\bar a|\,|e^{\i t}-e^{\i r}|.
\]
Therefore, taking the $W^{s,q}$ seminorm,
\[\begin{split}
&\bigl[e^{-\i\bar a}u-\id\bigr]_{W^{s,q}}^q\\
\le&
C\, [a]_{W^{s,q}}^q
+
C\, \int_0^{2\pi}\int_0^{2\pi}
\frac{|a(e^{\i r})-\bar a|^q |e^{\i t}-e^{\i r}|^q}
{|e^{\i t}-e^{\i r}|^{1+sq}}
\dd t\dd r\\
=&
C\, [a]_{W^{s,q}}^q
+C_{s,q}\, \int_0^{2\pi} |a(e^{\i r})-\bar{a}|^q\dd r.
\end{split}
\]
We conclude by Poincar\'e inequality.
\end{proof}

The following is to take care of the first Fourier modes, i.e.\ it helps us ensure \eqref{eq:Fouriermodesvanishv1}.
\begin{lemma}\label{la:local-normalization-phase}
There exist universal constants $\varepsilon_0>0$ and $C>0$ with the
following property. Let $u\colon \mathbb S^1\to\mathbb S^1$ be continuous and assume
\[
        \|u-\id\|_{L^\infty(\mathbb S^1)}\le \varepsilon_0 .
\]
Then there exist $\lambda\in\mathbb R$ and $b\in\mathbb C$ such that, with
\[
        m_b(z)\coloneqq \frac{z+b}{1+\overline b z},
        \qquad
        \widetilde u(z)\coloneqq e^{-\i\lambda}u(m_b(z)),
\]
the map $\widetilde u$ can be written uniquely as
\[
        \widetilde u(x)=\cos(a(x))x+\sin(a(x))x^\perp,
        \qquad
        a\colon \mathbb S^1\to\left[-\frac\pi4,\frac\pi4\right],
\]
and
\[
        \int_0^{2\pi}a(e^{\i\theta})\dd\theta=0,
        \qquad
        \int_0^{2\pi}a(e^{\i\theta})e^{\i\theta}\dd\theta=0.
\]
Moreover, it holds
\begin{equation}\label{eq:lambdabestmodes}
        |\lambda|+|b|
        \le
        C\|u-\id\|_{L^\infty(\mathbb S^1)},
\end{equation}
Thus, we have
\[
        \|\widetilde u-\id\|_{L^\infty(\mathbb S^1)}
        \le
        C\|u-\id\|_{L^\infty(\mathbb S^1)},
\]
If additionally $u-\id\in C^\gamma(\mathbb S^1)$ for some $0<\gamma\leq 1$, then
\[
        \|\widetilde u-\id\|_{C^\gamma(\mathbb S^1)}
        \le
        C
        \|u-\id\|_{C^\gamma(\mathbb S^1)}.
\]
and we have whenever the right-hand side is finite
\[
        [\widetilde u-\id]_{W^{s,q}(\S^1)}
        \le
        C_{s,q}
        [u-\id]_{W^{s,q}(\S^1)} + \|u-\id\|_{L^\infty}
\]
\end{lemma}
\begin{proof}
Since $\|u-\id\|_{L^\infty} \leq \eps_0$, there is a unique phase
$a_0\colon \mathbb S^1\to[-\pi/4,\pi/4]$ such that
\[
        u(x)=\cos(a_0(x))x+\sin(a_0(x))x^\perp .
\]
Equivalently, writing $x=e^{\i\theta}$, we have
\[
        u(e^{\i\theta})=e^{\i(\theta+a_0(e^{\i\theta}))}.
\]

For $|b|<1$, write
\[
        m_b(e^{\i\theta}) = \frac{e^{\i \theta} +b}{1+\overline b e^{\i \theta}} =
        e^{\i \theta} \frac{1 +b e^{-\i \theta}}{1+\overline {b e^{-\i \theta}} }
        \eqqcolon e^{\i(\theta+\phi_b(\theta))}.
\]
We have, for $b = b_1 + \i b_2$
\[
        \phi_b(\theta)
        \coloneqq 
        2\arctan
        \left(
        \frac{b_2 \cos (\theta) - b_1 \sin(\theta)}
        {1+b_1\cos(\theta) + b_2 \sin(\theta)}
        \right)=
        2\arctan
        \left(
        \frac{\operatorname{Im}(be^{-\i\theta})}
        {1+\operatorname{Re}(be^{-\i\theta})}
        \right) = 2\arctan
        \left(
        \frac{\operatorname{Im}(1+be^{-\i\theta})}
        {\operatorname{Re}(1+be^{-\i\theta})}
        \right)
\]
Indeed, $e^{\i \arctan{\frac{\operatorname{Im} g}{{\operatorname{Re} g}}}}=\frac{g}{|g|}$, as long as $\operatorname{Re}(g) > 0$ (Euler's formula), which is ensured since $|b| <1$. Thus it holds $e^{\i 2\arctan{\frac{\operatorname{Im} g}{{\operatorname{Re} g}}}}=\frac{g^2}{|g|^2} = \frac{g}{\bar{g} |g|^2}$.
A direct expansion gives uniformly in $\theta \in [0,2\pi]$, as $|b|\to 0$,
\[
        \phi_b(\theta)
        =
        2\operatorname{Im}(b)\cos\theta
        -
        2\operatorname{Re}(b)\sin\theta
        +
        O(|b|^2)
\]
We denote
\[
 R_b(\theta) \coloneqq  \phi_b(\theta) - 2\operatorname{Im}(b)\cos\theta+
        2\operatorname{Re}(b)\sin\theta.
\]
Now let $\lambda\in\mathbb R$ and $|b|<1/2$. Using the formula above,
\[
\begin{split}
e^{-\i\lambda}u(m_b(e^{\i\theta}))
&=
e^{-\i\lambda}
e^{\i(\theta+\phi_b(\theta))}
e^{\i a_0(m_b(e^{\i\theta}))} =
e^{\i\theta}
e^{\i\left[
-\lambda+\phi_b(\theta)+a_0(m_b(e^{\i\theta}))
\right]}.
\end{split}
\]
Therefore define
\begin{equation}\label{eq:alambdab}
        a_{\lambda,b}(e^{\i\theta})
        \coloneqq 
        -\lambda+\phi_b(\theta)+a_0(m_b(e^{\i\theta})).
\end{equation}
Then explicitly
\[
        e^{-\i\lambda}u(m_b(e^{\i\theta}))
        =
        e^{\i\theta}e^{\i a_{\lambda,b}(e^{\i\theta})}.
\]
Equivalently, it holds
\[
        e^{-\i\lambda}u(m_b(x))
        =
        \cos(a_{\lambda,b}(x))x
        +
        \sin(a_{\lambda,b}(x))x^\perp,
        \qquad x=e^{\i\theta}.
\]
We now choose $\lambda$ and $b$ so that the zero and first Fourier modes of
$a_{\lambda,b}$ vanish. This means solving the two equations
\begin{equation}\label{eq:zeromodewhatwewant}
        \int_0^{2\pi} a_{\lambda,b}(e^{\i\theta})\dd\theta=0, \quad \text{and} \quad
        \int_0^{2\pi} a_{\lambda,b}(e^{\i\theta})e^{\i\theta}\dd\theta=0.
\end{equation}
Observe that
\[
        \int_0^{2\pi}
        (2b_2\cos\theta-2b_1\sin\theta)
        \dd\theta
        =
        0.
\]
This implies that
\begin{equation}\label{eq:phibrbmode0}
 \int_{0}^{2\pi} \phi_b(\theta) \dd\theta = \int_{0}^{2\pi} R_b(\theta)  \dd\theta.
\end{equation}
We have
\[
\begin{split}
\int_0^{2\pi}
(2b_2\cos\theta-2b_1\sin\theta)e^{\i\theta}
\dd\theta
&=
2b_2\int_0^{2\pi}\cos\theta\,e^{\i\theta}\dd\theta
-
2b_1\int_0^{2\pi}\sin\theta\,e^{\i\theta}\dd\theta
\\
&=
2\pi b_2-2\pi \i b_1
\\
&=
-2\pi \i b.
\end{split}
\]
This implies
\begin{equation}\label{eq:phibrbmode1}
 \int_{0}^{2\pi} \phi_b(\theta) e^{\i \theta} \dd\theta = \int_{0}^{2\pi} R_b(\theta) \dd\theta - 2\pi \i b.
\end{equation}
Thus \eqref{eq:zeromodewhatwewant} combined with  \eqref{eq:phibrbmode0}  and \eqref{eq:phibrbmode1}. Using the expansion \eqref{eq:alambdab}, these equations become
\begin{equation}\label{eq:phibrbmode0v2}
        -2\pi\lambda
        +
        \int_0^{2\pi}a_0(m_b(e^{\i\theta}))\dd\theta
        +
        \int_0^{2\pi}\phi_b(\theta)\dd\theta
        =
        0,
\end{equation}
and
\begin{equation}\label{eq:phibrbmode1v2}
        -2\pi \i b
        +
        \int_0^{2\pi}a_0(m_b(e^{\i\theta}))e^{\i\theta}\dd\theta
        +
        \int_0^{2\pi}R_b(\theta)e^{\i\theta}\dd\theta
        =
        0.
\end{equation}
We solve the second equation first. Define, for $|b|\le 4\varepsilon$,
\[
        T(b)
        \coloneqq 
        -\frac{i}{2\pi}
        \left[
        \int_0^{2\pi}a_0(m_b(e^{\i\theta}))e^{\i\theta}\dd\theta
        +
        \int_0^{2\pi}R_b(\theta)e^{\i\theta}\dd\theta
        \right].
\]
Since $u$ is continuous, so is $a_0$. Thus $b \mapsto T(b)$ is continuous.
Moreover, we have
\[
        |T(b)|
        \le
        \|a_0\|_{L^\infty}
        +
        C|b|^2
        \le
        \|u-\id\|_{L^\infty(\mathbb S^1)}+C|b|^2.
\]
Hence, if $|b|\le \|u-\id\|_{L^\infty(\mathbb S^1)}$ and $\varepsilon_0$ is sufficiently small,
\[
        |T(b)|
        \le
        4\|u-\id\|_{L^\infty(\mathbb S^1)}.
\]
Thus $T$ maps the closed disk $\{|b|\le4 \|u-\id\|_{L^\infty(\mathbb S^1)}\}$ into itself. By
Brouwer's fixed point theorem, there exists $b$ with
\[
        |b|\le4\|u-\id\|_{L^\infty(\mathbb S^1)}
\]
such that
\[
        b=T(b).
\]
For this choice of $b$, \eqref{eq:phibrbmode1v2} holds --- and observe that this holds for any $\lambda \in \R$. Now, with $b$ fixed, we can choose $\lambda$ so that \eqref{eq:phibrbmode0v2} holds, simply define
\[
        \lambda
        \coloneqq 
        \frac1{2\pi}
        \left[
        \int_0^{2\pi}a_0(m_b(e^{\i\theta}))\dd\theta
        +
        \int_0^{2\pi}R_b(\theta)\dd\theta
        \right].
\]
Also,
\[
        |\lambda|
        \le
        \|u-\id\|_{L^\infty}+C|b|^2
        \aleq
        \|u-\id\|_{L^\infty}.
\]
Thus we have \eqref{eq:lambdabestmodes}. We set
\[
        a\coloneqq a_{\lambda,b},
        \qquad
        \widetilde u\coloneqq e^{-\i\lambda}u\circ m_b.
\]
Then it holds
\[
        \int_0^{2\pi} a(e^{\i\theta})\dd\theta=0,
        \qquad
        \int_0^{2\pi} a(e^{\i\theta})e^{\i\theta}\dd\theta=0.
\]
Furthermore,
\[
\begin{split}
        \|a\|_{L^\infty}
        &\le
        |\lambda|+\|\phi_b\|_{L^\infty}+\|a_0\|_{L^\infty} \le
        C\varepsilon.
\end{split}
\]
In particular, if $\varepsilon_0$ sufficiently small (uniform), we have
\[
        a\colon \mathbb S^1\to[-\pi/4,\pi/4].
\]
Thus
\[
        \widetilde u(x)\coloneqq \cos(a(x))x+\sin(a(x))x^\perp
\]
and $a\colon \S^1 \to [-\pi/4,\pi/4]$ is the unique map that satisfies this. Finally,
\[
        \|\widetilde u-\id\|_{L^\infty}
        \aleq \|a\|_{L^\infty}\aleq \|u-\id\|_{L^\infty}.
\]
If in addition $u-\id\in C^\gamma(\mathbb S^1)$, then
\[
        \|\widetilde u-\id\|_{C^\gamma}
        \aleq \|u-\id\|_{C^\gamma}.
\]
Indeed, it holds
\begin{equation}\label{eq:decompo_utilde_id}
\begin{split}
\widetilde u-\id
&=
e^{-\i\lambda}u\circ m_b-\id
\\
&=
e^{-\i\lambda}(u-\id)\circ m_b
+
(e^{-\i\lambda}m_b-\id).
\end{split}
\end{equation}
Since $|b|\le C\|u-\id\|_{L^\infty}$ is small, $m_b$ is uniformly bi-Lipschitz,
and hence
\[
        \|(u-\id)\circ m_b\|_{C^\gamma}
        \le
        C_\gamma\|u-\id\|_{C^\gamma}.
\]
Moreover,
\[
        \|e^{-\i\lambda}m_b-\id\|_{C^\gamma}
        \le
        C_\gamma(|\lambda|+|b|)
        \le
        C_\gamma\|u-\id\|_{L^\infty}
        \le
        C_\gamma\|u-\id\|_{C^\gamma}.
\]
Therefore
\[
        \|\widetilde u-\id\|_{C^\gamma}
        \le
        C_\gamma\|u-\id\|_{C^\gamma}.
\]
By \Cref{la:mobius-invariance}, $\widetilde{u}$ is still a critical point and we can conclude. Concerning the estimate in $W^{s,q}$, we proceed similarly, using \eqref{eq:decompo_utilde_id} and the fact that $m_b$ is uniformly bi-Lipschitz for $\eps$ small enough by \eqref{eq:lambdabestmodes}.
\end{proof}

We will also make use the following elementary inequality:
\begin{lemma}\label{la:one-dimensional-average-beta}
For all $A,B\in\mathbb R$ and all $\beta\in[0,1]$
\begin{equation}
\label{eq:one-dimensional-average-beta}
\int_0^1 |A+\tau B|^\beta\dd\tau
\ge
\frac{2^{-\beta}}{1+\beta}|A|^\beta .
\end{equation}
and for a uniform constant $c > 0$
\begin{equation}
\label{eq:one-dimensional-average-betav2}
\int_0^1 (1-\tau) |A+\tau B|^\beta\dd\tau
\ge c|A|^\beta .
\end{equation}
\end{lemma}
\begin{proof}
The case $\beta=0$ is immediate, so assume $0<\beta\le1$.

If $A=0$, then the right-hand side is zero, so there is nothing to prove, so assume $A\neq0$, and by homogeneity we may assume $A =1$.

That is, it suffices to prove
\[
 \int_0^1 |1-\tau c|^\beta\dd\tau
 \ge
 \frac{2^{-\beta}}{1+\beta}
 \qquad
 \forall c\in\mathbb R.
\]

If $c\leq 0$, then $|1-\tau c|\ge1$, hence
\[
 \int_0^1 |1-\tau c|^\beta\dd\tau\ge1
 \ge
 \frac{2^{-\beta}}{1+\beta}.
\]
If $c\in(0,1]$, then
\[
 \int_0^1 |1-\tau c|^\beta\dd\tau = \frac{1}{c} \int_{1-c}^1 \tilde{\tau}^\beta\dd\tau
 =
 \frac{1}{1+\beta} \frac{1-(1-c)^{1+\beta}}{c}.
\]
Since $1+\beta \geq 1$ and $1-c \in (0,1]$ we have
\[
 1-(1-c)^{1+\beta} \ge c
\]
Therefore, for any $c \in (0,1]$ we have shown
\[
 \int_0^1 |1-\tau c|^\beta\dd\tau
 \ge
 \frac1{1+\beta}
 \ge
 \frac{2^{-\beta}}{1+\beta}.
\]
Finally, suppose $c>1$. In this case
\[
\begin{split}
\int_0^1 |1-\tau c|^\beta\dd\tau
&=\frac{1}{c} \int_{1-c}^1 |\tilde{\tau}|^\beta\dd\tilde{\tau}\\
&=\frac{1}{c} \brac{\int_{0}^1 \tilde{\tau}^\beta\dd\tilde{\tau}+ \int_{0}^{c-1} \tilde{\tau}^\beta\dd\tilde{\tau}}\\
&=\frac{1}{\beta+1} \frac{1+ \brac{c-1}^{1+\beta}}{c}.
\end{split}
\]
By convexity, we have
\[
 \frac{1+ \brac{c-1}^{1+\beta}}{2} \geq \brac{\frac{1+ \brac{c-1}}{2}}^{1+\beta}= 2^{-\beta} c^{1+\beta}.
\]
Thus, since $c >1$, we obtain 
\[
\begin{split}
\int_0^1 |1-\tau c|^\beta\,d\tau
&\geq \frac{2^{-\beta}}{\beta+1} \frac{\brac{c}^{1+\beta}}{c} =\frac{2^{-\beta}}{\beta+1} \brac{c}^{\beta} \geq \frac{2^{-\beta}}{\beta+1}.
\end{split}
\]
This proves \eqref{eq:one-dimensional-average-beta}.
\end{proof}

\begin{lemma}\label{la:homotopy}
Set
$$
H_\tau(x)
\coloneqq 
\cos(\tau a(x))x+\sin(\tau a(x))x^\perp .
$$
Then it holds
$$
H_\tau\colon  \S^1 \to \S^1, \quad H_0(x)=x,
$$
We also have
\begin{equation}\label{eq:partialtauHtau}
\partial_\tau H_\tau(x)=a(x)H_\tau^\perp(x),
\end{equation}
and for any $\sigma \in [0,1]$ if $\|a\|_{L^\infty} \aleq 1$ then
\begin{equation}\label{eq:xmHtau}
 [x-H_\tau]_{C^\sigma} \aleq [a]_{C^\sigma} +  \|a\|_{L^\infty}.
\end{equation}
We also have
\[
\left.
\frac{\dd}{\dd\varepsilon}
\right|_{\varepsilon=0}
\frac{H_\tau+\varepsilon H_\tau^\perp\eta}
{|H_\tau+\varepsilon H_\tau^\perp\eta|}
=
\eta H_\tau^\perp,
\]
and
\[
\left.
\frac{\dd^2}{\dd\varepsilon^2}
\right|_{\varepsilon=0}
\frac{H_\tau+\varepsilon H_\tau^\perp\eta}
{|H_\tau+\varepsilon H_\tau^\perp\eta|}
=
-|\eta|^2H_\tau.
\]
\end{lemma}
\begin{proof}
We observe the following which is \eqref{eq:partialtauHtau}
\[
\partial_\tau H_\tau(x)
=
-a(x)\sin(\tau a(x))x+a(x)\cos(\tau a(x))x^\perp
=
a(x)H_\tau^\perp(x),
\]
In particular, from the fundamental theorem,
\[
\begin{split}
 |x-y-H_\tau(x)-H_\tau(y)| \leq& \int_0^\tau |a(x) H_t^\perp(x)-a(y) H_t^\perp(y)| \dd t\\
 \leq& |a(x)-a(y)| + \|a\|_{L^\infty} \max_{t \in [0,1]} |H_t(x)- H_t(y)|\\
 \aleq& |a(x)-a(y)| + \|a\|_{L^\infty} |x-y| +\|a\|_{L^\infty} |a(x)-a(y)|.
 \end{split}
\]
Consequently, for any $\sigma \in [0,1]$ we have \eqref{eq:xmHtau}. Let $\eta\colon \S^1 \to \R$. Observe that
\begin{equation}\label{eq:Htauvar}
\begin{split}
\frac{H_\tau+\varepsilon H_\tau^\perp\eta}{|H_\tau+\varepsilon H_\tau^\perp\eta|}
& =\frac{1}{\sqrt{1+\eps^2 |\eta|^2}} H_\tau  + \frac{\eps \eta }{\sqrt{1+\eps^2 |\eta|^2}} H_\tau^\perp\\
& =H_\tau-\frac{\eps^2 |\eta|^2}{\brac{1+\sqrt{1+\eps^2 |\eta|^2}} \sqrt{1+\eps^2 |\eta|^2}} H_\tau  + \frac{\eps \eta }{\sqrt{1+\eps^2 |\eta|^2}} H_\tau^\perp\\
\end{split}
\end{equation}
Indeed, we have
\[
 |H_\tau + \eps H_\tau^\perp \eta|^2 =|H_\tau|^2 + \eps^2 |H_\tau|^2\, |\eta|^2 = 1+\eps^2 |\eta|^2.
\]
Thus it holds
\[
 \frac{H_\tau+\varepsilon H_\tau^\perp\eta}{|H_\tau+\varepsilon H_\tau^\perp\eta|} = \frac{1}{\sqrt{1+\eps^2 |\eta|^2}} H_\tau  + \frac{\eps \eta }{\sqrt{1+\eps^2 |\eta|^2}} H_\tau^\perp,
\]
which is \eqref{eq:Htauvar}.

Thus, we end up with
\[
\left.
\frac{\dd}{\dd\varepsilon}
\right|_{\varepsilon=0}
\frac{H_\tau+\varepsilon H_\tau^\perp\eta}
{|H_\tau+\varepsilon H_\tau^\perp\eta|}
=
\eta H_\tau^\perp,
\]
and
\[
\left.
\frac{\dd^2}{\dd\varepsilon^2}
\right|_{\varepsilon=0}
\frac{H_\tau+\varepsilon H_\tau^\perp\eta}
{|H_\tau+\varepsilon H_\tau^\perp\eta|}
=
-|\eta|^2H_\tau.
\]
\end{proof}

As a consequence of \Cref{la:one-dimensional-average-beta} we can estimate our homotopy term when $\beta = p-2 \geq 0$.

\begin{lemma}\label{la:eq:averaged-chord-power}
There exist universal constants $\bar{\rho}>0$ and $C<\infty$ with the following property.
Let $\beta \in [0,1]$, $\rho \in (0,\bar{\rho}]$ and $a\colon \mathbb S^1\to\mathbb R$ with $\|a\|_{L^\infty(\mathbb S^1)}\le \rho$. For $0\le \tau\le1$ we set
\begin{equation}\label{eq:Htau}
H_\tau(z)
\coloneqq 
\cos(\tau a(z))z+\sin(\tau a(z))z^\perp,
\qquad \forall z\in\mathbb S^1.
\end{equation}
Then for all $x,y\in\mathbb S^1$ with $0<|x-y|\le \rho$, one has
 \begin{equation}
\label{eq:averaged-chord-power}
|x-y|^{\beta}-\int_0^1 |H_\tau(x)-H_\tau(y)|^\beta\dd\tau
\leq
C(\beta+\rho^2)\, |x-y|^{\beta}.
\end{equation}
Moreover we have for any $\beta \in [0,1]$ for a uniform constant $c > 0$
 \begin{equation}
\label{eq:averaged-chord-powerv2}
\int_0^1 (1-\tau) |H_\tau(x)-H_\tau(y)|^\beta\dd\tau \geq c |x-y|^\beta.
\end{equation}

\end{lemma}
\begin{proof}
The case $\beta=0$ is immediate, so assume $0<\beta\le 1$. Fix $\bar{\rho} \in (0,1)$ so that the following holds for a uniform constant $\Gamma > 0$: for $s \in [0,100\bar{\rho}]$,
\begin{equation}\label{eq:sintaylor}
\begin{split}
 \sin(s) & = s - s^3\brac{\frac{1}{3!} - s\sum_{i=3}^\infty (-1)^i \frac{s^{2i-5}}{(2i+1)!}}\\
 & \geq s - s^3\brac{\frac{1}{3!} + se^s}\\
 & \geq s(1- \Gamma s^2).
 \end{split}
\end{equation}
Up to reducing $\bar{\rho}$, we may assume that $1-\Gamma \bar{\rho}^2 > \frac{1}{2}$.

Let now $\rho \in [0,\bar{\rho}]$, $0<|x-y|\le \rho\le \bar\rho$. Then there is a unique
$\alpha=\alpha(x,y)\in(-\pi/2,\pi/2)$ such that
\[
 y=\cos(\alpha)\, x+\sin(\alpha)\, x^\perp .
\]
Then it holds
\[
 y^\perp=-\sin(\alpha)\,x+\cos(\alpha)\,x^\perp
\]
and we have
\[
 |x-y|
 =
 2\left|\sin\frac{\alpha}{2}\right|.
\]
In particular,
\begin{equation}\label{eq:distance_xy}
 |x-y|\le |\alpha|=
 2\arcsin\frac{|x-y|}{2}
 \le
 2\arcsin\frac{\rho}{2} \leq 2\rho
\end{equation}
We compute $H_\tau(y)$ in the basis $\{x,x^\perp\}$. By \eqref{eq:Htau}, we get
\[
\begin{split}
H_\tau(y)
&=
\cos(\tau a(y))
\left(\cos(\alpha)\,x+\sin(\alpha)\,x^\perp\right)
+
\sin(\tau a(y))
\left(-\sin\alpha\,x+\cos\alpha\,x^\perp\right)
\\
&=
\cos(\alpha+\tau a(y))\, x
+
\sin(\alpha+\tau a(y))\, x^\perp .
\end{split}
\]
Similarly, it holds
\[
H_\tau(x)
=
\cos(\tau a(x))x+\sin(\tau a(x))x^\perp.
\]
Therefore, we obtain 
\[
\begin{split}
 |H_\tau(x)-H_\tau(y)|^2 =& \abs{\cos(\alpha + \tau a(y))-\cos(\tau a(x))}^2 + \abs{\sin(\alpha + \tau a(y))-\sin(\tau a(x))}^2\\
 =& 2 - 2\cos(\alpha + \tau a(y)) \cos(\tau a(x))- 2\sin(\alpha + \tau a(y)) \sin(\tau a(x))\\
 =& 2 - 2\cos(\alpha - \tau \brac{a(x)-a(y)})\\
 =& 4 \sin^2\brac{\frac{\alpha - \tau \brac{a(x)-a(y)}}{2}}.
 \end{split}
\]
Thus, we end up with
\[
 |H_\tau(x)-H_\tau(y)|
 =
 2\left|
 \sin\left(
 \frac{\alpha-\tau(a(x)-a(y))}{2}
 \right)
 \right|.
\]
Moreover, we have
\[
 |a(y)-a(x)|
 \le
 |a(y)|+|a(x)|
 \le
 2\|a\|_{L^\infty}
 \le 2\rho.
\]
Thus, it holds, using \eqref{eq:distance_xy}
\[
 \left|\alpha-\tau(a(x)-a(y))\right|
 \le
 4\rho
 \qquad
 \forall \tau\in[0,1].
\]
From \eqref{eq:sintaylor} we find that for all $t \in [-10\rho,10\rho]$,
\[
 2\left|\sin\frac t2\right|
 \ge
 (1-\Gamma \rho^2)|t|.
\]
That is, we have, for $\bar{\rho}$ suitably small and $\rho < \bar{\rho}$, $|x-y| \leq \rho$
\begin{equation}\label{eq:distance_Htau}
 |H_\tau(x)-H_\tau(y)|
 \ge
 (1-\Gamma \rho^2)
 \left|\alpha-\tau(a(x)-a(y))\right|.
\end{equation}
Raising this inequality to the power $\beta\in(0,1]$, (observe that $(1-\Gamma \rho^2) \in (0,1)$ so that $(1-\Gamma \rho^2)^\beta \geq (1-\Gamma \rho^2)$),
\[
 |H_\tau(x)-H_\tau(y)|^\beta
 \ge
 (1-\Gamma \rho^2)^\beta
 \left|\alpha-\tau(a(x)-a(y))\right|^\beta \geq (1-\Gamma \rho^2)
 \left|\alpha-\tau(a(x)-a(y))\right|^\beta.
\]
Thus, we must have
\begin{equation}
\label{eq:chord-reduced-to-line}
\int_0^1 |H_\tau(x)-H_\tau(y)|^\beta\dd\tau
\ge
(1-\Gamma \rho^2)
\int_0^1
\left|\alpha-\tau(a(x)-a(y))\right|^\beta
\dd\tau .
\end{equation}
and
\begin{equation}
\label{eq:chord-reduced-to-linev2}
\int_0^1 (1-\tau) |H_\tau(x)-H_\tau(y)|^\beta\dd\tau
\ge
(1-\Gamma \rho^2)
\int_0^1
(1-\tau) \left|\alpha-\tau(a(x)-a(y))\right|^\beta
\dd\tau .
\end{equation}
Combining \eqref{eq:chord-reduced-to-line}
with  \Cref{la:one-dimensional-average-beta}, \eqref{eq:one-dimensional-average-beta}, we find with the choices $A=\alpha$, and $B=a(y)-a(x)$, that
\[
\int_0^1 |H_\tau(x)-H_\tau(y)|^\beta\dd\tau
\ge
(1-\Gamma \rho^2)
\frac{2^{-\beta}}{1+\beta}
|\alpha|^\beta .
\]
There exists a constant $\Upsilon > 0$ such that
\[
 \frac{2^{-\beta}}{1+\beta}
 \ge
 1-\Upsilon \beta
 \qquad
 \text{for }0\le\beta\le1,
\]
and we have
\[
\brac{1-\Gamma \rho^2} \brac{1-\Upsilon \beta} \geq 1-\Gamma \rho^2 -\Upsilon \beta.
\]
Since moreover $|\alpha|\ge |x-y|$, we obtain for $C \coloneqq \Gamma + \Upsilon$,
\[
\int_0^1 |H_\tau(x)-H_\tau(y)|^\beta\dd\tau
\ge
\left(1-C(\beta+\rho^2)\right)|x-y|^\beta .
\]
This proves \eqref{eq:averaged-chord-power}. Combining now \eqref{eq:one-dimensional-average-betav2}, \eqref{eq:distance_xy} and \eqref{eq:chord-reduced-to-linev2}, we obtain \eqref{eq:averaged-chord-powerv2}.
\end{proof}

For the case $p < 2$ we need an alternative to \Cref{la:eq:averaged-chord-power} of $\beta \coloneqq 2-p>0$. 
\begin{lemma}
\label{la:p-less-2-lip-chord-term}
There exists $C>0$ such that the following holds. Let $\beta \geq 0$ and $a\colon \mathbb S^1\to\mathbb R$ satisfy
\begin{equation}\label{hyp:a_Lip}
        \frac{a(x)-a(y)}{|x-y|} \le \Lambda.
\end{equation}
For $0\le\tau\le1$, set
\[
        H_\tau(x)
        \coloneqq 
        \cos(\tau a(x))x+\sin(\tau a(x))x^\perp .
\]
Then it holds
\[
|x-y|^{-\beta}-\int_0^1 |H_\tau(x)-H_\tau(y)|^{-\beta}\dd\tau
\leq C\, \beta\log(2+\Lambda)\, |x-y|^{-\beta}.
\]
Moreover we have
\[
\int_0^1 (1-\tau) |H_\tau(x)-H_\tau(y)|^{-\beta}\dd\tau
\geq \brac{\frac{1}{2} - C|\beta| \log(2+\Lambda) } |x-y|^{-\beta}.
\]
which is a good estimate for $\beta \approx 0$. Otherwise, it holds
\[
\int_0^1 (1-\tau) |H_\tau(x)-H_\tau(y)|^{-\beta}\dd\tau \geq \frac{1}{4} (1+ C \Lambda)^{-\beta} |x-y|^{-\beta}.
\]
\end{lemma}
\begin{proof}
For every $x,y\in\mathbb S^1$ and every
$\tau\in[0,1]$, we have
\[
\begin{split}
H_\tau(x)-H_\tau(y)
&=
\cos(\tau a(x))(x-y)
+
\sin(\tau a(x))(x^\perp-y^\perp)
\\
&\quad
+
\left(\cos(\tau a(x))-\cos(\tau a(y))\right)y
\\
&\quad
+
\left(\sin(\tau a(x))-\sin(\tau a(y))\right)y^\perp .
\end{split}
\]
The first line has length exactly $|x-y|$. Therefore
\[
\begin{split}
|H_\tau(x)-H_\tau(y)|
&\le
|x-y|
+
\left|\cos(\tau a(x))-\cos(\tau a(y))\right|
+
\left|\sin(\tau a(x))-\sin(\tau a(y))\right|
\\
&\le
|x-y|+C\tau |a(x)-a(y)|.
\end{split}
\]
By \eqref{hyp:a_Lip}, we obtain 
\[
        |H_\tau(x)-H_\tau(y)|
        \le
        |x-y|(1+C\tau\Lambda).
\]
Because $\beta>0$, this implies
\[
        |H_\tau(x)-H_\tau(y)|^{-\beta}
        \ge
        |x-y|^{-\beta}(1+C\tau\Lambda)^{-\beta}.
\]
Therefore, we have
\[
\begin{split}
&|x-y|^{-\beta}
-
\int_0^1 |H_\tau(x)-H_\tau(y)|^{-\beta}\dd\tau
 \le
|x-y|^{-\beta}
\left[
1-
\int_0^1 (1+C\tau\Lambda)^{-\beta}\dd\tau
\right].
\end{split}
\]
Observe that for $B>0$
\[
        1-(1+\tau B)^{-\beta}
        =
        1-e^{-\beta\log(1+\tau B)}
        \le
        \beta\, \log(1+\tau B)
        \le
        \beta\, \log(1+B).
\]
Integrating in $\tau$ proves
\[
        1-
        \int_0^1 (1+C\tau\Lambda)^{-\beta}\dd\tau
        \le
        C\, \beta\, \log(2+\Lambda).
\]
Similarly, we have
\[
\begin{split}
&\int_0^1 (1-\tau) |H_\tau(x)-H_\tau(y)|^{-\beta}\dd\tau\\
\geq& \left( \int_0^1 (1-\tau) (1+C\tau \Lambda)^{-\beta} \dd\tau\right) \, |x-y|^{-\beta}\\
\geq& \left( \int_0^1 (1-\tau) \brac{1+((1+C\tau \Lambda)^{-\beta} -1)}\dd\tau\right) |x-y|^{-\beta}\\
\geq& \brac{\frac{1}{2} - C|\beta| \log(2+\Lambda) } |x-y|^{-\beta}.
\end{split}
\]
This is a good estimate for $\beta \approx 0$. Otherwise, we have 
\[
\begin{split}
&\int_0^1 (1-\tau) |H_\tau(x)-H_\tau(y)|^{-\beta}\dd\tau\\
&\geq \left(\int_0^1 (1-\tau) (1+C\tau \Lambda)^{-\beta} \dd\tau\right) |x-y|^{-\beta}\\
&\geq \left( \int_0^{\frac{1}{2}} (1-\tau) (1+C\tau \Lambda)^{-\beta} \dd\tau\right) |x-y|^{-\beta}\\
& \geq \frac{1}{4} (1+ C \Lambda)^{-\beta} |x-y|^{-\beta}.
\end{split}
\]
We can conclude.
\end{proof}

 \begin{proof}[Proof of \Cref{th:localminisidv2}]
Set
\[
        p\coloneqq \frac1s,
        \qquad
        p_-\coloneqq \frac1{s_+},
        \qquad
        p_+\coloneqq \frac1{s_-}.
\]
Then we have
\[
        1<p_-\leq p\leq p_+<3.
\]

By \Cref{la:local-normalization-phase}, we can replace
\[
        \widetilde u:=e^{-i\lambda}u\circ m_b
\]
for some $\lambda\in\R$ and $|b|<1$, and by \Cref{la:lipschitz-phase} we have a unique phase $a\colon \S^1\to\left[-\frac\pi4,\frac\pi4\right]$ so that
\[
        \widetilde u(x)
        =
        \cos(a(x))x+\sin(a(x))x^\perp,
\]
with
\begin{equation}\label{eq:aftmodeszero}
        \int_0^{2\pi}a(e^{\i\theta})\dd\theta=0,
        \qquad
        \int_0^{2\pi}a(e^{\i\theta})e^{\i\theta}\dd\theta=0.
\end{equation}
Moreover, it holds
\[
        \|a\|_{L^\infty}
        \aleq
        \|\widetilde u-\id\|_{L^\infty}
        \aleq
        \|u-\id\|_{L^\infty}
        \leq \eps .
\]
If $p<2$, then, since $m_b$ is uniformly bi-Lipschitz for $\eps$ small,
\[
        [a]_{\lip}
        \aleq
        1+[\widetilde u]_{\lip}
        \aleq
        1+[u]_{\lip}
        \aleq
        1+\frac{\Lambda}{|p-2|}.
\]
By \Cref{la:mobius-invariance}, $\widetilde u$ is still a critical point. Set as in \Cref{la:homotopy}
\[
        H_\tau(x)
        \coloneqq 
        \cos(\tau a(x))x+\sin(\tau a(x))x^\perp,
        \qquad
        0\leq \tau\leq 1.
\]
Then we have $H_0=\id$, $H_1=\widetilde u$ and
\[
        \partial_\tau H_\tau(x)=a(x)H_\tau^\perp(x),
        \qquad
        \partial_{\tau\tau}H_\tau(x)=-a(x)^2H_\tau(x).
\]
Since both $\id$ and $\widetilde u$ are critical, we have
\[
        \left.
        \frac{\dd}{\dd\tau}
        \right|_{\tau=0}
        \frac1pE_s(H_\tau)
        =
        \left.
        \frac{\dd}{\dd\tau}
        \right|_{\tau=1}
        \frac1pE_s(H_\tau)
        =
        0.
\]
Thus, by the fundamental theorem of calculus,
\begin{equation}\label{eq:unweightedendpointcriticality}
\begin{split}
0
&=
\int_0^1
        \frac{\dd^2}{\dd\tau^2}
        \frac1pE_s(H_\tau)\dd\tau .
\end{split}
\end{equation}
By \eqref{eq:identity1} and \eqref{eq:identity2},
\[
\begin{split}
        \frac{\dd^2}{\dd\tau^2}
        \frac1pE_s(H_\tau)
        &=
        \iint_{\S^1\times\S^1}
        \frac{
        |H_\tau(x)-H_\tau(y)|^{p-2}
        \left[
        p-1-\frac p4|H_\tau(x)-H_\tau(y)|^2
        \right]
        }{|x-y|^2}
        |a(x)-a(y)|^2\dx\dy .
\end{split}
\]
Consequently, we have
\begin{equation}\label{eq:criticalidentitylocalmin}
\begin{split}
0
&=
\int_0^1
        \iint_{\S^1\times\S^1}
        \frac{
        |H_\tau(x)-H_\tau(y)|^{p-2}
        \left[
        p-1-\frac p4|H_\tau(x)-H_\tau(y)|^2
        \right]
        }{|x-y|^2}
        |a(x)-a(y)|^2\dx\dy\dd\tau .
\end{split}
\end{equation}

We now use the positivity of the second variation at $\id$, in the form of
\Cref{co:la:truncated-coercivity-one-kernel}. We first choose the constant
$A_\Lambda>0$ below. Then we apply \Cref{co:la:truncated-coercivity-one-kernel}
with $A=\frac12A_\Lambda$, and fix $\rho_0>0$ and $\bar c>0$ such that for any
$\rho<\rho_0$,
\begin{equation}\label{eq:Aweridcoercivity}
\begin{split}
&\frac12A_\Lambda
\iint_{\{|x-y|\leq\rho\}}
\frac{|a(x)-a(y)|^2}{|x-y|^{4-p}}\dx\dy
\\
&
+
\frac12
\iint_{\{|x-y|\geq\rho\}}
\frac{
|x-y|^{p-2}
\left[
p-1-\frac p4|x-y|^2
\right]
}{|x-y|^2}
|a(x)-a(y)|^2\dx\dy 
\geq
\bar c
[a]_{H^{\frac{3-p}{2}}(\S^1)}^2 .
\end{split}
\end{equation}

We split the integral in \eqref{eq:criticalidentitylocalmin} and first consider
\[
\begin{split}
        \int_0^1
        \iint_{\S^1\times\S^1} \chi_{\{|x-y| \leq \rho\} }
        \frac{
        |H_\tau(x)-H_\tau(y)|^{p-2}
        \left[
        p-1-\frac p4|H_\tau(x)-H_\tau(y)|^2
        \right]
        }{|x-y|^2} |a(x)-a(y)|^2\dx\dy\dd\tau .
\end{split}
\]
Observe that if $|x-y| \leq \rho$, then
\[
        |H_\tau(x)-H_\tau(y)|
        \aleq |x-y|+\|a\|_{L^\infty}
        \leq \rho+C\eps .
\]
Choosing $\rho>0$ small enough, and then $\eps\ll\rho$, we may assume, uniformly for
$p\in[p_-,p_+]$, that if $|x-y|\leq\rho$, then
\[
        p-1-\frac p4|H_\tau(x)-H_\tau(y)|^2
        \geq
        \frac{p-1}{2}.
\]
We obtain 
\begin{equation}\label{eq:second_var_a}
\begin{split}
        &\int_0^1
        \iint_{\S^1\times\S^1} \chi_{\{|x-y| \leq \rho\}}
        \frac{
        |H_\tau(x)-H_\tau(y)|^{p-2}
        \left[
        p-1-\frac p4|H_\tau(x)-H_\tau(y)|^2
        \right]
        }{|x-y|^2} |a(x)-a(y)|^2\dx\dy\dd\tau
        \\
        & \geq
        \frac{p-1}{2}
        \iint_{\S^1\times\S^1} \chi_{\{|x-y| \leq \rho\}}
        \frac{
        \int_0^1 |H_\tau(x)-H_\tau(y)|^{p-2}\dd\tau
        }{|x-y|^2}
        |a(x)-a(y)|^2\dx\dy .
\end{split}
\end{equation}
Fix $\gamma>0$ small enough so that
\begin{equation}\label{eq:gammachoiceasdad}
        \frac12-C\gamma\log\brac{C\brac{2+\frac{\Lambda}{\gamma}}}
        \geq
        \frac14 .
\end{equation}
We claim that
\begin{equation}\label{eq:localstableqrhoterm:goal}
\begin{split}
   &\frac{p-1}{2}
   \iint_{\S^1\times\S^1} \chi_{\{|x-y| \leq \rho\}}
   \frac{
   	\int_0^1 |H_\tau(x)-H_\tau(y)|^{p-2}\dd\tau
   }{|x-y|^2}
   |a(x)-a(y)|^2\dx\dy 
        \\
        & \geq
        A_\Lambda
        \iint_{\S^1\times\S^1} \chi_{ \{|x-y| \leq \rho\} }
        \frac{|a(x)-a(y)|^2}{|x-y|^{4-p}}\dx\dy .
\end{split}
\end{equation}
Here $A_\Lambda>0$ depends only on $p_-,p_+,\Lambda$, and indeed
\[
 A_\Lambda:= \begin{cases}
    \frac{p-1}{2} c,
    &\text{if $p>2$},\\
    \frac{p-1}{2},
    &\text{if $p=2$},\\
    \frac{p-1}{2}
    \brac{
        \frac12
        -
        C\gamma\log\brac{C\brac{2+\frac{\Lambda}{\gamma}}}
    },
    &\text{if $p\in [2-\gamma,2)$},\\
    \frac{p-1}{2}\frac14
    \brac{1+C\frac{\Lambda}{\gamma}}^{-1},
    &\text{if $p\in (p_-,2-\gamma)$}.
    \end{cases}
\]
where $c$ and $C$ denotes constants depending only on $s_-$ and $s_+$.

Indeed, \underline{if $p=2$}, there is nothing to show, and we get
\[
\begin{split}
   &\frac{p-1}{2}
   \iint_{\S^1\times\S^1} \chi_{\{|x-y| \leq \rho\}}
   \frac{
   	\int_0^1 |H_\tau(x)-H_\tau(y)|^{p-2}\dd\tau
   }{|x-y|^2}
   |a(x)-a(y)|^2\dx\dy 
   \\
        &=
        \frac{p-1}{2}
        \iint_{\S^1\times\S^1} \chi_{\{|x-y| \leq \rho\}}
        \frac{|a(x)-a(y)|^2}{|x-y|^{4-p}}\dx\dy .
\end{split}
\]

If \underline{$p>2$}, we apply \eqref{eq:averaged-chord-powerv2}. Since
\[
        \int_0^1 |H_\tau(x)-H_\tau(y)|^{p-2}\dd\tau
        \geq
        \int_0^1(1-\tau)|H_\tau(x)-H_\tau(y)|^{p-2}\dd\tau,
\]
we get from \eqref{eq:averaged-chord-powerv2}
\[
\begin{split}
   &\frac{p-1}{2}
   \iint_{\S^1\times\S^1} \chi_{\{|x-y| \leq \rho\}}
   \frac{
   	\int_0^1 |H_\tau(x)-H_\tau(y)|^{p-2}\dd\tau
   }{|x-y|^2}
   |a(x)-a(y)|^2\dx\dy 
   \\
    & \geq \frac{p-1}{2}c
        \iint_{\S^1\times\S^1} \chi_{\{|x-y| \leq \rho\} }
        \frac{|a(x)-a(y)|^2}{|x-y|^{4-p}}\dx\dy .
\end{split}
\]

If \underline{$p\in[2-\gamma,2)$}, we apply \Cref{la:p-less-2-lip-chord-term} with $\beta=2-p$.
Using $[a]_{\lip}\aleq 1+\frac{\Lambda}{2-p}$, and the choice of $\gamma$ in \eqref{eq:gammachoiceasdad}, we obtain
\[
\begin{split}
   &\frac{p-1}{2}
   \iint_{\S^1\times\S^1} \chi_{\{|x-y| \leq \rho\}} \frac{\int_0^1 |H_\tau(x)-H_\tau(y)|^{p-2}\dd\tau}{|x-y|^2} |a(x)-a(y)|^2\dx\dy  \\
    &\geq \frac{p-1}{2}
        \brac{\frac12-C(2-p)\log\brac{C\brac{2+\frac{\Lambda}{2-p}}} }
        \iint_{\S^1\times\S^1} \chi_{\{|x-y| \leq \rho\}}
        \frac{|a(x)-a(y)|^2}{|x-y|^{4-p}}\dx\dy
        \\
        & \geq
        \frac{p-1}{8}
        \iint_{\S^1\times\S^1} \chi_{\{|x-y| \leq \rho\}}
        \frac{|a(x)-a(y)|^2}{|x-y|^{4-p}}\dx\dy .
\end{split}
\]
Finally, if \underline{$p\in(p_-,2-\gamma)$}, then again by \Cref{la:p-less-2-lip-chord-term},
\[
\begin{split}
   &\frac{p-1}{2}
   \iint_{\S^1\times\S^1} \chi_{\{|x-y| \leq \rho\}} \frac{\int_0^1 |H_\tau(x)-H_\tau(y)|^{p-2}\dd\tau}{|x-y|^2} |a(x)-a(y)|^2\dx\dy  \\
    & \geq \frac{p-1}{8}
        \brac{1+C\frac{\Lambda}{\gamma}}^{-1}
        \iint_{\S^1\times\S^1} \chi_{\{|x-y| \leq \rho\}}
        \frac{|a(x)-a(y)|^2}{|x-y|^{4-p}}\dx\dy .
\end{split}
\]
Thus, \eqref{eq:localstableqrhoterm:goal} holds for all
$p\in[p_-,p_+]$.

Combining \eqref{eq:second_var_a} and \eqref{eq:localstableqrhoterm:goal}, we obtain
\begin{equation}\label{eq:plarger2:stablealsdvcjxopv1}
\begin{split}
        &\int_0^1
        \iint_{\S^1\times\S^1} \chi_{\{|x-y| \leq \rho\}}
        \frac{
        |H_\tau(x)-H_\tau(y)|^{p-2}
        \left[
        p-1-\frac p4|H_\tau(x)-H_\tau(y)|^2
        \right]
        }{|x-y|^2}
        |a(x)-a(y)|^2\dx\dy\dd\tau
        \\
        &\qquad\geq
        A_\Lambda
        \iint_{\S^1\times\S^1} \chi_{\{|x-y| \leq \rho\}}
        \frac{|a(x)-a(y)|^2}{|x-y|^{4-p}}\dx\dy .
\end{split}
\end{equation}
For the remaining integral we write
\[
\begin{split}
        &\int_0^1
        \iint_{\S^1\times\S^1} \chi_{\{|x-y| \geq \rho\}}
        \frac{
        |H_\tau(x)-H_\tau(y)|^{p-2}
        \left[
        p-1-\frac p4|H_\tau(x)-H_\tau(y)|^2
        \right]
        }{|x-y|^2}
        |a(x)-a(y)|^2\dx\dy\dd\tau
        \\
        =&
        \int_0^1
        \iint_{\S^1\times\S^1} \chi_{\{|x-y| \geq \rho\}}
        \frac{
        |x-y|^{p-2}
        \left[
        p-1-\frac p4|x-y|^2
        \right]
        }{|x-y|^2}
        |a(x)-a(y)|^2\dx\dy\dd\tau
        \\
        &+
        \int_0^1
        \iint_{\S^1\times\S^1} \chi_{\{|x-y| \geq \rho\}}
        \frac{
        (p-1)\brac{|H_\tau(x)-H_\tau(y)|^{p-2}-|x-y|^{p-2}}
        }{|x-y|^2}
        |a(x)-a(y)|^2\dx\dy\dd\tau
        \\
        &+
        \int_0^1
        \iint_{\S^1\times\S^1} \chi_{\{|x-y| \geq \rho\}}
        \frac{
        -\frac p4\brac{|H_\tau(x)-H_\tau(y)|^p-|x-y|^p}
        }{|x-y|^2}
        |a(x)-a(y)|^2\dx\dy\dd\tau .
\end{split}
\]
The first term is exactly
\[
        \iint_{\S^1\times\S^1} \chi_{\{|x-y| \geq \rho\}}
        \frac{
        |x-y|^{p-2}
        \left[
        p-1-\frac p4|x-y|^2
        \right]
        }{|x-y|^2}
        |a(x)-a(y)|^2\dx\dy .
\]
Moreover, for $\eps\ll\rho$, if $|x-y|\geq\rho$, then by \eqref{eq:distance_Htau}
\[
        |H_\tau(x)-H_\tau(y)|\ageq \rho .
\]
Hence, uniformly for $p\in[p_-,p_+]$,
\[
        \abs{|H_\tau(x)-H_\tau(y)|^{p-2}-|x-y|^{p-2}}
        \aleq
        \rho^{p-3}\|a\|_{L^\infty}
\]
and
\[
        \abs{|H_\tau(x)-H_\tau(y)|^p-|x-y|^p}
        \aleq
        \|a\|_{L^\infty}.
\]
Therefore, we obtain
\begin{equation}\label{eq:plarger2:stablealsdvcjxopv2}
\begin{split}
        &\int_0^1
        \iint_{\S^1\times\S^1} \chi_{\{|x-y| \geq \rho\}}
        \frac{
        |H_\tau(x)-H_\tau(y)|^{p-2}
        \left[
        p-1-\frac p4|H_\tau(x)-H_\tau(y)|^2
        \right]
        }{|x-y|^2}
        |a(x)-a(y)|^2\dx\dy\dd\tau
        \\
        \geq&
        \iint_{\S^1\times\S^1} \chi_{\{|x-y| \geq \rho\}}
        \frac{
        |x-y|^{p-2}
        \left[
        p-1-\frac p4|x-y|^2
        \right]
        }{|x-y|^2}
        |a(x)-a(y)|^2\dx\dy
        \\
        &-
        \tilde C
        \brac{\rho^{p_- -5}+\rho^{-2}}
        \|a\|_{L^\infty}
        [a]_{H^{\frac{3-p}{2}}}^2 .
\end{split}
\end{equation}

Combining \eqref{eq:criticalidentitylocalmin},
\eqref{eq:plarger2:stablealsdvcjxopv1}, and
\eqref{eq:plarger2:stablealsdvcjxopv2}, we find
\[
\begin{split}
0
\geq&
A_\Lambda
\iint_{\S^1\times\S^1} \chi_{\{|x-y| \leq \rho\}}
        \frac{|a(x)-a(y)|^2}{|x-y|^{4-p}}\dx\dy
\\
&+
\iint_{\S^1\times\S^1} \chi_{\{|x-y| \geq \rho\}}
        \frac{
        |x-y|^{p-2}
        \left[
        p-1-\frac p4|x-y|^2
        \right]
        }{|x-y|^2}
        |a(x)-a(y)|^2\dx\dy
\\
&-
\tilde C
        \brac{\rho^{p_- -5}+\rho^{-2}}
        \|a\|_{L^\infty}
        [a]_{H^{\frac{3-p}{2}}}^2 .
\end{split}
\]
Applying \eqref{eq:Aweridcoercivity}, multiplied by $2$, gives
\[
\begin{split}
0
\geq
\brac{
2\bar c
-
\tilde C
        \brac{\rho^{p_- -5}+\rho^{-2}}
        \|a\|_{L^\infty}
}
[a]_{H^{\frac{3-p}{2}}}^2 .
\end{split}
\]
Now choose $\eps=\eps(s_-,s_+,\Lambda)>0$ small enough so that
\[
        \|a\|_{L^\infty}
        \leq
        C\|u-\id\|_{L^\infty}
        \leq
        C\eps
\]
and
\[
        \tilde C
        \brac{\rho^{p_- -5}+\rho^{-2}}
        \|a\|_{L^\infty}
        \leq
        \bar c .
\]
Then
\[
        0\geq
        \bar c
        [a]_{H^{\frac{3-p}{2}}}^2 .
\]
Hence $[a]_{H^{\frac{3-p}{2}}}=0$. Thus $a$ is constant. Since the zeroth Fourier mode of $a$ vanishes by
\eqref{eq:aftmodeszero}, we have $a\equiv0$.
Therefore $\widetilde u=\id$.
Since $\widetilde u=e^{-i\lambda}u\circ m_b$, this implies that $u$ differs from
$\id$ by a rotation and a Möbius transform. Hence $u\in\mathcal M$.
\end{proof}

\begin{proof}[Proof of \Cref{th:minisId}]
By \Cref{th:stable} for any $\eps >0$ there exists a $\delta > 0$ such
any degree minimizer $u_s$ of $E_s$, $s \in [\frac{1}{2}-\delta,\frac{1}{2}+\delta]$ satisfies (up to rotation and conformal transforms) for some
\[
		\|u_s - \id\|_{C^{\gamma_0}} + \|u_s-\id\|_{W^{s_0,p_0}(\S^1,\R^2)} \leq \eps
		\]
If $p<2$, we apply \Cref{th:Lipregularity} to get $[u_s]_{\lip} \leq C(2-p)^{-\frac{1}{2}}$ for a uniform constant $C$.

Thus \Cref{th:localminisidv2} is applicable and we conclude $u_s \in \mathcal{M}$.
\end{proof}

\section{Consequences for lateral stability}\label{s:localstable}
We prove the following.
\begin{theorem}[Local stability]\label{th:rollingstablev2}
Let $\frac{1}{3}<s_{-} < s_+ <1$ and $\Lambda > 0$, $\alpha > 0$

There exists a $\delta_0 = \delta_0(s_-,s_+,\Lambda,\alpha)$ and a constant $c = c(s_-,s_+,\Lambda,\alpha)>0$ such that the following holds for any $s \in [s_-,s_+]$. Set $p = \frac{1}{s}$.
For any $\eps > 0$ and $t \in (0,s]$ there exists $\delta > 0$ such that the following holds.
Let $u\in W^{\frac{3-p}{2},2}\cap W^{s,\frac{1}{s}}\cap C^\alpha(\S^1,\S^1)$ be such that
\[
 [u]_{W^{s,\frac{1}{s}}}+ [u]_{C^\alpha} \leq \Lambda.
\]
If $p < 2$ we assume in addition that
\[
 [u]_{\lip} \leq \frac{\Lambda}{|p-2|}.
\]
If for some $t < s$, we have $[u-\id]_{W^{t,\frac{1}{t}}} \leq \delta$, then there exists $\phi =e^{\i \lambda} m_b\in \mathcal{M}$ with $|b| \leq \eps$, $\lambda \in \R$ and
\[
 E_s(u\circ\phi)-E_s(\id) \geq c\, \inf_{\psi\in \mathcal{M}}[u\circ \psi-\id]_{H^{\frac{3-p}{2}}}^2.
\]
\end{theorem}

The philosophy of our stability result is simple: If $u_1$ and $\id$ are close to each other, then we can find a homotopy $H_\tau\colon \S^1 \to \S^1$, $H_0 = \id$, $H_1 = u_1$, and we can hope for $\partial_\tau H_\tau \approx u_1-\id$, in a reasonable sense.

Formally by Taylor, $E_s(u)-E_s(\id) = \delta E_s(\id)[\partial_\tau H_\tau] +\frac{1}{2} \delta^2 E_s(H_\tau)[\partial_\tau H_\tau,\partial_\tau H_\tau] + o(|\partial_{\tau} H_{\tau}|^2)$. Since $\delta E_s(\id)[\partial_\tau H_\tau] = 0$ (since the $\id$ is critical) and $\delta^2 E_s(H_\tau)$ is (almost) positive definite, we expect
\[
 E_s(u)-E_s(\id) \geq c |\partial_\tau H_\tau|^2.
\]
The difficulty lies in obtaining all of this in the right function space, which happens to be $H^{\frac{3-p}{2}}$ and which makes the comparison of $\delta^2 E_s(H_\tau)$ and $\delta^2 E_s(\id)$ challenging. Also, ``positivity'' of $\delta^2 E_s(\id)$ holds in all but three Fourier modes (corresponding to the directions of conformal invariance).

\begin{proof}[Proof of \Cref{th:rollingstablev2}]
By Gagliardo--Nirenberg, using also \Cref{la:poincareonS1}, for some $\sigma_1 > 0$
\[
 \|u-\id\|_{C^{\frac{\alpha}{2}} (\S^1)} \leq C\brac{\|u\|_{C^\alpha}} [u-\id]_{W^{s,\frac{1}{s}}}^{\sigma_1}.
\]
Thus, there is no harm in assuming w.l.o.g. (otherwise replace $\alpha$ by $\frac{\alpha}{2}$)
\[
 \|u-\id\|_{C^{\alpha}} \leq \delta_0.
\]
By \Cref{la:local-normalization-phase}, we can replace
\[
        \widetilde u\coloneqq e^{-i\lambda}u\circ m_b,
\]
for some $\lambda\in\R$ and $|b|<1$, and by \Cref{la:lipschitz-phase} we have a unique phase $a\colon \S^1 \to [-\frac{\pi}{4},\frac{\pi}{4}]$, so that
\[
        \widetilde u(x)
        =
        \cos(a(x))x+\sin(a(x))x^\perp,
\]
with zero Fourier modes as in \eqref{eq:aftmodeszero}.
Moreover, after decreasing $\delta_0$ if necessary,
\[
        \|a\|_{C^\alpha}\aleq \|u-\id\|_{C^\alpha}\leq \delta_0.
\]
If $p<2$, then also
\[
        [a]_{\lip}\aleq 1+[u]_{\lip}\aleq 1+\frac{\Lambda}{|p-2|}.
\]
By \Cref{la:mobius-invariance}, we have $E_s(\widetilde u)=E_s(u)$. Set as in \Cref{la:homotopy}
\[
        H_\tau(x)
        \coloneqq 
        \cos(\tau a(x))x+\sin(\tau a(x))x^\perp,
        \qquad
        0\leq \tau\leq 1.
\]
Then it holds $H_0=\id$, $H_1=\widetilde u$ and
\[
        \partial_\tau H_\tau(x)=a(x)H_\tau^\perp(x),
        \qquad
        \partial_{\tau\tau}H_\tau(x)=-a(x)^2H_\tau(x).
\]
Since $\id$ is critical,
\begin{equation}\label{eq:ddtauEshtaueq0}
        \left.
        \frac{\dd}{\dd\tau}
        \right|_{\tau=0}
        \frac1p E_s(H_\tau)=0.
\end{equation}
Thus, by applying the fundamental theorem of calculus twice,
\[
\begin{split}
        \frac1p\left(E_s(\widetilde u)-E_s(\id)\right)
        &=
        \int_0^1
        \frac{\dd}{\dd t}
        \frac1pE_s(H_t)\dd t \overset{\eqref{eq:ddtauEshtaueq0}}{=}
        \int_0^1\int_0^t
        \frac{\dd^2}{\dd\tau^2}
        \frac1pE_s(H_\tau)\dd\tau\dd t .
\end{split}
\]
By \eqref{eq:identity1}, and
\eqref{eq:identity2},
\[
\begin{split}
        \frac{\dd^2}{\dd\tau^2}
        \frac1pE_s(H_\tau)
        &=
        \iint_{\S^1\times\S^1}
        \frac{
        |H_\tau(x)-H_\tau(y)|^{p-2}
        \left[
        p-1-\frac p4|H_\tau(x)-H_\tau(y)|^2
        \right]
        }{|x-y|^2}
        |a(x)-a(y)|^2\dx\dy .
\end{split}
\]
Hence, since $E_s(\widetilde{u}) = E_s(u)$,
\[
\begin{split}
        &\frac1p\left(E_s(u)-E_s(\id)\right)\\
        &=
        \int_0^1(1-\tau)
        \iint_{\S^1\times\S^1}
        \frac{
        |H_\tau(x)-H_\tau(y)|^{p-2}
        \left[
        p-1-\frac p4|H_\tau(x)-H_\tau(y)|^2
        \right]
        }{|x-y|^2} |a(x)-a(y)|^2\dx\dy\dd\tau\\
        &=
        \int_0^1 (1-\tau)
        \iint_{\S^1\times\S^1}
        \frac{
        |H_\tau(x)-H_\tau(y)|^{p-2}
        \left[
        p-1-\frac p4|H_\tau(x)-H_\tau(y)|^2
        \right]
        }{|x-y|^2} |a(x)-a(y)|^2\dx\dy\dd\tau.
\end{split}
\]
We now use the positivity of the second variation at $\id$, in the form of \Cref{co:la:truncated-coercivity-one-kernel}.
For an $A$ yet to choose, we fix $\rho_0$ from \Cref{co:la:truncated-coercivity-one-kernel} and obtain a constant $c>0$ such that for any $\rho < \rho_0$
\begin{equation}\label{eq:Aweridcoercivityv2}
\begin{split}
&A
\iint_{\{|x-y|\leq\rho\}}
\frac{|a(x)-a(y)|^2}{|x-y|^{4-p}}\dx\dy
\\
&+
\frac12
\iint_{\{|x-y|\geq\rho\}}
\frac{
|x-y|^{p-2}
\left[
p-1-\frac p4|x-y|^2
\right]
}{|x-y|^2}
|a(x)-a(y)|^2\dx\dy
 \geq
c
[a]_{H^{\frac{3-p}{2}}(\S^1)}^2 .
\end{split}
\end{equation}

We split the integral and first consider
\[
\begin{split}
        \int_0^1(1-\tau)
        \iint_{\S^1\times\S^1} \chi_{\{|x-y| \leq \rho\}}
        \frac{
        |H_\tau(x)-H_\tau(y)|^{p-2}
        \left[
        p-1-\frac p4|H_\tau(x)-H_\tau(y)|^2
        \right]
        }{|x-y|^2} |a(x)-a(y)|^2\dx\dy\dd\tau.
\end{split}
\]
Observe that if $|x-y| \leq \rho$
\[
 |H_\tau(x)-H_\tau(y)| \aleq |x-y| + \|a\|_{L^\infty} \leq \rho + \delta.
\]
In particular, we may assume that ($\rho$ and $\delta$ only depend on $p$ in a uniform way since $p>p_-$) if $|x-y| \leq \rho$, then
\[
 p-1-\frac p4|H_\tau(x)-H_\tau(y)|^2 \geq \frac{p-1}{2}.
\]
We obtain 
\[
\begin{split}
        &\int_0^1(1-\tau)
        \iint_{\S^1\times\S^1} \chi_{\{|x-y| \leq \rho\}}
        \frac{
        |H_\tau(x)-H_\tau(y)|^{p-2}
        \left[
        p-1-\frac p4|H_\tau(x)-H_\tau(y)|^2
        \right]
        }{|x-y|^2} |a(x)-a(y)|^2\dx\dy\dd\tau\dd t\\
        & \geq \frac{p-1}{2} \iint_{\S^1\times\S^1} \chi_{\{|x-y| \leq \rho\}}
        \frac{
        \int_0^1(1-\tau)|H_\tau(x)-H_\tau(y)|^{p-2}\dd\tau
        }{|x-y|^2} |a(x)-a(y)|^2\dx\dy.
\end{split}
\]
Fix $\gamma > 0$. We are going to show the following:
\begin{equation}\label{eq:localstableqrhoterm:goalv2}
\begin{split}
   &\frac{p-1}{2} \iint_{\S^1\times\S^1} \chi_{\{|x-y| \leq \rho\}}
        \frac{
        \int_0^1(1-\tau)|H_\tau(x)-H_\tau(y)|^{p-2}\dd\tau
        }{|x-y|^2} |a(x)-a(y)|^2\dx\dy \\
        & \geq A_\Lambda \iint_{\S^1\times\S^1} \chi_{\{|x-y| \leq \rho\}}
        \frac{|a(x)-a(y)|^2}{|x-y|^{4-p}}\dx\dy,
        \end{split}
\end{equation}
where (assuming $\rho$ is small enough, but still uniform)
\[
 A_\Lambda\coloneqq  \begin{cases}
    \frac{p-1}{2} c \quad &\text{if $p>2$}\\
     \frac{p-1}{2} \quad &\text{if $p=2$}\\
     \frac{p-1}{2} \brac{\frac{1}{2}-C\gamma \log(2+\frac{1}{\gamma}\Lambda)} \quad &\text{if $p\in [2-\gamma,2)$}\\
     \frac{p-1}{2} \frac{1}{4} (1+C\frac{1}{\gamma} \Lambda)^{-1} &\text{if $p\in (p_-,2-\gamma)$},
    \end{cases}
\]
where $c,C$ is a uniform constant. We see that depending on $\Lambda$ there exists a $\gamma$ that makes $A_\Lambda > 0$.

Indeed, \underline{if $p=2$} there is nothing to show.

If \underline{$p > 2$} we can apply \eqref{eq:averaged-chord-powerv2} (taking $\rho$ even smaller, but still uniform) and have

\[\begin{split}
        &\frac{p-1}{2} \iint_{\S^1\times\S^1} \chi_{\{|x-y| \leq \rho\}}
        \frac{
        \int_0^1(1-\tau)|H_\tau(x)-H_\tau(y)|^{p-2}\dd\tau
        }{|x-y|^2} |a(x)-a(y)|^2\dx\dy\\
        & \geq \frac{p-1}{2} c \iint_{\S^1\times\S^1} \chi_{\{|x-y| \leq \rho\}}
        \frac{|a(x)-a(y)|^2}{|x-y|^{4-p}}\dx\dy\dd\tau.
\end{split}
\]
So we can choose $A \coloneqq \frac{p-1}{2} c$ in \eqref{eq:localstableqrhoterm:goalv2}.

\underline{If $p \in [2-\gamma,2)$}, we apply \Cref{la:p-less-2-lip-chord-term} to obtain 
\[\begin{split}
        &\frac{p-1}{2} \iint_{\S^1\times\S^1} \chi_{\{|x-y| \leq \rho\}}
        \frac{
        \int_0^1(1-\tau)|H_\tau(x)-H_\tau(y)|^{p-2}\dd\tau
        }{|x-y|^2} |a(x)-a(y)|^2\dx\dy\\
        & \geq \frac{p-1}{2} \brac{\frac{1}{2}-C{\gamma} \log(2+\frac{1}{|p-2|}\Lambda)} \iint_{\S^1\times\S^1} \chi_{\{|x-y| \leq \rho\}}
        \frac{|a(x)-a(y)|^2}{|x-y|^{4-p}}\dx\dy\dd\tau.
\end{split}
\]
If \underline{if $p \in [p_-,p-\gamma]$}, we apply \Cref{la:p-less-2-lip-chord-term}
\[\begin{split}
        &\frac{p-1}{2} \iint_{\S^1\times\S^1} \chi_{\{|x-y| \leq \rho\}}
        \frac{
        \int_0^1(1-\tau)|H_\tau(x)-H_\tau(y)|^{p-2}\dd\tau
        }{|x-y|^2} |a(x)-a(y)|^2\dx\dy\\
        & \geq \frac{p-1}{8} (1+C\frac{1}{|p-2|}\Lambda)^{-|p-2|} \iint_{\S^1\times\S^1} \chi_{\{|x-y| \leq \rho\}} \frac{|a(x)-a(y)|^2}{|x-y|^{4-p}}\dx\dy.
\end{split}
\]
This establishes \eqref{eq:localstableqrhoterm:goalv2} for all $p \in (p_-,p_+)$, and thus we have obtained
\begin{equation}\label{eq:plarger2:stablealsdvcjxopv1v2}
\begin{split}
        & \int_0^1(1-\tau)
        \iint_{\S^1\times\S^1} \chi_{\{|x-y| \leq \rho\}}
        \frac{
        |H_\tau(x)-H_\tau(y)|^{p-2}
        \left[
        p-1-\frac p4|H_\tau(x)-H_\tau(y)|^2
        \right]
        }{|x-y|^2} |a(x)-a(y)|^2\dx\dy\dd\tau\\
	& \geq A_\Lambda \iint_{\S^1\times\S^1} \chi_{\{|x-y| \leq \rho\}}
        \frac{|a(x)-a(y)|^2}{|x-y|^{4-p}}\dx\dy.
        \end{split}
\end{equation}

For the remaining integral we write
\[
\begin{split}
        &\int_0^1(1-\tau)
        \iint_{\S^1\times\S^1} \chi_{\{|x-y| \geq \rho\}}
        \frac{|H_\tau(x)-H_\tau(y)|^{p-2}
        \left[ p-1-\frac p4|H_\tau(x)-H_\tau(y)|^2\right]
        }{|x-y|^2} |a(x)-a(y)|^2\dx\dy\dd\tau\\
        =&\int_0^1(1-\tau)
        \iint_{\S^1\times\S^1} \chi_{\{|x-y| \geq \rho\}}
        \frac{
        |x-y|^{p-2}
        \left[
        p-1-\frac p4|x-y|^2
        \right]
        }{|x-y|^2} |a(x)-a(y)|^2\dx\dy\dd\tau\\
        &+\int_0^1(1-\tau)
        \iint_{\S^1\times\S^1} \chi_{\{|x-y| \geq \rho\}}
        \frac{
        (p-1)\brac{|H_\tau(x)-H_\tau(y)|^{p-2} - |x-y|^{p-2}}
        }{|x-y|^2} |a(x)-a(y)|^2\dx\dy\dd\tau\\
        &+\int_0^1(1-\tau)
        \iint_{\S^1\times\S^1} \chi_{\{|x-y| \geq \rho\}}
        \frac{
        -\frac p4 \brac{|H_\tau(x)-H_\tau(y)|^p - |x-y|^p}
        }{|x-y|^2} |a(x)-a(y)|^2\dx\dy\dd\tau\\
\geq&\frac{1}{2}
        \iint_{\S^1\times\S^1} \chi_{\{|x-y| \geq \rho\}}
        \frac{
        |x-y|^{p-2}
        \left[
        p-1-\frac p4|x-y|^2
        \right]
        }{|x-y|^2} |a(x)-a(y)|^2\dx\dy\\
        &- C\rho^{-2} \max_{\tau \in [0,1]} \iint_{\S^1\times\S^1} \chi_{\{|x-y| \geq \rho\}}
        \abs{|H_\tau(x)-H_\tau(y)|^{p-2}  - |x-y|^{p-2}} |a(x)-a(y)|^2\dx\dy\\
        &- C\rho^{-2} \max_{\tau \in [0,1]} \iint_{\S^1\times\S^1} \chi_{\{|x-y| \geq \rho\}}
        \abs{|H_\tau(x)-H_\tau(y)|^{p}  - |x-y|^{p}} |a(x)-a(y)|^2\dx\dy.
\end{split}
\]
Since $1 \leq p <p_+ <3$, and for $\|a\|_{L^\infty} \leq \delta \ll \rho$ we have $|x-y| \geq \rho$ implies $|H_\tau(x)-H_\tau(y)| \ageq \rho$, so
\[
 \abs{|H_\tau(x)-H_\tau(y)|^{p-2}  - |x-y|^{p-2}} \aleq \rho^{p-3} \|a\|_{L^\infty}
\]
and
\[
 \abs{|H_\tau(x)-H_\tau(y)|^{p}  - |x-y|^{p}} \aleq \|a\|_{L^\infty}
\]
Thus, we obtain 
\begin{equation}\label{eq:plarger2:stablealsdvcjxopv2v2}
\begin{split}
        &\int_0^1(1-\tau)
        \iint_{\S^1\times\S^1} \chi_{\{|x-y| \geq \rho\}}
        \frac{
        |H_\tau(x)-H_\tau(y)|^{p-2}
        \left[
        p-1-\frac p4|H_\tau(x)-H_\tau(y)|^2
        \right]
        }{|x-y|^2} |a(x)-a(y)|^2\dx\dy\dd\tau\\
& \geq \frac{1}{2}
        \iint_{\S^1\times\S^1} \chi_{\{|x-y| \geq \rho\}}
        \frac{
        |x-y|^{p-2}
        \left[
        p-1-\frac p4|x-y|^2
        \right]
        }{|x-y|^2} |a(x)-a(y)|^2\dx\dy\\
        &\qquad - C\rho^{-2} \|a\|_{L^\infty} \int_{\S^1} \int_{\S^1} |a(x)-a(y)|^2\dx\dy\\
& \geq \frac{1}{2}
        \iint_{\S^1\times\S^1} \chi_{\{|x-y| \geq \rho\}}
        \frac{
        |x-y|^{p-2}
        \left[
        p-1-\frac p4|x-y|^2
        \right]
        }{|x-y|^2} |a(x)-a(y)|^2\dx\dy\\
        &\qquad - \tilde{C}\rho^{-2} \|a\|_{L^\infty} [a]_{H^{\frac{3-p}{2}}}^2.
\end{split}
\end{equation}

Combining \eqref{eq:plarger2:stablealsdvcjxopv1v2} and \eqref{eq:plarger2:stablealsdvcjxopv2v2} we find

\[
\begin{split}
&\iint_{\S^1\times\S^1} \chi_{\{|x-y| \geq \rho\}}
        \frac{
        (p-1)\brac{|H_\tau(x)-H_\tau(y)|^{p-2}  - |x-y|^{p-2}}
        }{|x-y|^2} |a(x)-a(y)|^2\dx\dy\dd\tau\\
    & \geq A_\Lambda \iint_{\S^1\times\S^1} \chi_{\{|x-y| \leq \rho\}}
        \frac{|a(x)-a(y)|^2}{|x-y|^{4-p}}\dx\dy\dd\tau\\
        &\qquad +\frac{1}{2}
        \iint_{\S^1\times\S^1} \chi_{\{|x-y| \geq \rho\}}
        \frac{
        |x-y|^{p-2}
        \left[
        p-1-\frac p4|x-y|^2
        \right]
        }{|x-y|^2} |a(x)-a(y)|^2\dx\dy\\
        &\qquad - \tilde{C}\rho^{-2} \|a\|_{L^\infty} [a]_{H^{\frac{3-p}{2}}}^2.
\end{split}
        \]
Applying \eqref{eq:Aweridcoercivityv2} we get
\[
\begin{split}
&\iint_{\S^1\times\S^1} \chi_{\{|x-y| \geq \rho\}}
        \frac{
        (p-1)\brac{|H_\tau(x)-H_\tau(y)|^{p-2}  - |x-y|^{p-2}}
        }{|x-y|^2} |a(x)-a(y)|^2\dx\dy\dd\tau\\
    & \geq \brac{\bar{c}- \tilde{C}\rho^{-2} \|a\|_{L^\infty}} [a]_{H^{\frac{3-p}{2}}}^2.
\end{split}
        \]
So if $\|a\|_{L^\infty}$ is even smaller, we can conclude
\[
\begin{split}
&\int_0^1 (1-\tau)
        \iint_{\S^1\times\S^1}
        \frac{
        |H_\tau(x)-H_\tau(y)|^{p-2}
        \left[
        p-1-\frac p4|H_\tau(x)-H_\tau(y)|^2
        \right]
        }{|x-y|^2} |a(x)-a(y)|^2\dx\dy\dd\tau  \\
        & \geq \frac{\bar{c}}{2} [a]_{H^{\frac{3-p}{2}}}^2.
\end{split}
        \]
That is, for any $p\in(p_-,p_+)$ we have shown with a uniform constant
\[
        E_s(u)-E_s(\id)\geq c [a]_{H^{\frac{3-p}{2}}}^2
\]
We conclude with \Cref{la:acontrol}: for some $\lambda \in \R$
\[
        E_s(u)-E_s(\id)\geq c [e^{\i \lambda}\widetilde{u}-\id]_{H^{\frac{3-p}{2}}}^2
\]
\end{proof}

	\bibliographystyle{abbrv}%
	\bibliography{bib}%
\end{document}